\documentclass[11pt,reqno,a4paper]{amsart} 

\usepackage[margin= 2.4 cm, bmargin=2cm,tmargin=3cm]{geometry}
\usepackage[dvipsnames]{xcolor}
\usepackage[hidelinks]{hyperref}
\hypersetup{colorlinks = true,
linkcolor = red, anchorcolor =red,
citecolor = ForestGreen,
filecolor = red,urlcolor = red,
            pdfauthor=author}

\usepackage[capitalize]{cleveref}
\usepackage{mparhack}

\usepackage{mathtools}

\usepackage{amsmath,amssymb,bm}
\usepackage[alphabetic]{amsrefs}
\usepackage{indentfirst}
\usepackage{graphicx}
\usepackage{amsthm}

\usepackage{enumitem}
\makeatletter
\def\namedlabel#1#2{\begingroup
	#2%
	\def\@currentlabel{#2}%
	\phantomsection\label{#1}\endgroup
}
\makeatother

\makeatletter 
\@mparswitchfalse%
\makeatother
\normalmarginpar 

\usepackage{todonotes}

\newtheorem{theorem}{Theorem}[section]
\newtheorem{lemma}[theorem]{Lemma}
\newtheorem{proposition}[theorem]{Proposition}
\newtheorem{corollary}[theorem]{Corollary}

\numberwithin{equation}{section}
\newtheorem{example}{Example}[section]
\newtheorem{condition}{Condition}
\newtheorem{assumption}{Assumption}

\newtheorem{conditionalt}{Condition}

\newenvironment{conditionnp}[2]{%
  \begingroup
  \renewcommand{\theconditionalt}{#1}
  \hypersetup{linkcolor=black}
  \begin{conditionalt}[\hypersetup{linkcolor=red}#2]
}{%
  \end{conditionalt}
  \hypersetup{linkcolor=red}
  \endgroup
}

\theoremstyle{remark}
\newtheorem{remark}{Remark}

\theoremstyle{definition}
\newtheorem{definition}[theorem]{Definition}

\DeclareMathOperator{\tr}{tr}

\DeclareMathOperator{\ran}{Ran}

\DeclareMathOperator{\dom}{Dom}

\newcommand{\ud}{\,\mathrm{d}}
\newcommand{\rd}{\mathrm{d}}

\newcommand{\R}{\mathbb{R}}
\newcommand{\C}{\mathbb{C}}

\newcommand{\Or}{\mathcal{O}}

\newcommand{\dd}{\cdot}

\newcommand{\na}{\nabla}

\newcommand{\maf}[1]{\mathfrak{#1}}

\IfFileExists{mathabx.sty}%
  {\DeclareFontFamily{U}{mathx}{\hyphenchar\font45}%
   \DeclareFontShape{U}{mathx}{m}{n}{<->mathx10}{}%
   \DeclareSymbolFont{mathx}{U}{mathx}{m}{n}%
   \DeclareFontSubstitution{U}{mathx}{m}{n}%
   \DeclareMathAccent{\widebar}{0}{mathx}{"73}%
}{%
  \PackageWarning{mathabx}{%
    Package mathabx not available, therefore\MessageBreak substituting
    widebar with overline\MessageBreak }%
  \newcommand{\widebar}[1]{\overline{#1}}%
}

\newcommand{\mc}[1]{\mathcal{#1}}

\newcommand{\mf}[1]{\mathsf{#1}}

\newcommand{\veps}{\varepsilon}

\newcommand{\lad}{\lambda}
\newcommand{\si}{\sigma}

\newcommand{\p}{\partial}
 
\def \ep {\varepsilon}
\def \ww {\omega}
\def \l {\langle} 
\def \r {\rangle}
\def \d {\delta}
\def \vp {\varphi}

\def \mhh {\mc{H}_O}

\def \bh {\mc{B}(\maf{H})}
\def \dh {\mc{D}(\maf{H})}
\def \dhh {\mc{D}_+(\maf{H})}

\def \ka {\kappa}

\def \mi {{\bf 1}}

\def \ld {\mc{L}_S}
\def \lh {\mc{L}_A}
\def \lo {\mc{L}_O}
\def \mh {\mc{H}}
\def \fl {\mf{L}}
\def \fll {\mf{L}_O}
\def \ls {\mf{L}_S}
\def \la {\mf{L}_A}
\def \es {E_S}
\def \hl {\hat{L}}
\def \hm {\hat{\mu}}
\def \hs {\hat{\mc{S}}}

\newcommand{\norm}[1]{\lVert#1\rVert}

\newcommand{\bra}[1]{\langle#1\rvert}

\newcommand{\ket}[1]{\lvert#1\rangle}

\renewcommand{\Re}{\mathfrak{Re}}

\def \q {\quad}
\def \w {\widetilde}
\def \mm {\left[\begin{matrix}}
\def \nn {\end{matrix}\right]}
\def \id {\mathrm{id}}

\title{Speeding up quantum Markov processes through lifting}

\author{Bowen Li} %
\address[B. Li]{Department of Mathematics,  City University of Hong Kong, Kowloon Tong, Hong Kong SAR}
\email{bowen.li@cityu.edu.hk}

\author{Jianfeng Lu}
\address[J. Lu]{Departments of Mathematics, Physics, and Chemistry, Duke University, Durham, NC 27708.}
\email{jianfeng@math.duke.edu}

\begin{document}

\begin{abstract}
We generalize the concept of \emph{non-reversible lifts} for reversible diffusion processes initiated by Eberle and L\"{o}rler (2024) to quantum Markov dynamics. The lifting operation, which naturally results in hypocoercive processes, can be formally interpreted as, though not restricted to, the \emph{reverse} of the overdamped limit. 
We prove that the $L^2$ convergence rate of the lifted process is bounded above by the square root of the spectral gap of its overdamped dynamics, indicating that the lifting approach can at most achieve a transition from diffusive to ballistic mixing speeds. Further,
using the variational hypocoercivity framework based on space-time Poincar\'{e} inequalities, we derive a lower bound for the convergence rate of the lifted dynamics. These findings not only offer quantitative convergence guarantees for hypocoercive quantum Markov processes but also characterize the potential and limitations of accelerating the convergence through lifting. In addition, we develop an abstract lifting framework in the Hilbert space setting applicable to any symmetric contraction $C_0$-semigroup, thereby unifying the treatment of classical and quantum dynamics. As applications, we construct optimal lifts for various detailed balanced classical and quantum processes, including the symmetric random walk on a chain, the depolarizing semigroup, Schur multipliers, and quantum Markov semigroups on group von Neumann algebras. 
\end{abstract}

\maketitle

\section{Introduction}
Quantum Markov dynamics is the Markovian approximation of the time evolution of open quantum systems interacting with an external environment in the weak coupling limit \cite{breuer2002theory}. Mathematically modeled by the Lindblad master equations, i.e., \emph{quantum Markov semigroups} (QMS), these processes capture fundamental dissipative phenomena such as decoherence and thermalization \cite{schlosshauer2019quantum}, playing a crucial role in quantum statistical mechanics, quantum information theory, and quantum computing \cites{verstraete2009quantum,kastoryano2011dissipative,zhang2024driven}. Thanks to the rapid advancements of quantum computers and quantum algorithms, open quantum dynamics has found fruitful applications in various fields, including quantum Gibbs state preparation \cites{chen2023quantum,chen2023efficient,ding2024efficient}, ground state preparation \cites{ding2023single,zhan2025rapid}, nonconvex optimization problems \cite{chen2025quantum}, and quantum computational advantage \cites{chen2024local,bergamaschi2024quantum}.

This work is motivated by recent advancements in Monte Carlo-type quantum Gibbs samplers based on quantum Markov dynamics, in analogy with classical Glauber and Metropolis dynamics \cites{chen2023quantum,chen2023efficient,ding2024efficient}. Specifically, consider a quantum Gibbs state $\sigma_\beta \propto \exp(-\beta H)$ associated with a many-body Hamiltonian $H$, where $\beta$ represents the inverse temperature. The objective is to design efficiently (quantumly) simulatable dynamics of the form:
\[
\partial_t \rho_t = \mathcal{L}^\dag \rho_t,
\]
such that $\rho_t \to \sigma_\beta$ as $t \to \infty$, where $\mathcal{L}^\dag$ denotes the generator in the Schrödinger picture. To ensure feasibility in quantum implementations, these quantum Gibbs samplers typically employ local or quasi-local jumps \cites{rouze2024efficient,rouze2024optimal}, while the invariance of $\sigma_\beta$ under the evolution is generally guaranteed by the quantum detailed balance condition \cite{temme2010chi}. 
A key mathematical challenge lies in quantifying the rate at which the designed quantum Markov dynamics converge to their target invariant states, a measure commonly referred to as the mixing time. In this context, the rapid mixing property is both desirable and essential, as it ensures the efficiency and scalability of the proposed quantum Gibbs samplers, particularly for large-scale quantum many-body systems \cites{kochanowski2024rapid,rouze2024optimal,zhan2025rapid}.

For traditional classical Markov Chain Monte Carlo (MCMC) methods, two main obstacles hinder rapid or fast mixing: the energy barrier and the entropic barrier. Identifying these barriers (or bottlenecks) would provide \emph{lower bounds} on the mixing time \cite{levin2017markov}.
The energy barrier refers to regions in the energy landscape characterized by high peaks: a high-energy region separates two basins of local energy minima. Transitioning from one basin to the other one requires traversing a region of (often exponentially) low probability, making such transitions rare and inherently difficult.
The canonical examples include the Ising model \cite{thomas1989bound} and the hard-core model \cite{randall2006slow} on the lattices. On the other hand, the entropic barrier arises when a state basin contains a large number of comparably significant states, resulting in high entropy. In such cases, a sampling algorithm may require an enormous number of steps to effectively explore the vast configuration space, thereby severely impeding rapid convergence \cites{hayes2005general,ding2011mixing}. The physical intuition underlying energy and entropic barriers naturally applies to thermalization processes in open quantum dynamics. In particular, the energy barrier is fundamental for designing stable quantum memories \cite{hong2025quantum}. Recent studies have also generalized these classical bottleneck tools to the quantum setting, establishing unconditional lower bounds on the mixing times of local (or quasi-local) quantum Gibbs samplers \cites{gamarnik2024slow,rakovszky2024bottlenecks}.

Given these challenges, there has been considerable interest in developing methods to accelerate the convergence of MCMC sampling. A key observation is that detailed balance, while sufficient for ensuring a target stationary distribution, is not necessary. Strong theoretical results and numerical evidence demonstrate that \emph{nonreversible Markov chains}, which break detailed balance, often exhibit faster mixing compared to their reversible counterparts; see, for example, \cites{hwang1993accelerating,hwang2005accelerating,lelievre2013optimal,chen2013accelerating,michel2014generalized,guillin2016optimal,bierkens2019zig}. 
One fundamental approach to constructing nonreversible Markov processes is through \emph{lifting}, where the state space is augmented with auxiliary variables, such as momentum or direction. 
The resulting processes, known as \emph{lifted Markov chains}, can be potentially useful for overcoming entropic barriers (though they may not be as effective for escaping deep local energy minima). In certain cases, they have been shown to achieve substantial improvements in convergence rates, with mixing times reduced by up to a square root \cites{turitsyn2011irreversible,    vucelja2016lifting}.

For fintie Markov chains, \cite{chen1999lifting} considered the lift of a symmetric random walk on $\mathbb{Z}\backslash (n \mathbb{Z})$ to the augmented state space $\mathbb{Z}\backslash (n \mathbb{Z}) \times \{+1, - 1\}$ and showed that the mixing time can be reduced from $\Theta(n^2)$ to $\Theta(n)$. Furthermore, \cite{chen1999lifting} established that, in general, this square-root reduction, corresponding to a diffusive-to-ballistic speed-up, is best possible, and that if the lifted chain is reversible, then the improvement is at most a constant factor. 
Another notable example is from \cite{turitsyn2011irreversible} which introduced the non-reversible lift of Metropolis-Hastings and applied it to the Curie-Weiss model. It was numerically observed that lifted Metropolis-Hastings (LMH) can reduce the mixing time from $\mathcal{O}(n^{1.43})$ to $\mathcal{O}(n^{0.85})$ with $n$ being the number of spins. The theoretical justification for this improvement was later provided in \cite{10.1214/16-AAP1217} by identifying the diffusion limits of LMH and MH. 

In the case of continuous state spaces, canonical examples of sampling algorithms incorporating lifting ideas are piecewise deterministic Markov processes (PDMPs), which evolve according to deterministic motion interspersed with random discrete velocity jumps  \cites{davis1984piecewise,michel2014generalized,bierkens2019zig}. While PDMPs have demonstrated their efficiency in practical applications, obtaining quantitative estimates of their convergence rates has long been a significant challenge. These processes usually admit degenerate diffusions and their analysis are closely related to the hypocoercivity theory for kinetic PDEs \cites{villani2009hypocoercivity}. Refining the variational hypocoercivity analysis with a space-time Poincaré inequality pioneered by Albritton et al. \cite{albritton2019variational}, the recent work \cite{cao2023explicit} obtained convergence rates for underdamped Langevin dynamics with the sharp estimate and established a quadratic speed-up compared to overdamped Langevin dynamics when the potential is convex. These arguments were later extended to certain PDMPs, yielding quantitative convergence rates with explicit dimension dependence \cite{lu2022explicit}. 
Motivated by the concept of lifting proposed by Chen et al. \cite{chen1999lifting} in the discrete setting, Eberle et al. \cite{eberle2024non} introduced the second-order 
non-reversible lifts for continuous reversible diffusion processes. They derived the upper bounds on the convergence rates of lifted processes and showed that the relaxation time of a reversible diffusion process can at most be reduced by a square root through lifting, analogous to the result for the discrete Markov chains \cite{chen1999lifting}.  Additionally, the lifting framework simplifies the variational hypocoercivity approach used in \cites{albritton2019variational,cao2023explicit,lu2022explicit}, enabling the derivation of upper bounds on the convergence speeds for lifted processes via hypocoercivity. The analysis in \cite{eberle2024non} shows the optimality of the rate estimates obtained in \cites{cao2023explicit,lu2022explicit}. Such a second-order lifting framework was also extended to reversible diffusion processes on Riemannian manifolds \cite{eberle2024space} and PDMPs with non-trivial boundary conditions \cite{eberle2025convergence}.  

In light of the fact that lifting has proven effective in overcoming the diffusive behavior of classical sampling and achieving a ballistic speedup, a natural question arises: \emph{is there an analogous lifting framework for detailed-balanced quantum Markov processes, where the lifted quantum Markov semigroup on an extended state space can achieve a quadratic speedup?} In this work, we shall provide an affirmative answer to this question to bridge this gap. 

\subsection{Contribution and Outline}
The main contributions of this work are threefold: 

\smallskip 

\noindent (1) We develop a concept of second-order lifts for detailed balanced QMS, in analogy with \cite{eberle2024non}, with a particular emphasis on its connection to the overdamped limit in \cref{sec:liftingmatrix}. This framework enables us to establish both upper and lower bounds on the convergence rates of the lifted hypocoercive processes and to characterize the limits of accelerating QMS through lifting.

Since the lifted QMS considered in this work are necessarily hypocoercive (as will become clear), we begin in \cref{sec:hypoqms} by reviewing fundamentals of hypocoercive QMS in the non-primitive setting, where we assume the semigroup admits at least one full-rank invariant state (possibly non-unique). For QMS with detailed balance, the inverse spectral gap characterizes the $L^2$ relaxation time \eqref{eq:trelax}. However, when detailed balance is broken and the process becomes hypocoercive, the Poincar\'e inequality no longer holds. In this case, we demonstrate that the inverse singular value gap of the generator, which replaces the spectral gap, provides a lower bound for the relaxation time (\cref{lem:singulargap}). We remark that such a singular value gap estimate has also been widely used in analyzing classical non-reversible Markov chains \cites{fill1991eigenvalue,eberle2024non,chatterjee2023spectral}. 

Next, in \cref{sec:liftinglower}, we consider a family of non-primitive QMS generators $\{\mc{L}_\gamma\}_{\gamma > 0}$ on a finite-dimensional von Neumann algebra $\mc{M}$ in the reversible-irreversible decomposition:
\begin{align} \label{def:decompolga}
    \mc{L}_\gamma = \lh + \gamma \ld\,,
\end{align}
where $\lh$ and $\ld$ are self-adjoint and anti-self-adjoint, respectively, with respect to the $\si$-KMS inner product for a full-rank state $\si$ on $\mc{M}$ (\cref{asspA}). Here, $\si$ is an invariant state of $\exp(t \mc{L}^\dag)$, and the fixed-point algebra $\mc{F}(\mc{L}_\gamma)$ of $\mc{L}_\gamma$ is independent of $\gamma$ (\cref{lem:symanti}). To motivate the definition of lifts for quantum Markov dynamics, we derive the overdamped limit of $\exp(t \mc{L}_\gamma)$ as $\gamma \to \infty$. For the limiting effective dynamics to remain non-trivial, the QMS $\exp(t \mc{L}_\gamma)$ must be hypocoercive (\cref{asspB} and \cref{rem:asspb}). Let $\es$ denote the conditional expectation from $\mc{M}$ to the fixed-point algebra $\mc{F}(\ld)$ of $\ld$. If $\es \lh \es \neq 0$, a standard asymptotic analysis yields the effective dynamics generated by $\mc{L}_{O,(1)} := \es \lh \es$, which satisfies 
\begin{align*}
    \exp(t \mc{L}_\gamma) = \exp(t \mc{L}_{O,(1)}) + \Or(\gamma^{-1})\,.
\end{align*} 
In this case, $\mc{L}_\gamma$ can be defined as a first-order lift of $\mc{L}_{O,(1)}$, which is proven useful for discrete Markov chains \cite{chen1999lifting} but not for continuous diffusion dynamics \cite{eberle2024non}. When $\es \lh \es = 0$ (\cref{assump:PHP}) holds, a second-order perturbation expansion must be employed, with a time rescaling, to derive the limiting dynamics $\exp(t \mc{L}_{O,(2)})$ with $\mc{L}_{O,(2)} := - (\lh \es)^{\star} (- \ld)^{-1} \lh \es$, which satisfies, as shown in \cref{prop:overdampedmx}, 
\begin{align*}
    \exp(t \gamma \mc{L}_\gamma) = \exp(t \mc{L}_{O, (2)}) + \Or(\gamma^{-1})\,.
\end{align*} 
In this case, one may define $\mc{L}_\gamma$ is a second-order lift of $\mc{L}_{O, (2)}$. Note that deriving the overdamped limit $\mc{L}_{O, (2)}$ relies on Conditions~\ref{asspA}, \ref{asspB}, and~\ref{assump:PHP}. Also, the lifting operation should act on a detailed balanced QMS, requiring us to assume a priori that $\exp(t \mc{L}_{O, (2)})$ is a QMS (\cref{assump:overdp}). This motivates our formal definition of a second-order lift for quantum dynamics (\cref{def:2ndlift}): $\exp(t \mc{L}_\gamma)$ with \eqref{def:decompolga} is a second-order lift of $\lo$ if Conditions \ref{asspA}--\ref{assump:overdp} hold. However, we emphasize that \cref{assump:overdp} is only conceptually needed for lifting, as all convergence rate estimates of $\exp(t \mc{L}_\gamma)$ hold without it (see \cref{rem:asspforlift} and also \cite{li2024quantum}). 

Within the lifting framework, by leveraging the connection between the singular value gap of the generator and its $L^2$ convergence rate, in \cref{thm:lower}, we show that the $L^2$ convergence rate of the lifted QMS $\exp(t \mc{L}_\gamma)$ is upper bounded by the square root of the spectral gap of its overdamped limit (collapsed QMS) $\exp(t \lo)$: for any $\gamma > 0$, 
\begin{align} \label{est:upper}
    \nu(\mc{L}_\gamma) = \Or(\sqrt{\lad_O})\,,
\end{align}
where $\lad_O$ is the spectral gap of $\lo$. 
This shows that lifting can achieve, at most, a diffusive-to-ballistic speed-up, analogous to classical results \cites{chen1999lifting,eberle2024non}. To establish a lower bound for the convergence rate of $\exp(t \mc{L}_\gamma)$, we apply the variational hypocoercivity framework based on the space-time Poincaré inequality. The key idea is that, due to hypocoercivity, the standard Poincaré inequality (i.e., coercivity):
\begin{equation}\label{eq:stdpoincare}
    \lad \norm{X}_{2,\si}^2 \le \mc{E}_{\ld}(X)\,,\q \forall \text{\,$X \in \mc{M}$ with $E_{\mc{F}}(X) = 0$\,,}    
\end{equation}
fails for $\exp(t \mc{L}_\gamma)$, where $\mc{E}_{\ld}$ is the Dirichlet form associated with $\ld$, and 
$E_\mc{F}$ is the conditional expectation from $\mc{M}$ to the fixed-point algebra of $\mc{L}_\gamma$.  However, the desired coercivity can be recovered, if we consider the time-augmented form of \eqref{eq:stdpoincare}: 
\begin{align*}
    \lad \int_0^T \norm{X_t}_{2,\si}^2 \le \int_0^T \mc{E}_{\ld}(X_t)\,,\q \text{$X_t = \exp(t \mc{L}_\gamma) X_0$ with $E_\mc{F}(X_0) = 0$.}
\end{align*}
See also the discussion after \cref{rem:hypocoerpt}. 
In \cref{thm:uppermatrix}, we prove that the maximal $L^2$ convergence rate of $\exp(t \mc{L}_\gamma)$ satisfies:
\begin{align} \label{est:lower}
    \max_{\gamma > 0}\, \nu(\mc{L}_\gamma) \ge \Omega\left(\frac{\sqrt{\lad_O}}{1 + K_1 + K_2 \lad_O^{-1/2}}\right)\,,
\end{align}
where constants $K_1$ and $K_2$ are from \eqref{eq:constk1k2mm}. If $K_1 + K_2 \lad_O^{-1/2} = \Or(1)$, the upper and lower bounds of the convergence rate $\nu(\mc{L}_\gamma)$ match, yielding the optimal lift of $\lo$ satisfying
\begin{align*}
    \nu(\mc{L}_\gamma) = \Theta(\sqrt{\lad_O})\,.
\end{align*}

\medskip 

\noindent (2)  In \cref{sec:lifthilbert}, we extend the lifting framework from \cite{eberle2025convergence} to a general Hilbert space setting in an abstract manner, applicable to any symmetric contraction $C_0$ semigroup $\exp(t \fll)$ which is assumed to coercive (\cref{assp:locoercive}). This generalization encompasses both classical dynamics \cites{eberle2024non,eberle2025convergence} and quantum dynamics discussed in \cref{sec:liftingmatrix} as special cases. It is worth noting that Conditions \ref{asspA}--\ref{assump:overdp} are, in fact, stronger than those required for the proofs of 
\cref{thm:lower,thm:uppermatrix}. In the general framework, we identify the minimal conditions under which convergence rate estimates analogous to \eqref{est:upper} and \eqref{est:lower} hold for the lifted semigroup $\exp(t \fl)$ of $\exp(t \fll)$. 


Specifically, let $\fl$ be a lift of $\fll$ in the sense of \cref{def:generallift}. Using the singular value gap, we show in \cref{thm:lowerbound} that the convergence rate of $\exp(t \fl)$ is upper bounded by $\Or(\sqrt{\lad_O})$, corresponding to the best achievable quadratic acceleration, where $\lad_O$ is the spectral gap of $\fll$. For the lower bound (\cref{thm:lowerboundpt}), however, additional structural assumptions on $\fl$ are required beyond the basic lifting framework. 
To be precise, we impose \cref{assmp:decomposition} on $\fl$ and exploit the variational hypocoercivity theory with a flow Poincaré inequality (\cref{thm:flowpoincare}), which in turn relies on the technical \cref{assump:est1}. Again, when the upper and lower convergence rate estimates match, we can refer to $\fl$ as an optimal lift of $\fll$.

\medskip 

\noindent (3) We construct optimal lifts for a variety of detailed balanced classical and quantum Markov processes, beyond the examples shown in \cites{chen1999lifting,eberle2024non,eberle2024space,eberle2025convergence}.

Specifically, in \cref{sec:liftclassical}, we demonstrate how the general lifting framework in \cref{sec:lifthilbert} can be used to recover the results from \cites{eberle2024non,eberle2025convergence} on lifting reversible diffusions. In \cref{sec:liftfinitemarkov}, we show how to lift a symmetric finite Markov chain to a quantum Markov semigroup in the sense of \cref{def:2ndlift}. Furthermore, we prove that this construction yields the optimal quantum lift for the symmetric random walk on the circle, complementing the result in \cite{diaconis2000analysis}.

Finally, in \cref{sec:appbipart}, we address how to lift a detailed balanced QMS $\exp(t \lo)$ on $\mc{M}_B$ to a QMS $\exp(t \mc{L}_\gamma)$ on a bipartite quantum system $\mc{M}_A \otimes \mc{M}_B$, where $\mc{M}_A$ is an extended space to be constructed. To achieve this, we first show in \cref{sec:characterlift} that the conditions in \cref{asspA}--\ref{assump:overdp} used to define the second-order lift of QMS (\cref{def:2ndlift}) can be simplified or, in some cases, eliminated. Then, in \cref{sec:constructlift}, we establish that for any $\si$-GNS detailed balanced QMS, there exists a lift in the sense of \cref{def:2ndlift}; see \cref{thm:liftgns}. 
Moreover, we prove in \cref{thm:optimalliftgns} that our construction is optimal if $\lo$ is detailed balanced with respect to the maximally mixed state $\si \propto \mi$ and satisfies the $\kappa$-intertwining condition with $\kappa > 0$ (\cref{def:intercondi}). This class of QMS includes depolarizing semigroups, Schur multipliers, and infinite-temperature Fermi and Bose Ornstein-Uhlenbeck semigroups, among others.

\subsection{Other related works} Here we briefly review recent progress on the mixing time analysis of quantum Markov semigroups,  for the sake of completeness. 

 The most common approach for estimating the mixing time of a detailed-balanced Markov semigroup is through the spectral gap, which corresponds to the Poincar\'e inequality constant \cites{diaconis1991geometric,temme2010chi}. 
 However, the spectral gap approach often overestimates the mixing time, as it bounds the $L^1$-type norm convergence using the $L^2$-type norm. 
 A general approach for improved bounds involves the quantum modified log-Sobolev inequality (MLSI), which quantifies the exponential convergence of the relative entropy \cites{majewski1996quantum,olkiewicz1999hypercontractivity,kastoryano2013quantum}.
Inspired by the mixing analysis of classical spin systems \cites{stroock1992logarithmic,stroock1992equivalence,martinelli1999lectures}, recent research has focused on proving uniformly positive spectral gaps and MLSI constants for Davies generators associated with geometrically local, commuting Hamiltonians through the static properties of the Gibbs state; see \cites{kastoryano2016quantum,bardet2024entropy,kochanowski2024rapid} and the references therein. It is important to note that the detailed balance condition, along with the commutativity and locality of the Hamiltonians, are essential components in the aforementioned results.


The analysis of rapid mixing in the presence of noncommuting Hamiltonians remains a largely open and challenging research area. Recently, Rouz\'e et al.~\cite{rouze2024efficient} demonstrated the fast mixing with a mixing time of $t_{\rm mix} = \mc{O}(n)$ for quasi-local Lindbladians featuring general local noncommuting Hamiltonians at high temperatures. This result was further improved in \cite{rouze2024optimal}, where the authors employed the oscillator norm technique to achieve rapid mixing. For additional developments in this direction, we refer readers to \cites{tong2024fast,zhan2025rapid}. 

Regarding the mixing behavior of QMS without detailed balance, which is also an open area of study, Laracuente \cite{laracuente2022self} analyzed general QMS involving both coherent and dissipative components, and proved that when dissipation dominates coherent time evolution, the exponential convergence rate to the decoherence-free subspace, as measured by quantum relative entropy, is upper bounded by the inverse decay rate of the dissipative component. More recently, \cite{fang2025mixing} and \cite{li2024quantum} extended the classical DMS method \cite{dolbeault2015hypocoercivity} and the variational framework based on space-time Poincaré inequalities \cites{albritton2019variational,cao2023explicit} to investigate the convergence of hypocoercive QMS without detailed balance. Our work follows along this line of research.

\section{Lifting quantum Markov semigroups} \label{sec:liftingmatrix}

To illustrate the main ideas of lifting, we focus on quantum Markov dynamics on matrix algebras in this section. In \cref{sec:hypoqms}, we introduce the hypocoercive quantum Markov semigroup (QMS) and its fundamental convergence properties.
We shall see that the lifted dynamics studied in this work form a special subclass of hypocoercive semigroups. Then, in \cref{sec:liftinglower}, we 
derive the overdamped limit for hypocoercive quantum dynamics and introduce the lifting of detailed balanced quantum Markov semigroups, which could be formally interpreted as the \emph{reverse} of the overdamped limit. 
In \cref{sec:upperlift}, we establish two main results: (1) the lifted QMS can achieve at most a square-root improvement in the $L^2$ relaxation time, and (2) an extension of the recent hypocoercivity framework based on space-time Poincar\'e inequalities to derive upper bounds on the relaxation time of lifted QMS. 
The general lifting framework in the Hilbert space setting will be investigated in \cref{sec:lifthilbert}.

We begin by fixing the notations used below. Let $\maf{H}$ be a finite-dimensional Hilbert space, and $\mathcal{B}(\maf{H})$ be the associated algebra of linear operators. We denote by $\mi \in \bh$ the identity element and by $\id$ the identity map on $\bh$. We also denote by $X^*$ the adjoint of $X \in \mc{B}(\maf{H})$, and by $\l X, Y\r = \tr (X^* Y)$ the Hilbert-Schmidt (HS) inner product on $\bh$ with the induced norm $\norm{X} = \sqrt{\l X, X\r}$. Let $\mc{M} \subset \bh$ be a von Neumann algebra. A quantum state (or density operator) on $\mc{M}$ is an element $\rho \in \mc{M}$ satisfying $\rho \succeq 0$ and $\tr(\rho) = 1$. We denote by $\mc{D}(\mc{M})$ the set of all quantum states, and by $\mc{D}_+(\mc{M})$ the subset of full-rank states. If $\mc{M} = \bh$, we shall simply write $\dh$ and $\dhh$.
The adjoint of a superoperator $\Phi: \mc{B}(\maf{H}) \to \mc{B}(\maf{H})$ (generally, $\mc{M} \to \mc{M}$) for HS inner product is written as $\Phi^\dag$. Throughout this work, we use the following asymptotic notations along with the standard big $\Or$ notation. We write $f = \Omega(g)$ if $g = \Or(f)$, and $f = \Theta(g)$ if both $f = \Or(g)$ and $g = \Or(f)$ hold.

\subsection{Hypocoercive Quantum Markov semigroups} \label{sec:hypoqms}

A \emph{quantum Markov semigroup} $(\mc{P}_t)_{t \ge 0}: \mc{M} \to \mc{M}$ is defined as a semigroup of completely positive and unital maps. Its generator, defined by $\mc{L}(X): = \lim_{t \to 0} t^{-1}(\mc{P}_t (X) - X)$, has the GKSL representation \cites{Lindblad1976,GoriniKossakowskiSudarshan1976}: there exists a Hamiltonian $H \in \bh$ and jump operators $\{L_j\}_{j \in \mc{J}} \subset \bh$ such that  
\begin{align} \label{eq:lindbladform}
     \mc{L} (X) = i [H, X] + \sum_{j \in \mc{J}} \left( L_j^* X L_j - \frac{1}{2}\left\{L_j^* L_j, X  \right\} \right)\,,\q X \in \mc{M}\,,
\end{align}
where $\mc{J}$ is a finite index set. Given a full-rank quantum state $\si \in \mc{D}_+(\maf{H})$, the \emph{modular operator} is defined as 
\begin{equation} \label{def:modularoperator}
\Delta_{\si}(X) = \si X \si^{-1}\,,\q  \forall X \in \bh\,.
\end{equation}
We introduce a family of inner products on $\bh$, parameterized by $s \in \R$:
\begin{align} \label{def:s-inner}
\l X, Y \r_{\si,s} := \tr(\si^s X^* \si^{1-s} Y) = \l X, \Delta_\si^{1-s}(Y) \si \r\,, \q \forall X, Y \in \bh\,,
\end{align}
with $\l \dd, \dd \r_{\si,1}$ and $\l \dd, \dd \r_{\si,1/2}$ corresponding to the GNS and KMS inner products, respectively. 
\begin{definition}
    A QMS $\mc{P}_t = \exp(t \mc{L}): \mc{M} \to \mc{M}$ is $\si$-{\rm GNS} (resp., $\si$-{\rm KMS}) 
    \emph{detailed balanced} for a full-rank state $\si \in \mc{D}_+(\mc{M})$ if the generator $\mc{L}$ is self-adjoint with respect to the GNS inner product $\l\dd, \dd\r_{\si,1}$  (resp., the KMS inner product $\l\dd, \dd\r_{\si,1/2}$).     
\end{definition}
We remark that the KMS detailed balance is a strictly weaker notion than the GNS one \cite{carlen2017gradient}*{Appendix B}. For a QMS generator $\mc{L}$ with a full-rank invariant state $\si \in \mc{D}_+(\mc{M})$, we introduce its Dirichlet form given by 
\begin{align*}
    \mc{E}_{\mc{L}}(X,Y): = - \l X, \mc{L} (Y) \r_{\si,1/2}\,,\q \forall\, X, Y \in \mc{M}\,.
\end{align*}
If $X = Y$, we simply write $\mc{E}_{\mc{L}}(X) = \mc{E}_{\mc{L}}(X, X)$. We also define the modulus of $X \in \bh$ by $|X|: = \sqrt{X^* X}$ and the $\si$-weighting operator for $\si \in \dhh$ by 
\begin{align} \label{def:weighting}
    \Gamma_\si (X) := \si^{\frac{1}{2}}X\si^{\frac{1}{2}}\,.
\end{align} 
The \emph{noncommutative $2$-norm} induced by the KMS inner product is given by 
\begin{align*}
    \norm{X}_{2,\si} := \tr\Bigl(|\Gamma_\si^{1/2}(X)|^2 \Bigr)^{1/2} = \tr\bigl( \si^{1/2} X \si^{1/2} X \bigr)^{1/2}.
\end{align*}
Moreover, we denote by $\Phi^\star$ the adjoint of a superoperator $\Phi: \mc{M} \to \mc{M}$ for the KMS inner product $\l X, \Phi (Y) \r_{\si, 1/2} = \l \Phi^\star(X), Y\r_{\si,1/2}$. In particular, the following representation holds:
\begin{equation} \label{eq:adjointkms}
    \Phi^\star = \Gamma_\si^{-1} \circ \Phi^\dag \circ \Gamma_\si\,,
\end{equation}
where $\Phi^\dag$ is the adjoint under the Hilbert-Schmidt inner product.  

We next introduce some basic mixing properties of a QMS $\mc{P}_t = \exp(t \mc{L})$. We first define the fixed-point space $\mc{F}(\mc{L})$ and the decoherence-free subalgebra $\mc{N}(\mc{L})$ as follows: 
\begin{align} \label{def:fixedpointspace}
    \mc{F}(\mc{L}) := \{X \in \mc{M}\,;\ \mc{P}_t(X) = X\,,\ \forall t \ge 0\} = \ker(\mc{L})\,,
\end{align}
and 
\begin{align} \label{def:decoherence}
    \mc{N}(\mc{L}) := \{X \in \mc{M}\,;\ \mc{P}_t(X^* X) = \mc{P}_t(X)^*  \mc{P}_t(X) \,,\ \forall t \ge 0\}\,.
\end{align}
Assuming that $\mc{P}_t^\dag$ admits a full-rank invariant state $\si$, the algebra $\mc{N}(\mc{L})$ is the largest von Neumann subalgebra $\mc{N}$ of $\mc{M}$ such that $\mc{P}_t|_{\mc{N}}$ is a group of $*$-automorphisms \cites{robinson1982strongly,carbone2015environment}.
In the finite-dimensional setting, the algebra $\mc{N}(\mc{L})$ is invariant under the modular automorphism group $\Delta_\si^{is}$, which, by \cref{lem:extencece}, guarantees the existence of the conditional expectation onto $\mc{N}(\mc{L})$ denoted by $E_{\mc{N}}$. In addition, the subspace $\mc{F}(\mc{L})$ also forms a von Neumann algebra \cites{spohn1977algebraic,frigerio1978stationary}. 
Furthermore, the following ergodicity holds.  
\begin{lemma}[\cite{frigerio1982long}*{Theorems 3.3 and 3.4}]\label{lem:ergopt}
    Let $\mc{P}_t = \exp(t \mc{L})$ be a QMS with invariant state $\si \in \mc{D}_+(\mc{M})$. Then, we have 
    \begin{align*}
        \lim_{t \to \infty} \mc{P}_t (\id - E_{\mc{N}}) = 0\,,
    \end{align*}
    with $E_{\mc{N}}$ being the conditional expectation to $\mc{N}(\mc{L})$, and the limit 
    \begin{align} \label{def:cefixed}
        \lim_{t \to \infty} \frac{1}{t} \int_0^t \mc{P}_s \ud s =: E_{\mc{F}}
    \end{align}
    exists and defines the conditional expectation $E_{\mc{F}}$ onto $\mc{F}(\mc{L})$ with respect to $\si$. 
    Moreover, the following convergence holds:
    \begin{align} \label{eq:ergodic}
        \lim_{t \to \infty} \mc{P}_t  = E_{\mc{F}}\,,
    \end{align}
    if and only if $\mc{N}(\mc{L})$ coincides with the fixed-point algebra $\mc{F}(\mc{L})$, i.e., $\mc{N}(\mc{L}) = \mc{F}(\mc{L})$.  
\end{lemma}

The basics of the noncommutative conditional expectation are recalled in \cref{app:quantumce} for completeness. In particular, \cref{coro:uniquence} shows that $E_{\mc{F}}$ and 
$E_{\mc{N}}$ are orthogonal projections with respect to both GNS and KMS inner products. 

\begin{definition}
    We say that
    a QMS $\mc{P}_t = \exp(t \mc{L})$ with a full-rank invariant state is \emph{ergodic} if the convergence in \eqref{eq:ergodic} holds, equivalently, $\mc{N}(\mc{L}) = \mc{F}(\mc{L})$. 
    We say that $\mc{P}_t$ is \emph{primitive} if it admits a unique full-rank invariant state $\si$.    
\end{definition}

In the primitive case, we have $E_{\mc{F}}(X) = \tr(\si X) \mi$ and, by \cite{wolf2012quantum}*{Proposition 7.5}, 
\begin{align} \label{eq:conver_qmsg}
    \lim_{t \to \infty} \mc{P}_t(X) = \tr(\si X)\mi\,,\q \forall\, X \in \mc{M}\,.
\end{align}
The structures of the subalgebras $\mc{N}(\mc{L})$ and $\mc{F}(\mc{L})$ and their connections to the ergodicity of $\mc{P}_t$ can also be analyzed through the spectral properties of the generator $\mc{L}$; see \cite{carbone2013decoherence}*{Proposition 8 and Theorem 9} and \cite{carbone2015environment}*{Theorem 29 and Proposition 31}. We summarize them into the following results. 
\begin{lemma} \label{lem:characnf}
    Let $\mc{P}_t = \exp(t \mc{L})$ be a QMS with an invariant state $\si \in \mc{D}_+(\mc{M})$, and $\mc{F}(\mc{L})$ and $\mc{N}(\mc{L})$ be the subalgebras defined in \eqref{def:fixedpointspace} and \eqref{def:decoherence}. Then,
    \begin{itemize}
        \item Any purely imaginary eigenvalue $\lad = i \mu$ with $\mu \in \R$ of $\mc{L}$ is simple, namely, the associated Jordan block is of size one. 
        \item Define $V_0 := {\rm Span}\{X \in \mc{M}\,;\ \mc{L}(X) = i \mu X \ \text{for some $\mu \in \R$}\}$, and let $V_-$ be the direct sum of generalized eigenspaces of $\mc{L}$ associated with eigenvalues $\lad$ with $\Re(\lad) < 0$. 
        We have 
        \begin{equation*}
            V_0 = \mc{N}(\mc{L})\,, \q V_- = \{X \in \mc{M}\,;\ \lim_{t \to \infty} \mc{P}_t(X) = 0\}\,,
        \end{equation*}
        and the orthogonal decomposition with respect to both $\si$-GNS and KMS inner products:
        \begin{align*}
            \mc{M} =  V_0 \oplus V_-\,.
        \end{align*}
    \end{itemize}
\end{lemma}
\begin{corollary}\label{coro:ergodiceig}
     Let $\mc{P}_t = \exp(t \mc{L})$ be a QMS with an invariant state $\si \in \mc{D}_+(\mc{M})$. Then $\mc{P}_t$ is ergodic if and only if $\mc{L}$ has no purely imaginary eigenvalues $\lad = i\mu$ with real $\mu \neq 0$. In this case, every eigenvalue $\lad \in \C$ of $\mc{L}$ satisfies either $\lad = 0$ or $\Re \lad < 0$.  
\end{corollary}
For an ergodic QMS $\mc{P}_t$ with a full-rank invariant state $\si$, by \cref{coro:ergodiceig}, we can define the spectral gap of its generator  $\mc{L}$ by
\begin{equation} \label{eq:spectralgap}
     \lad(\mc{L}) := \inf \left\{\Re(\lad)\,;\ \lad \in {\rm Spec}(-\mc{L})\backslash\{0\}\right\}\,.
\end{equation}
We also define the $L^1$ mixing time as
\begin{align} \label{eq:tmix}
    t_{\rm mix}(\mc{L}) := \inf\left\{t \ge 0\,;\ \norm{\mc{P}_t^\dag(\rho) - E_{\mc{F}}^\dag(\rho)}_{\tr} \le e^{-1}\, \  \text{for all}\ \, \rho \in \mc{D}(\mc{M})\right\},
\end{align}
and the $L^2$ relaxation time as 
\begin{align} \label{eq:trelax}
    t_{\rm rel}(\mc{L}) := \inf\left\{t \ge 0\,;\ \norm{\mc{P}_t (X)}_{2,\si} \le e^{-1} \norm{X}_{2,\si}\ \, \text{for all $X \in \mc{M}$ with $E_{\mc{F}}(X) = 0$}\right\}.
\end{align}
The relaxation time $t_{\rm rel}$ generally depends on the choice of the full-rank invariant state $\si$, unless $\mc{P}_t$ satisfies certain special conditions, such as being $\si$-GNS detailed balanced \cite{gao2021complete}*{Lemma 2.6}. In the following, we always fix the reference invariant state $\si$ used in the KMS inner product $\l \dd, \dd \r_{\si,1/2}$ and the norm $\norm{\dd}_{2,\si}$. In the case of primitive QMS, $t_{\rm mix}$ and $t_{\rm rel}$ can be easily related by the following lemma. 

\begin{lemma} \label{lem:relatingmix}
    Let $\mc{P}_t$ be a primitive QMS with invariant state $\si \in \mc{D}_+(\mc{M})$. Then there holds
    \begin{align*}
        t_{\rm mix}(\mc{L}) \le \left(2 + \log \left(\si_{\min}^{-1/2}\right)\right) t_{\rm rel}(\mc{L})\,,
    \end{align*}
    where $\si_{\min}$ is the minimal eigenvalue of $\si$. 
\end{lemma}

The proof of \cref{lem:relatingmix} is given in \cref{appA}. Here, the minimal eigenvalue $\sigma_{\min}$ typically exhibits polynomial scaling in the system dimension $\dim(\mh)$. 
Since this work is not aimed at improving the dimension dependence of the mixing time $t_{\rm mix}$, it suffices to focus on estimating $t_{\rm rel}$, i.e., the $L^2$ convergence rate. 

In general, for an ergodic QMS $\mc{P}_t = \exp(t \mc{L})$ with invariant state $\si$, 
without the detailed balance, one can only expect: for $C \ge 1$ and $\nu > 0$, 
\begin{equation} \label{ex:expdecay}
    \norm{\mc{P}_t (X) - E_{\mc{F}}(X)}_{2,\si} \le C e^{- \nu t} \norm{X - E_{\mc{F}}(X)}_{2,\si}\,.
\end{equation}
According to \cite{engel2000one}*{Chapter IV}, the sharp convergence rate 
$$
\nu_0 := \sup\{\nu > 0\,; \ \text{there exists $C \ge 1$ such that \eqref{ex:expdecay} holds}\}
$$
is given by the spectral gap \eqref{eq:spectralgap}:
\begin{equation} \label{eq:sharprate_gap}
    \nu_0 = \lad(\mc{L}) =  - \lim_{t \to \infty} \frac{1}{t} \log \sup_{X \neq 0}\frac{\norm{\mc{P}_t (X) - E_{\mc{F}}(X)}_{2,\si}}{\norm{X - E_{\mc{F}}(X)}_{2,\si}}.
 \end{equation}
 The prefactor $ C\ge 1$ in \eqref{ex:expdecay} is, in general, unavoidable and actually characterizes the hypocoercivity of the dynamics. The following definition is adapted from \cite{villani2009hypocoercivity}*{Theorem 12}, which considered classical dynamics. 

\begin{definition} \label{def:coercive}
An ergodic QMS $\mc{P}_t$ 
with a full-rank invariant state $\si$ 
is \emph{hypocoercive} if the estimate \eqref{ex:expdecay} only holds for $C > 1$ and $\nu > 0$, and it is \emph{coercive} if \eqref{ex:expdecay} holds with $C = 1$. 
\end{definition}

 The  hypocoercivity of a QMS  $\mathcal{P}_t = \exp(t \mc{L})$ is essentially determined by the equilibrium space of its 
 symmetric part $(\mc{L} + \mc{L}^\star)/2$. 

\begin{lemma} \label{lem:hypoqms}
Let QMS $\mathcal{P}_t = \exp(t \mc{L})$ be ergodic with an invariant state $\si \in \mathcal{D}_+(\mc{M})$, and let $\mc{F}(\mc{L})$ be the fixed-point algebra defined in \eqref{def:fixedpointspace}. 
Then we have $\mathcal{F}(\mc{L}) \subset \ker(\mc{L} + \mc{L}^\star)$, and 
$\mathcal{P}_t$ is hypocoercive, i.e., $C > 1$ in \eqref{ex:expdecay}, if and only if 
\begin{align*}
    \dim \ker(\mc{L} + \mc{L}^\star) > \dim \mathcal{F}(\mc{L})\,,
\end{align*}
that is, $\mathcal{F}(\mc{L})$ is a strict subspace of $\ker(\mc{L} + \mc{L}^\star)$. 
\end{lemma}

    \begin{proof}
        We first note that $\mc{L}(X) = 0$ implies 
        $$
        \l X, (\mc{L} + \mc{L}^\star)X \r_{\si,1/2} = 2\, \Re \l X, \mc{L}(X) \r_{\si,1/2} = 0\,.
        $$
        Since $\mc{L}$ is dissipative, the operator $\mc{L} + \mc{L}^\star$ is negative, and then $(\mc{L} + \mc{L}^\star)X = 0$. Therefore, 
        $\ker(\mc{L}) \subset \ker(\mc{L} + \mc{L}^\star)$ holds. Then, by taking the time derivative of $$\norm{\mc{P}_t (X) - E_{\mc{F}}(X)}^2_{2,\si} \le e^{-  2 \nu t} \norm{X - E_{\mc{F}}(X)}_{2,\si}^2\,,$$ and Gr\"{o}nwall's inequality, we find that the coercivity of $\mc{P}_t$ (i.e.,
        the exponential decay \eqref{ex:expdecay} with $C = 1$) is equivalent to
        \begin{align*}
            \nu \norm{X - E_{\mc{F}}(X)}_{2,\si}^2 \le - \Big\l X, \frac{\mc{L} + \mc{L}^\star}{2} X \Big\r_{\si,1/2}\,, \q \text{for $X \in \mc{M}$\,,}
        \end{align*}
        which implies
        \begin{equation*}
         \ker\Big(\frac{\mc{L} + \mc{L}^\star}{2}\Big) \subset \mc{F}(\mc{L}) = \ker(\mc{L})\,.
        \end{equation*}
        The proof is complete. 
        \end{proof}

In addition, similarly to  \cite{eberle2024non}*{Lemma 10} and \cite{li2024quantum}*{Lemma 2.5}, we can show that the $L^2$ relaxation time of an ergodic $\mc{P}_t$ is lower bounded by the inverse singular value gap of $\mc{L}$, which is crucial for characterizing the potential quadratic speedup of lifted QMS (see \cref{thm:lower}).
We define the singular value gap of an ergodic $\mc{L}$ with a full-rank invariant state $\si$ by the spectral gap of $\sqrt{\mc{L}^\star\mc{L}}$, that is,  
$s(\mc{L}) := \lad(\sqrt{\mc{L}^\star\mc{L}})$. We have 
\begin{align*}
\frac{1}{s(\mc{L})} = \sup_{0 \neq X \in \mathcal{F}(\mc{L})^\perp} \frac{\norm{\mc{L}^{-1} (X)}_{2,\si}}{\norm{X}_{2,\si}}\,.
\end{align*}
where the orthogonal complement $\perp$ is with respect to the KMS inner product. 

\begin{lemma} \label{lem:singulargap}
Given an ergodic QMS generator $\mc{L}$ with an invariant state $\si \in \mc{D}_+(\mc{M})$, let $s(\mc{L})$ be the singular value gap of $\mc{L}$, $C \ge 1$ and $\nu > 0$ be constants in exponential convergence \eqref{ex:expdecay}, and  let $t_{\rm rel}$ be the relaxation time \eqref{eq:trelax}. Then, it holds that 
\begin{align*}
\nu \le (1 + \log C) s(\mc{L})\,,\quad
    t_{\rm rel} \ge  \frac{1}{2 s(\mc{L})}\,.
\end{align*}
\end{lemma}
The proof is similar to that of \cite{li2024quantum}*{Lemma 2.5} and is therefore omitted. See also \cref{lem:singulargapgeneral} for a generalization to the more abstract setting.

\subsection{Overdamped limit and lifting of quantum dynamics} \label{sec:liftinglower}

We now develop the lifting framework for quantum Markov semigroups. To begin, we consider the following family of generators for QMS, analogous to the classical kinetic Fokker-Planck equation:
\begin{align} \label{eq:decom}
     \mc{L}_\gamma = \mc{L}_A + \gamma \mc{L}_S\,,
\end{align}
where $\gamma > 0$ is the damping parameter (see  \cite{li2024quantum}*{Remark 1}). Suppose that $\mc{L}_\gamma$ satisfies the following condition:

\begin{condition}\label{asspA}
The operators $\ld$ and $\lh$, independent of $\gamma$, are self-adjoint and anti-self-adjoint for $\si$-KMS inner product for some $\si \in \mc{D}_+(\mc{M})$:
\begin{equation*}
    \ld = \ld^\star\,, \quad  \lh = - \lh^\star\,.
\end{equation*}
\end{condition}

\begin{remark} 
A paradigmatic choice of $\lh$ is the generator $\lh(\dd) := i [G,\dd]$ of a unitary group defined by some Hamiltonian $G$ satisfying $[G, \si] = 0$. Any generator $\mc{L}$ can be decomposed into a reversible-irreversible form as in \eqref{eq:decom}. Specifically, we may take $\ld$ as the additive symmetrization $\ld = (\mc{L} + \mc{L}^\star)/(2\gamma)$, and then $\lh = \mc{L} - \gamma \ld$, although both terms generally depend on $\gamma$. Moreover, 
this condition can be relaxed as in \cref{assmp:decomposition}, where we only assume the anti-symmetry of $\lh$ on the subspace $\ker(\ld)$.  
\end{remark}

\begin{lemma} \label{lem:symanti}
Under \cref{asspA}, the following hold:
\begin{itemize}
    \item Both $\lh$ and $\ld$ admit $\si$ as an invariant state, i.e.,
    \begin{align*}
        \ld^\dag(\si) = \lh^\dag(\si) = 0\,,   
    \end{align*}
    and thus $\mc{L}_\gamma^\dag(\si) = 0$ for all $\gamma > 0$.
    \item $\ld$ is the generator of a $\si$-KMS detailed balanced quantum Markov semigroup. 
    \item The fixed-point algebra $\mc{F}(\mc{L}_\gamma)$ is characterized by 
    \begin{align} \label{eq:kerlamma}
        \mc{F}(\mc{L}_\gamma) = \ker(\mc{L}_\gamma) = \ker(\lh) \cap \ker(\ld)\,,
    \end{align}
    which is independent of $\gamma$. 
\end{itemize}
\end{lemma}

\begin{proof}
Note that the semigroup generated by $\mc{L}_\gamma^\star = \Gamma_\si^{-1} \circ \mc{L}_\gamma^\dag \circ \Gamma_\si$ is also a QMS. Then 
    $\ld = (\mc{L}_\gamma + \mc{L}_\gamma^\star)/(2\gamma)$, as a linear combination of $\mc{L}_\gamma$ and $\mc{L}_\gamma^\star$, is a QMS generator with $\si$-KMS detailed balance, and hence $\ld^\dag(\si) = 0$. Next, since both $\mc{L}_\gamma$ and $\ld$ are QMS generators, there holds $\mc{L}_\gamma(\mi) = \ld(\mi) = 0$. Combined with the fact that $\lh$ is anti-self-adjoint, it gives 
    \begin{equation*}
      0 = \l X, \lh(\mi)\r_{\si,1/2} = - \l \lh(X), \si\r = - \l X, \lh^\dag(\si)\r\,,\q \forall X \in \mc{M}\,,
    \end{equation*}
    that is, $\lh^\dag(\si) = 0$. For the last item, by \cref{lem:hypoqms}, $\mc{L}_\gamma (X) = 0$ gives  $\ld(X) = 0$, and hence $\lh(X) = 0$. It follows that $\ker(\mc{L}_\gamma) \subset \ker(\lh) \cap \ker(\ld)$. The reverse direction follows directly from \eqref{eq:decom}. The proof is complete. 
    \end{proof}

For later use, we denote by $\mc{F}(\mc{L}_S)$ the fixed-point algebra of $\ld$ as in \eqref{def:fixedpointspace}, and denote by $E_{\mc{F}}$ and 
$\es$ the conditional expectations from $\mc{M}$ to $\mc{F}(\mc{L}_\gamma)$ and
$\mc{F}(\ld)$, respectively, which are defined similarly to \cref{def:cefixed}. Again, these conditional expectations are also the orthogonal projections with respect to the $\si$-KMS inner product. 

\medskip 

\noindent \textbf{Overdamped Limit.} 
To motivate our definition of lifted QMS, we next derive the leading-order dynamics of $\exp(t \mc{L}_\gamma)$ in the overdamped limit $\gamma \to \infty$. 
We will see that these effective dynamics are confined to the fixed-point algebra $\mc{F}(\mc{L}_S)$.
For this derivation to be non-trivial,
$\exp(t \mc{L}_\gamma)$ must exhibit hypocoercivity ensured by the following \cref{asspB}, due to \cref{lem:hypoqms}. 

\begin{condition}\label{asspB} 
$\dim \mc{F}(\ld) > \dim \mc{F}(\mc{L}_\gamma) $.
\end{condition}

\begin{remark} \label{rem:asspb}
In fact, if $\mc{L}_\gamma$ is coercive, then by \cref{lem:hypoqms}, we have $\mc{F}(\ld) = \mc{F}(\mc{L}_\gamma)$. Furthermore, \cref{eq:kerlamma} in \cref{lem:symanti} implies that $\mc{F}(\ld) \subset \ker(\lh)$. As a result, we find that $\lh \es = 0$, meaning that the leading-order dynamics derived in \eqref{eq:eff1st} and \eqref{eq:eff2nd} below are constant.
\end{remark}

To match with the usual convention of asymptotic analysis, we write the small parameter $\veps = \gamma^{-1} \ll 1$ and consider the evolution
\begin{align} \label{eq:asym1}
    \frac{\rd}{\rd t} X^\veps(t) = (\lh + \veps^{-1} \ld) X^\veps(t)\,.
\end{align}
Assume the solution admits a formal asymptotic expansion:
\begin{align*}
    X^\veps(t) = X_0(t) + \veps X_1(t) + \mathcal{O}(\veps^2)\,.  
\end{align*}
Substituting this expansion into \eqref{eq:asym1} yields
\begin{align*}
    \veps^{-1} \ld X_0 + 
    \left( \ld X_1 + \lh X_0 - \frac{\rd}{\rd t} X_0 \right) 
    + \mathcal{O}(\veps) = 0\,.    
\end{align*}
Matching powers of $\veps$ leads to:
\begin{align}
     \mathcal{O}(\veps^{-1}): \quad &\ld X_0 = 0\,, \label{subeq1} \\ 
     \mathcal{O}(1): \quad &\ld X_1 = \frac{\rd}{\rd t} X_0 - \lh X_0 \,. \label{subeq2} 
\end{align}
From \eqref{subeq1}, we deduce that $\es X_0(t) = X_0(t)$ for all $t \geq 0$. Combining this with \eqref{subeq2} yields the effective dynamics:
\begin{align} \label{eq:eff1st}
    \frac{\rd}{\rd t} X_0(t) = \es \lh \es X_0(t)\,.
\end{align}
It shows that $X^\veps(t)$ is approximated by $X_0(t)$ living in $\mc{F}(\mc{L}_S)$, governed by \cref{eq:eff1st}. 

However, when $\es \lh \es = 0$, the leading-order dynamics \eqref{eq:eff1st} only yields $X_0$ being constant in time. To capture non-trivial dynamics in this case, we must rescale time as $t \to \gamma t$ and consider the modified evolution:
\begin{align} \label{eq:asym2}
    \frac{\rd}{\rd t} X^\veps(t) = (\veps^{-1}\lh + \veps^{-2} \ld) X^\veps(t)\,.
\end{align}
We again seek a solution through asymptotic expansion:
\begin{align*}
    X^\veps(t) = X_0(t) + \veps X_1(t) + \veps^2 X_2(t) + \mathcal{O}(\veps^3)\,.
\end{align*}
Substituting this ansatz into \eqref{eq:asym2} and matching orders yields:
\begin{align}
     \mathcal{O}(\veps^{-2}): \quad &\ld X_0 = 0\,, \label{eq:0order} \\
     \mathcal{O}(\veps^{-1}): \quad &\ld X_1  = -\lh X_0 \,, \label{eq:1order} \\
     \mathcal{O}(1): \quad &\ld X_{2} = \frac{\rd}{\rd t} X_0 - \lh X_{1} \,. \label{eq:norder}
\end{align}
The leading-order equation \eqref{eq:0order} implies $\es X_0 = X_0$ as before, while the solvability condition for \eqref{eq:1order} requires $\ran(\lh \es) \subset \ran(\ld)$. This is equivalent to the key structural assumption:
\begin{condition} \label{assump:PHP}
    $\es \lh \es = 0$.
\end{condition}

Under \cref{assump:PHP}, we obtain the first-order correction:
\begin{equation*}
    X_1 = -\ld^{-1} \lh X_0 + \ker(\ld)\,,
\end{equation*}
where $\ld^{-1}$ denotes the pseudoinverse. To close the equation for the leading-order term $X_0$, we apply the projection $\es$ to both sides of \eqref{eq:norder}, yielding:
\begin{equation} \label{eq:eff2nd}
    \begin{aligned}
        \frac{\rd}{\rd t} X_0 &= \es \lh X_1 \\
        &= -\es\lh \ld^{-1} \lh X_0 \\
        &= - (\lh \es)^\star (-\ld)^{-1} (\lh \es) X_0\,.
    \end{aligned}
\end{equation}
In the final equality, we used \cref{asspA} (which states $\lh^\star = -\lh$) and the fact that $\es$ is $\si$-KMS detailed balanced. Thus, in this regime, the dynamics $X^\veps(t)$ is effectively described by $X_0(t)$ evolved as \eqref{eq:eff2nd}, which is also supported on $\mc{F}(\mc{L}_S)$.

We define the effective generator on $\mc{F}(\mc{L}_S)$ from \eqref{eq:eff2nd}:
\begin{equation} \label{eq:odlimit_generator}
 \mc{L}_O := - (\lh \es)^{\star} (- \ld)^{-1} \lh \es\,,
\end{equation}
which satisfies the following properties:
\begin{lemma}\label{lem:properlo}
For the operator $\lo$ in \eqref{eq:odlimit_generator}, it holds that $\mc{L}_O(\mi) = 0$ and $\lo$ is negative self-adjoint with respect to the KMS inner product $\l \dd, \dd \r_{\si,1/2}$. It follows that $\exp(t \mc{L}_O)$ is a unital contraction semigroup.  
\end{lemma}
For completeness, we conclude this subsection by rigorously justifying the approximation of the QMS $\exp(t \mc{L}_\gamma)$ by $\exp(t \mc{L}_O)$ in the limit $\gamma \to \infty$. The proof is provided in \cref{appA}. A similar convergence result holds for the case \eqref{eq:eff1st} and will be omitted. 

\begin{proposition} \label{prop:overdampedmx}
Under Conditions \ref{asspA}, \ref{asspB}, and \ref{assump:PHP}, let $X^\ep$ and $X_0$ solve \eqref{eq:asym2} and \eqref{eq:eff2nd}, respectively, with $X^\ep(0) = X_0(0)$. We have $\norm{X^\ep(t) - X_0(t)}_{2,\si} = \mc{O}(\ep)$ for any $t \ge 0$. 
\end{proposition}

\medskip

\noindent \textbf{Lifting quantum Markov process.}  We next introduce the concept of lifted quantum Markov semigroups. Let $\mc{L}_\gamma$ be the QMS generator in \eqref{eq:decom}, satisfying \cref{asspA}. By \cref{lem:symanti}, the dual semigroup $\exp(t \mc{L}_\gamma^\dag)$ admits a full-rank invariant state $\si$. 

Recalling the effective generator $\mc{L}_O$ defined in \eqref{eq:odlimit_generator}, while the semigroup $\exp(t \mc{L}_O)$ is always unital and contractive (\cref{lem:properlo}), it may fail to be completely positive and thus may not generate a proper QMS. When the complete positivity holds, we may interpret $\mc{L}_\gamma$ as a second-order lift of $\mc{L}_O$. 
The formal definition of \emph{lifts} is presented in \cref{def:2ndlift} below, constituting our primary focus. 
Recall that $\mc{F}(\mc{L}_S)$ is a finite-dimensional von Neumann algebra \cites{spohn1977algebraic,frigerio1978stationary}. We now introduce the following condition:

\begin{condition} \label{assump:overdp}
    The generator $\mc{L}_O$ induces a QMS on the von Neumann algebra $\mc{F}(\mc{L}_S)$ satisfying $\si_o$-KMS detailed balance, where $\si_o$ denotes a full-rank invariant state of $\exp(t \lo^\dag)$. Moreover, the state $\si_o \in \mc{D}_+(\mc{F}(\mc{L}_S))$ is compatible with $\si \in \mc{D}_+(\mc{M})$ in the sense that
    \begin{align} \label{eq:compabsio}
            \l X, \lo (Y)\r_{\si,1/2} = \l X, \lo (Y)\r_{\si_o,1/2}\,,\q \forall X, Y \in \mc{F}(\mc{L}_S)\,.
    \end{align}
\end{condition}

\begin{lemma} \label{lem:proplo}
    Under \cref{assump:overdp}, $\si_o$ is the reduced state of $\si$ on $\mc{F}(\mc{L}_S)$. Moreover, the fixed-point algebra $\mc{F}(\lo)$ of $\lo$ defined on $\mc{F}(\ld)$ is characterized by 
    \begin{align} \label{eq:kernello}
        \mc{F}(\lo) = \{X \in \mc{F}(\ld)\,;\ \lh (X) = 0\} = \mc{F}(\mc{L}_\gamma)\,,
    \end{align}
    where the operators $\ld$ and $\lh$ satisfy \cref{asspA}.
\end{lemma}

\begin{proof}
By the condition \eqref{eq:compabsio}, there holds $ \l X, \exp(t \lo) Y\r_{\si,1/2} = \l X, \exp(t \lo) Y\r_{\si_o,1/2}$. We let $Y = \mi$ and then have
    \begin{align*}
        \tr(\si X) = \tr(\si_o X)\,,\q \forall X \in \mc{F}(\mc{L}_S)\,,
    \end{align*}
implying that $\si_o$ is the reduced state of $\si$. By definition \eqref{eq:odlimit_generator}, we have $\lo (X) = 0$ if and only if $\lh \es (X) = 0$, which completes the proof by \cref{lem:symanti}.
\end{proof}

Building on terminology from classical Markov processes \cites{chen1999lifting,eberle2024non}, we define:

\begin{definition} \label{def:2ndlift}
For a QMS generator $\mc{L}$ of the form \eqref{eq:decom} satisfying \cref{asspA}, the QMS $\exp(t \mc{L})$ is called a \emph{second-order lift} of $\exp(t \lo)$ if it satisfies Conditions~\ref{asspB}, \ref{assump:PHP}, and \ref{assump:overdp}.
\end{definition}

\begin{remark}[Assumptions for lifting]  \label{rem:asspforlift}
 Conditions~\ref{asspA}--\ref{assump:PHP} arise naturally in the derivation of the overdamped limit and are essential for the convergence estimates of $\exp(t \mc{L}_\gamma)$ in \cref{sec:upperlift} below. While \cref{assump:overdp} may appear to be the most restrictive among Conditions \ref{asspA}--\ref{assump:overdp}, it is not strictly necessary for analyzing the convergence of $\exp(t \mc{L}_\gamma)$. In fact, one can define the QMS $\exp(t \mc{L}_\gamma)$ as a second-order lift of $\exp(t \lo)$ under Conditions~\ref{asspA}, \ref{asspB}, and \ref{assump:PHP} alone, and \cref{thm:lower,thm:uppermatrix} remain valid in this setting (though some minor adjustments are required for the corresponding statements).

However, as we have emphasized, a key motivation for the lifting framework is its ability to characterize the potential acceleration of convergence for the detailed balanced QMS via non-reversibility. In this context, \cref{assump:overdp} becomes indispensable, since $\exp(t \lo)$ needs to be a QMS \emph{a priori}. For ease of exposition, we have chosen to include \cref{assump:overdp} in \cref{def:2ndlift} so that we can discuss the convergence of $\exp(t \mc{L}_\gamma)$ and the acceleration of $\exp(t \lo)$ in a uniform manner. 
Finally, we remark that \cref{assump:overdp} can be simplified for bipartite quantum systems, as will be explored in \cref{sec:appbipart}.
\end{remark}

The general concept of lifting in Hilbert spaces (\cref{def:generallift}) will be introduced in \cref{sec:lifthilbert}, providing a unified framework for both lifted classical and quantum Markov semigroups. The conditions in \cref{def:2ndlift} are obviously stronger than those in \cref{def:generallift}.
We now illustrate how the current setup fits into the abstract framework in \cref{sec:lifthilbert}:
\begin{itemize}
    \item The dynamics $ \exp(t \mc{L}_\gamma)$ and $\exp(t \lo)$ induce contraction \( C_0 \)-semigroups \( P_t  \) on $ \mh = \mc{M}$ and \( P_{t, O} \) on \( \mhh = \mc{F}(\mc{L}_S) \). Here, $\mc{M}$ and $\mc{F}(\mc{L}_S)$ are equipped with the KMS inner product \( \l \dd, \dd \r_{\si,1/2} \), turning them into Hilbert spaces.
    \item  Thanks to \cref{asspA}, the condition \eqref{eq:phpgen} in \cref{def:generallift} reduces to \cref{assump:PHP}:
\begin{align*}
    \l X, \mc{L}_\gamma (Y) \r_{\si,1/2} = \l X, \lh (Y) \r_{\si,1/2} = 0\,, \quad \forall X, Y \in \mc{F}(\mc{L}_S)\,,
\end{align*}
namely, \(\es\lh\es= 0 \). Furthermore, the condition \eqref{eq:liftgen} follows from the definition of \( \lo \) in \eqref{eq:odlimit_generator}, where $\mf{S} = (- \ld)^{-1}$. Thus, \( \mc{L}_\gamma \) in \cref{def:2ndlift} is indeed a lifted QMS in the sense of \cref{def:generallift}.
\end{itemize}

\begin{remark} \label{rem:first_order}
    Though not our focus, the effective dynamics \eqref{eq:eff1st} under condition $\es \lh \es\neq 0$ is closely related to the \emph{first-order lifts} of Markov semigroups. Specifically, $\exp(t\mc{L}_\gamma)$ can be interpreted as a first-order lift of \eqref{eq:eff1st} (see \cref{rem:first-oderlift} and \cite{eberle2024non}*{Remark 12}). The terminology \emph{first-order} and \emph{second-order} reflect the fact that the overdamped limits \eqref{eq:eff1st} and \eqref{eq:eff2nd} originate from first-order and second-order perturbation expansions, respectively.
\end{remark}

\subsection{Lower and upper bounds of convergence rate}\label{sec:upperlift}
We have shown that under Conditions \ref{asspA}, \ref{asspB}, \ref{assump:PHP}, and \ref{assump:overdp}, the QMS $\exp(t \mc{L}_\gamma)$ is a lift of $\exp(t \lo)$ in the sense of \cref{def:2ndlift} (hence also \cref{def:generallift}), and the QMS $\exp(t \lo)$ is the overdamped limit of $\exp(t \mc{L}_\gamma)$ as $\gamma \to \infty$. An immediate observation is that such QMS $\exp(t \mc{L}_\gamma)$ is ergodic.

\begin{lemma}
    Let $\mc{L}_\gamma$ be a QMS generator of the form \eqref{eq:decom} satisfying Conditions~\ref{asspA} and~\ref{assump:PHP}. Then the semigroup $\exp(t \mc{L}_\gamma)$ is ergodic: 
    \begin{align} \label{eq:ergolgamma}
        \lim_{t \to \infty} \exp(t \mc{L}_\gamma) = E_{\mc{F}}\,,
    \end{align}
    where $E_{\mc{F}}$ is the conditional expectation onto $\mc{F}(\mc{L}_\gamma)$. 
\end{lemma}
\begin{proof}
    By \cref{lem:characnf}, it suffices to show that $\mc{L}_\gamma$ has no purely imaginary eigenvalue. To see this, suppose that there exists $X \neq 0$ such that $\mc{L}_\gamma X = i \mu X$ for some real $\mu \neq 0$. Then, $$\l  X, \ld (X) \r_{\si,1/2} = \Re \l X, \mc{L}_\gamma (X) \r_{\si,1/2} = \Re\, i \mu  \l X,  X \r_{\si,1/2} = 0 $$ 
    gives $\ld (X) = 0$ and hence $\lh (X) = i \mu X$. It follows that 
    \begin{align}\label{auxeqq1}
        \es \lh X = i \mu X\,,
    \end{align}
    which contradicts with $\es \lh \es = 0$, since $\mu \neq 0$ and $X \in \mc{F}(\ld)$ in \eqref{auxeqq1}.
\end{proof}

This allows us to apply the convergence results for general lifted semigroups in \cref{sec:lifthilbert} to estimate the convergence \eqref{eq:ergolgamma} of $\exp(t \mc{L}_\gamma)$, from which we will see how it depends on the $L^2$ convergence rate (governed by the spectral gap) of $\exp(t \lo)$ and the parameter $\gamma$.

We first give the upper bound of the convergence rate of $\exp(t \mc{L}_\gamma)$, which is a direct application of \cref{thm:lowerbound}, noting that \cref{assp:locoercive} naturally holds in the finite-dimensional setting. We denote by $\lad_O$ the spectral gap of $\lo$:
\begin{equation}\label{eq:spgaplado}
    \lad_O := \inf_{X \in \mc{F}(\mc{L}_S)\backslash\{0\}} \frac{- \l X, \mc{L}_O (X) \r_{\si,1/2}}{\norm{X - E_O(X) }^2_{\si,1/2}}\,,
\end{equation}
where $E_O$ is the conditional expectation from $\mc{M}$ to the fixed-point algebra $\mc{F}(\lo)$ of $\lo$.

\begin{theorem}[Upper bound]\label{thm:lower}
    Suppose that the QMS $\exp(t \mc{L}_\gamma)$ is a lift of $\exp(t \lo)$ in the sense of \cref{def:2ndlift}. Let $C \ge 1$ and $\nu > 0$ be the prefactor and the convergence rate in the estimate \eqref{ex:expdecay} for $\exp(t \mc{L}_\gamma)$, and $t_{rel}$ be the relaxation time \eqref{eq:trelax} of $\exp(t \mc{L}_\gamma)$. Then, 
     \begin{align} \label{eq:lowerboundest}
        \nu \le (1 + \log C) \sqrt{\w{s_{\rm m}}^{-1} \lad_O}\,,\quad
            t_{\rm rel} \ge  \frac{1}{2 \sqrt{\w{s_{\rm m}}^{-1} \lad_O}}\,,
    \end{align}
where $\w{s_{\rm m}}$ is the minimal nonzero eigenvalue of $\Pi_1 (- \ld)^{-1} \Pi_1$. Here $\Pi_1$ is the orthogonal projection to $\ran(\lh \es)$ for the KMS inner product. 
\end{theorem}

A self-contained proof of \cref{thm:lower} is provided in \cref{appA}, although it can also be derived from \cref{thm:lowerbound} directly. Theorem~\ref{thm:lower} shows that the $L^2$ convergence rate of QMS $\exp(t\mc{L}_\gamma)$ is bounded above by the square root of the spectral gap of $\mc{L}_O$. Notably, this upper bound \eqref{eq:lowerboundest} for $\nu$ is independent of $\gamma$.

We next present the lower bound estimate of $\nu$ by applying \cref{thm:lowerboundpt} with \cref{rem:constantprefactor}. Unlike the upper bound, the lower bound typically exhibits $\gamma$-dependence and can be maximized at a particular value of $\gamma$. 
We recall the spectral gap $\lad_S$ of $\ld$ as follows:
\begin{equation}\label{eq:spgaplads}
    \lad_S := \inf_{X \in \mc{M}\backslash\{0\}} \frac{- \l X, \mc{L}_S (X) \r_{\si,1/2}}{\norm{X - \es (X)}^2_{\si,1/2}}\,.
\end{equation}
We shall take the observation period $T = \lad_O^{-1/2}$ in \cref{thm:lowerboundpt} to obtain explicit expressions. This choice is motivated by its asymptotic optimality for the convergence rate estimate, as demonstrated in \cites{cao2023explicit,li2024quantum}. 


\begin{theorem}[Lower bound]\label{thm:uppermatrix}
    Let the QMS $\exp(t \mc{L}_\gamma)$ be a lift of $\exp(t \lo)$ in the sense of \cref{def:2ndlift}. The exponential convergence holds: for $X \in \mc{M}$ with $E_{\mc{F}}(X) = 0$, 
        \begin{align*}
            \norm{e^{t \mc{L}_\gamma} (X)}_{2,\si} \le C e^{-\nu t} \norm{X}_{2,\si}\,,
        \end{align*}
    with $ C = \mc{O}(1)$ and $\nu = \frac{\gamma}{\gamma^2 C_0 + C_1}$,
    where 
    \begin{align*}
        C_0 = \Theta\left(\lad_O^{-1} \right)\,, \q C_1 = \Theta\left(
         \left(\lad_S^{-1} + 1 + \lad_S^{-1/2}(K_1 + K_2 \lad_O^{-1/2}) \right)^2
           \right)\,.
    \end{align*}
   The spectral gaps $\lad_O$ and $\lad_S$ are given in \eqref{eq:spgaplado} and \eqref{eq:spgaplads}, respectively, and the constants $K_1$ and $K_2$ are such that 
   \begin{align}  \label{eq:constk1k2mm}
    \norm{(\id - \es)\lh^\dag (-\ld)^{-1} \lh Y}_{2,\si} 
  \le & \left( K_1 \norm{\lo (Y)}_{2,\si} + K_2 \sqrt{\mc{E}_{\lo}(Y)} \right)\,.
\end{align}
If $\gamma = \gamma_{\max} := \sqrt{C_1(T)/C_0(T)}$, we have the maximal convergence rate:
\begin{align} \label{eq:rateestmatrix}
    \max_{\gamma > 0} \, \nu(\mc{L}_\gamma) \ge \nu(\mc{L}_{\gamma_{\max}}) = \Omega\left(\frac{\lad_O^{1/2}}{ \lad_S^{-1} + 1 + \lad_S^{-1/2}(K_1 + K_2 \lad_O^{-1/2})}\right)\,.
\end{align}
\end{theorem}

\begin{remark}[Bounds for \( K_1 \) and \( K_2 \)]
    The inequality \eqref{eq:constk1k2mm} holds trivially with crude bounds 
\begin{equation} \label{eq:constk1k2m}
    K_1 = 0\,,\q K_2 = \norm{(\id - \es)\lh^{\dag} (-\ld)^{-1/2}}_{(2,\si) \to (2,\si)}\,,
\end{equation}
by $\norm{(-\ld)^{-1/2} \lh Y}_{2,\si} = \sqrt{\mc{E}_{\mc{L}_O}(Y)}$ and 
\begin{align*}
 \norm{(\id - \es)\lh^\dag (-\ld)^{-1} \lh Y}_{2,\si} & = \norm{(\id - \es)\lh^\dag (-\ld)^{-1/2} (-\ld)^{-1/2} \lh Y}_{2,\si} \\
 & \le \norm{(\id - \es)\lh^{\dag} (-\ld)^{-1/2}}_{(2,\si) \to (2,\si)} \norm{(-\ld)^{-1/2} \lh Y}_{2,\si}\,.
\end{align*}
Here, $\norm{\dd}_{(2,\si) \to (2,\si)}$ denotes the operator norm for a mapping on the Hilbert space $(\mc{M},\norm{\dd}_{2,\si})$. 
It is possible to obtain sharper estimates for $K_1, K_2$ and hence sharper rate estimates \eqref{eq:rateestmatrix} for bipartite quantum systems by assuming the intertwining properties of $\lo$ (see \cref{sec:appbipart}). 
\end{remark}

\begin{proof}[Proof of \cref{thm:uppermatrix}]
    It suffices to verify Assumptions \ref{assp:locoercive}, \ref{assmp:decomposition}, and \ref{assump:est1} required for \cref{thm:lowerboundpt}. First, \cref{assp:locoercive} follows from the finite dimensionality. Conditions~\ref{condd1} and~\ref{condd2} in \cref{assmp:decomposition} are by \cref{asspA} and \cref{rem:liftla}, while the condition \ref{condd3} is from \cref{eq:kernello,eq:kerlamma} with \cref{asspB}. For \cref{assump:est1}, the density conditions \ref{asspmm0} and \ref{asspmm1} are trivial in the finite-dimensional case.  We next estimate constants $K_0$ and $K_3$
    in
    inequalities \eqref{asspc0} and \eqref{asspc2}:
    \begin{align} \label{eq:constk0k3m}
        K_0 = K_3 = 1\,.
    \end{align}
    Indeed, letting $X \in \mc{M}$ and $Y, Z \in \mc{F}(\mc{L}_S)$, 
    for \eqref{asspc0}, we have 
    \begin{equation} \label{asspc0m}
        \begin{aligned}
            -  \l  \ld X, (- \ld)^{-1} \lh Y \r_{\si, 1/2} 
             = &    \l  (-\ld)^{1/2} X, (- \ld)^{-1/2} \lh Y \r_{\si, 1/2}  \\ 
            \le & \sqrt{\mc{E}_{\ld}(X) \mc{E}_{\lo}(Y)}\,.
        \end{aligned}
   \end{equation} 
   For \eqref{asspc2}, we have 
   \begin{equation} \label{asspc2m}
    \begin{aligned}
        - \l \es X - X, \lh Z \r_{\si,1/2}  \le &  \l (- \ld)^{1/2} (X - \es X), (- \ld)^{-1/2}\lh Z \r_{\si,1/2}  \\ 
        \le & \sqrt{\mc{E}_{\ld}(X) \mc{E}_{\lo}(Z)}\,.
    \end{aligned}
   \end{equation}
   The inequality \eqref{asspc1} is a simple consequence of \eqref{eq:constk1k2mm}. Moreover, by \cref{rem:constantprefactor}, the prefactor $C$ is of order one since $K_0 = K_3 = 1$. 
   The proof is complete. 
\end{proof}

\medskip 

\noindent \textbf{Accelerated mixing and 
optimal lifts.} As shown in \cref{thm:lower,thm:uppermatrix}, the lifting framework provides convergence guarantees for the hypocoercive QMS $\exp(t \mc{L}_\gamma)$. From the perspective of accelerating convergence to equilibrium, this framework also identifies both the fundamental limits and the potential opportunities for speeding up the convergence of a given detailed balanced QMS $\exp(t \mc{L}_O)$ through non-reversible lifts to extended state spaces.

Specifically, \cref{thm:lower} shows that $\nu(\mc{L}_\gamma) = \mc{O}(\sqrt{\lad_O})$ for all $\gamma > 0$, indicating that second-order lifts of KMS-detailed balanced QMS can achieve, at best, a quadratic speedup. Therefore, it is natural to define dynamics with convergence rates that reach the upper bound $\mc{O}(\sqrt{\lad_O})$ as the optimal lifts of $\exp(t \mc{L}_O)$.


\begin{definition}[Optimal Lift] \label{def:optimallift}
   Let $\mc{L}_O$ be an ergodic QMS generator on a finite-dimensional von Neumann algebra that satisfies $\si_o$-KMS detailed balance for a full-rank state $\si_o$. Let $\mc{L}$ denote its second-order lift as defined in \cref{def:2ndlift}. We say that $\mc{L}$ is an \emph{optimal (second-order) lift} of $\mc{L}_O$ if the $L^2$ convergence rate of $\mc{L}$ satisfies
    \begin{align*}
        \nu(\mc{L}) = \Theta(\sqrt{\lad_O})\,,
    \end{align*}
    where $\lad_O$ is the spectral gap of $\lo$. 
\end{definition}

In applications, the minimal and maximal non-zero eigenvalues of $-\mc{L}_S$ are typically of order one. For instance, this is the case when $\ld$ is a quantum depolarizing semigroup. In this case, \cref{thm:uppermatrix} implies that at $\gamma = \gamma_{\max} > 0$, 
\begin{align} \label{eq:rateestmatrix2}
    \nu(\mc{L}_{\gamma_{\max}}) = \Omega\left(\frac{\sqrt{\lad_O}}{  1 + K_1 + K_2 \lad_O^{-1/2}}\right)\,,
\end{align}
and then, $\mc{L}_{\gamma_{\max}}$ is an optimal lift if there holds
\begin{align} \label{eq:keyest}
    K_1 + K_2 \lad_O^{-1/2} = \mc{O}(1)\,.
\end{align}
Examples of such optimal lifts shall be provided in \cref{sec:appbipart}. 

\begin{remark}[Collapsed generator $\lo$] \label{rem:modifylo}
  It is worth noting that, due to the general lifting framework in \cref{sec:lifthilbert}, the definition \eqref{eq:odlimit_generator} of $\lo$ can be modified to, as in \eqref{eq:liftgen},
    \begin{align} \label{def:generallo}
        \mc{L}_O = - (\lh \es)^{\star} \mc{S} \lh \es\,,     
    \end{align}
   for any positive bounded operator $\mathcal{S}$. While this modification breaks the connection with the overdamped limit, all the main results in this section (\cref{def:2ndlift}, \cref{thm:lower,thm:uppermatrix}) remain valid without any significant changes.
\end{remark}

\begin{remark}[Comparison with \cite{li2024quantum}]  
Similar convergence results were obtained for primitive hypocoercive QMS in the recent work \cites{li2024quantum} by the same authors under Conditions~\ref{asspA}, \ref{asspB}, and~\ref{assump:PHP}. Specifically, \cite{li2024quantum}*{Corollary 4.4} proves that for a primitive $\mc{L}_\gamma = \lh + \gamma \ld$ satisfying \cref{asspA}, $\dim \ker(\ld) > 1$, and $\es \lh \es = 0$, the maximal $L^2$ convergence rate (achieved at a particular  $\gamma_{\max}$) admits the lower bound:
\begin{align} \label{eq:estprvious}  
  \nu(\mc{L}_{\gamma_{\max}}) = \Omega\left(\frac{\sqrt{\lad_S}s_A^{2}}{\big( s_A + \norm{(\id - \es)\lh^\dag}_{(2,\si) \to (2,\si)}\big) \sqrt{\| \ld \|_{(2,\si) \to (2,\si)}}}\right)\,,  
\end{align}  
where $s_A$ is the singular value gap of the operator $\lh \es$.  

In this work, our analysis, on the one hand, extends to non-primitive hypocoercive generators $\mc{L}_\gamma$, and also on the other hand, involves the spectral gap $\lad_O$ of the effective generator
$\mc{L}_O = - (\lh \es)^{\star} (- \ld)^{-1} \lh \es$ (rather than $s_A$) to connect with the overdamped limit of $\mc{L}_\gamma$ and to quantify convergence acceleration via lifting. 
Additionally, in the case where $\lad_S = \Theta(1) = \| \ld \|_{(2,\si) \to (2,\si)}$, we find that $s_A^2 = \Theta(\lad_O)$, and the estimate in \eqref{eq:estprvious} simplifies to  
\begin{align*}  
  \nu(\mc{L}_{\gamma_{\max}}) = \Omega\left(\frac{\lad_O}{ \sqrt{\lad_O} + \norm{(\id - \es)\lh^\dag}_{(2,\si) \to (2,\si)}}\right)\,,  
\end{align*}  
which is consistent with \cref{eq:rateestmatrix2}, incorporating $K_1$ and $K_2$ from the trivial bound 
\eqref{eq:constk1k2m}. 
\end{remark}

\section{Lifting in Hilbert spaces}  \label{sec:lifthilbert}

In this section, we introduce a general Hilbert space framework with \emph{minimal structural assumptions} to analyze the convergence of lifted semigroups that inherently exhibit hypocoercivity and, meanwhile, to characterize the acceleration of mixing through lifting.
This framework unifies the results of \cites{eberle2024non,brigati2024hypocoercivity,eberle2025convergence} for classical dynamics and the analysis of quantum dynamics in \cref{sec:liftingmatrix}. The following presentation is parallel to that in \cref{sec:liftingmatrix} and \cite{eberle2025convergence}, with the primary technical challenges stemming from the infinite dimensionality of the underlying Hilbert space and the unbounded nature of the operators involved.

We first fix some notations. Let $\mc{H}$ be a Hilbert space with inner product $\langle \cdot, \cdot \rangle_{\mc{H}}$ and induced norm $\|\cdot\|_{\mc{H}}$. Unless otherwise specified, all operators considered below are assumed to be unbounded. For an unbounded operator $(T, \dom(T))$ with domain $\dom(T) \subset \mc{H}$, we will often write simply $T$ when the domain is clear from context. The adjoint of $T$ is denoted by $T^*$.
Given two Hilbert spaces $\mc{H}$ and $\widetilde{\mc{H}}$, the operator norm of a bounded linear operator $T: \mc{H} \to \widetilde{\mc{H}}$ is denoted by $\norm{T}_{\mc{H} \to \w{\mc{H}}} = \sup_{x \in \mh \backslash\{0\}} \norm{T x}_{\w{\mc{H}}}/\norm{x}_{\mc{H}}$. 

\subsection{General framework of lifting} 

Let $\{P_t\}_{t \ge 0}$ be a contraction strongly continuous semigroup on $\mc{H}$. Specifically, $P_t$ satisfies the properties: (i) $P_0 = \id$; (ii) $P_s P_t = P_{s + t}$ for $s, t \ge 0$; (iii) $\lim_{t \to 0^+}\norm{P_t x - x}_{\mh} = 0$ for all $x \in \mh$; (iv) $\norm{P_t}_{\mh \to \mh} \le 1$, where conditions (i)-(iii) define a $C_0$-semigroup and condition (iv) imposes contractivity. The generator $\mf{L}$ of $\{P_t\}_{t \geq 0}$ is a closed, densely defined operator given by
\begin{align*}
    \mf{L}x := \lim_{t \to 0^+} \frac{P_t x - x}{t},
\end{align*}
with $\dom(\mf{L})$ consisting of all $x \in \mh$ for which this limit exists in the norm topology of $\mh$. The following result is standard, and follows from Hille-Yosida and Lumer-Phillips theorems \cite{engel2000one}.

\begin{lemma} \label{lem:dissipative}
Let $P_t$ be a $C_0$-semigroup with generator $(\fl,\dom(\fl))$. It holds that:
\begin{itemize}
    \item $x_t := P_t x_0$ for $x_0 \in \dom(\fl)$ is continuous differentiable on $[0, \infty)$ and satisfies $\dot{x}_t = \fl x_t$ with initial condition $x_t|_{t = 0} = x_0$, and $x_t \in \dom(\fl)$ for all $t \ge 0$, that is, $$P_t (\dom (\fl)) \subset \dom(\fl)\,.$$  
    \item If $P_t$ is contractive,  we have  $ \sigma(\mf{L}) \subset \{\lad \in \C\,;\ \Re \lad \le 0\}$ and $\norm{(\lad - \mf{L})^{-1}}_{\mh} \le (\Re \lad)^{-1}$ for $\lad \in \C$ with $\Re \lad > 0$. Moreover, $\fl$ is dissipative: 
    \begin{align*}
    \Re \l x, \fl x \r_{\mh} \le 0\,,\q \forall x \in \dom(\fl)\,.
    \end{align*}
\end{itemize}
\end{lemma}

The equilibrium subspace of the semigroup $\{P_t\}_{t \geq 0}$ is the closed subspace
\begin{align} \label{eq:mixsubspace}
    \mc{F} := \{x \in \mh \mid P_t x = x \text{ for all } t \geq 0\} = \ker(\mf{L}).
\end{align}
Let $P_\infty$ denote the orthogonal projection from $\mh$ onto $\mc{F}$. 
\begin{definition} \label{def:hypocoercivesemi}
    We say that $\{P_t\}_{t \geq 0}$ \emph{converges to equilibrium} if for each $x \in \mh$, there exists a decay function $r: \mathbb{R}_+ \to \mathbb{R}_+$ with $\lim_{t \to \infty} r(t) = 0$ such that
\begin{align} \label{eq:exponentialconverge}
    \|P_t x - P_\infty x\|_{\mh} \leq r(t) \|x - P_\infty x\|_{\mh}.
\end{align}
When $r(t) = Ce^{- \nu t}$ for some $C >  1, \nu > 0$, we say the semigroup is \emph{hypocoercive}. In the case of $r(t) = e^{-\nu t}$ (i.e., $C = 1$), it is called \emph{coercive}. 
\end{definition}

\begin{remark}
The coercivity and hypocoercivity of $P_t$ can also be characterized via its generator $\mf{L}$; see \cite{villani2009hypocoercivity}. A standard argument shows that $P_t$ is coercive with $r(t) = e^{-\nu t}$ if and only if 
\begin{align} \label{eq:coercive}
    - \Re \l x, \mf{L} x \r_{\mh} \ge \nu \norm{x - P_\infty x}^2_{\mh}\,, \q \forall x \in \dom (\fl)\,.
\end{align}
It implies that the symmetric part of $\fl$ has a spectral gap $\nu$.     
\end{remark}

We are interested in the degenerate (hypocoercive) case where the inequality \eqref{eq:coercive} only holds with $\nu = 0$, and yet $P_t$ still exhibits exponential convergence $r(t) = Ce^{- \nu t}$ in
\eqref{eq:exponentialconverge}. In this case, it necessarily holds $C > 1$.

Analogous to \cref{lem:singulargap}, the convergence rate of a hypocoercive semigroup can be bounded above by the singular value gap of its generator.  
To formalize this, observe that: 
\begin{lemma}\label{lem:kernelf}
$\mc{F}^\perp$ is a closed $\{P_t\}_{t \ge 0}$-invariant subspace of $\mc{H}$.     
\end{lemma}
Restricting the semigroup to $\mc{F}^\perp$, we obtain the subspace semigroup $P_t: \mc{F}^\perp \to \mc{F}^\perp$, which remains contractive and strongly continuous. Its generator is the restriction $\mf{L}|_{\mc{F}^\perp}$ with domain $\dom(\mf{L}|_{\mc{F}^\perp}) = \dom(\mf{L}) \cap \mc{F}^\perp$. 
 We define $|\fl| := \sqrt{\fl^* \fl}$ via the polar decomposition of $\fl$, which is positive, self-adjoint, and satisfies $\ker(\mf{L}) = \ker(|\mf{L}|) = \mc{F}$ and $\dom(\mf{L}) = \dom(|\mf{L}|)$ \cite{weidmann2012linear}. The \emph{singular value gap} of $\fl$ is given by the spectral gap of $|\fl|$: 
\begin{align}
    s(\fl) & = \inf\{\l x, |\fl| x \r_{\mc{H}}\,;\ x \in \dom(|\fl|) \cap \mc{F}^\perp,\, \norm{x}_{\mc{H}} = 1\} \notag \\ &= \inf\{\norm{\fl x}_{\mc{H}} \,;\ x \in \dom(\fl) \cap \mc{F}^\perp,\, \norm{x}_{\mc{H}} = 1\}\,. \label{eq:singulargap}
\end{align}

\begin{lemma} \label{lem:singulargapgeneral}
    Let $P_t$ be a hypocoercive $C_0$-semigroup with generator $\fl$. Then for the decay function $r(t) = C e^{-\nu t}$ in \eqref{eq:exponentialconverge}, there holds 
    \begin{align*}
        \nu \le (1 + \log C) s(\fl)\,.
    \end{align*}
\end{lemma}

\begin{proof}
 Note that the exponential decay $ \|P_t x - P_\infty x\|_{\mh} \leq C e^{-\nu t} \|x - P_\infty x\|_{\mh}$ is equivalent to: for some $\nu > 0$ and $T = (\log C)/ \nu$ with $C \ge 1$, 
\begin{equation}  \label{eq:delayedexponential}
    \norm{P_t x - P_\infty x}_{\mc{H}} \le  e^{- \nu (t - T)} \norm{x - P_\infty x}_{\mc{H}}\,.
\end{equation}
We now estimate, for any $T,\nu > 0$ in \eqref{eq:delayedexponential}, and $x \in \mc{F}^\perp$, 
    \begin{align*}
        \int_0^\infty \norm{P_t x}_{\mh} \ud t & = \int_0^T \norm{P_t x}_{\mh} \ud t  +  \int_T^\infty \norm{P_t x}_{\mh} \ud t \\
        & \le \int_0^T \norm{x}_{\mh} \ud t  +  \int_T^\infty e^{- \nu (t - T)} \norm{x}_{\mh} \ud t \\
        & \le \left(\frac{1}{\nu} + T \right) \norm{x}_{\mh}\,. 
    \end{align*}
    Therefore, the operator $ \int_0^\infty P_t x \ud t$ for $x \in \mc{F}^\perp$ is well-defined. It follows that $0$ is in the resolvent set of the generator of the contraction (subspace) semigroup $P_t: \mc{F}^\perp \to \mc{F}^\perp$, and we can define the bounded inverse operator $(- \fl|_{\mc{F}^\perp})^{-1}$ on $\mc{F}^\perp$ that satisfies 
    \begin{equation*}
        \norm{(- \fl|_{\mc{F}^\perp})^{-1} x}_{\mh} \le  \int_0^\infty \norm{P_t x}_{\mh} \ud t \le \left(\frac{1}{\nu} + T \right) \norm{x}_{\mh}\,, 
    \end{equation*}
    which, by $T = (\log C)/ \nu$ and \eqref{eq:singulargap}, implies the desired estimate: 
    \begin{align*}
      \frac{1}{s(\fl)} = \sup_{y \in \dom(\fl|_{\mc{F}^\perp}) \backslash\{0\}} \frac{ \norm{y}_{\mh}}{\norm{\fl|_{\mc{F}^\perp} y}_{\mh}}  = \sup_{x \in \mc{F}^\perp \backslash\{0\}} \frac{ \norm{(- \fl|_{\mc{F}^\perp})^{-1} x}_{\mh}}{\norm{x}_{\mh}} &\le \frac{1}{\nu}(1 + \log C)\,. \qedhere
    \end{align*}
\end{proof}

\medskip
\noindent \textbf{Lifting in Hilbert spaces.}
We now formalize the concept of lifting in abstract Hilbert spaces. Specifically, under certain structural conditions (in \cref{def:generallift}), a hypocoercive $C_0$-semigroup $P_t$ can be interpreted as a \emph{lifted dynamic} of a symmetric coercive $C_0$-semigroup $P_{t,O}$ acting on a subspace. 
This perspective offers a unified framework for examining both the convergence to equilibrium of $P_t$ and the limit in accelerating the convergence of $P_{t, O}$ through non-self-adjointness (i.e., non-reversibility for Markov semigroups). See \cref{sec:example} for various examples.  

For simplicity, let the Hilbert space $\mh_O$ be a closed strict subspace of $\mh$ with the induced inner product $\l x, y\r_{\mh_O} = \l x,y\r_{\mh}$ for $x,y \in \mh_O$. In general, one could consider $\mh_O$ as a Hilbert space that is continuously embedded in $\mh$. Let $\{P_{t,O}\}_{t \ge 0}$ be a symmetric contraction $C_0$-semigroup with the negative self-adjoint generator $(\fll, \dom(\fll))$ satisfying: 

\begin{assumption}[Coercivity]\label{assp:locoercive}
The semigroup $\{P_{t,O}\}_{t \ge 0}$ on $\mhh$ is coercive with rate $\lad_O > 0$:
\begin{align*}
    \|P_{t,O} x\|_{\mhh} \leq e^{-\lad_O t} \|x \|_{\mhh}\,,
\end{align*}
for $x \in \ker(\fll)^{\perp_{\mhh}}$, where $\perp_{\mhh}$ denotes the orthogonal complement in $\mhh$. Moreover, we assume that the generator $(\fll,\dom(\fll))$ has a purely discrete spectrum on $\mhh$. 
\end{assumption}

\begin{remark}
  The discrete spectrum assumption serves to simplify the proof of the abstract divergence lemma (\cref{divergencelemma}), by ensuring that $\fll$ admits a spectral decomposition into \emph{finite-dimensional} eigenspaces (thus, the equilibrium subspace $\ker(\fll)$ is also finite-dimensional). All examples considered in this work satisfy this assumption. We also point out that this assumption can be relaxed using the techniques developed in \cite{brigati2023construct}*{Section 6}.
\end{remark}

Without loss of generality, we let $\lad_O$ in \cref{assp:locoercive} be the optimal rate characterized by the spectral gap of $\fll$:   
\begin{equation} \label{eq:spectralgaplo}
\lad_O = \inf \{- \l x, \fll x \r_{\mhh}\,;\ x \in \dom(\fll) \cap \ker(\fll)^{\perp_{\mhh}},\, \norm{x}_{\mhh} = 1 \}\,.
\end{equation}
Since $\mhh$ is closed  in $\mh$ with the induced topology, we have 
\begin{align} \label{eq:spectralgaplo2}
    \dom(\fll) \cap \ker(\fll)^{\perp_{\mhh}} = \dom(\fll) \cap \ker(\fll)^{\perp}\,,
\end{align}
due to $\ker(\fll)^{\perp_{\mhh}}
= \ker(\fll)^{\perp} \cap \mhh$. 

The following definition extends the concept of the second-order lift of classical Markov processes initiated by Eberle et al. \cites{eberle2024non,eberle2025convergence} and the discussion of finite-dimensional quantum dynamics in \cref{sec:liftingmatrix} to the abstract Hilbert space setting. This would enable new applications in infinite-dimensional quantum systems.

\begin{definition}[Lifting] \label{def:generallift}
Let $P_t$ and $P_{t,O}$ be contraction $C_0$-semigroups on $\mh$ and $\mhh$ with generators $\fl$ and $\fll$, respectively.
 $P_t$ is a second-order \emph{lifted semigroup} of $P_{t,O}$ if
    \begin{enumerate}[label=(\roman*)]
     \item \label{condg1}
        The domain relation holds
        \begin{align*}
            \dom(\fll) \subset \dom(\fl)\,.
        \end{align*}
        \item \label{condg2}  There holds, for any $x \in \mhh$, $y \in \dom(\fl) \cap \mhh$, 
        \begin{align} \label{eq:phpgen}
            \l x, \fl y \r_{\mh} = 0 \,. \tag{L1}
        \end{align}
        \item \label{condg3} There exists a positive  bounded operator $\mf{S}: \mhh^\perp \to  \mhh^\perp$ such that
        \begin{align} \label{eq:liftgen}
            \l \fl x,  \mf{S} \fl y \r_{\mh} = - \l x, \fll y \r_{\mhh}\,,\q  \forall x \in \dom(\fl) \cap \mhh,\, y \in \dom(\fll)\,. \tag{L2}
        \end{align}
        \end{enumerate}
    In this case, we also say that $P_{t,O}$ is a \emph{collapsed semigroup} of $P_t$. 
\end{definition}

\begin{remark} \label{rem:approx}
   The condition \eqref{eq:phpgen} gives $\fl y \in \mhh^\perp$, and hence $\mf{S}\fl y$ in \eqref{eq:liftgen} is well-defined. 
\end{remark}

For later use, we denote the operator norm (i.e., maximal eigenvalue) of $\mf{S}$ by
\begin{align} \label{eq:eigss}
    s_{\rm M}:= \sup_{x \in \mhh^\perp \backslash \{0\}} \frac{\norm{\mf{S}x}_{\mhh}}{\norm{x}_{\mhh}}\,.
\end{align}
The operator $\mf{S}$ does not appear in the earlier works on lifting \cites{eberle2024non,eberle2025convergence}, where it corresponds to $\mf{S} = {\rm id}$ (see \cref{sec:liftclassical}). Although the choice of $\mf{S}$ may not significantly affect the convergence properties of $\fll$ if it is bounded from above and below, we include the additional operator $\mf{S}$ in our framework for two main reasons: (i) to establish a clear connection between the lifting and the overdamped limit (as detailed in \cref{sec:liftingmatrix}), and (ii) to potentially extend the applicability of the lifting framework to a wider class of interesting examples.

\begin{remark}[First-order lifting]\label{rem:first-oderlift}
    One could define the \emph{first-order lifting} in our framework by modifying conditions \eqref{eq:phpgen} and \eqref{eq:liftgen} to instead require:
\begin{align}\label{eq:first_order_condition}
    \langle x, \fl y \rangle_{\mh} = - \langle x, \fll y \rangle_{\mhh}, \quad \forall x,y \in \dom(\fll)\,.
\end{align}
This structure proves useful in the context of lifting discrete-time finite-state Markov chains \cite{chen1999lifting}; see \cite{eberle2024non}*{Remark 12} for the detailed discussion. 
\end{remark}

Since this work only focuses on the second-order lifting, we often simply say that $P_t$ is a lifted semigroup of $P_{t, O}$ when conditions in \cref{def:generallift} hold.

We next establish quantitative \emph{upper and lower bounds} on the $L^2$ convergence rate $\nu(\fl)$ of the hypocoercive $C_0$-semigroup $P_t = \exp(t\fl)$. As discussed in \cref{sec:upperlift} for quantum Markov processes, these estimates also characterize the potential acceleration achievable through lifting $P_{t, O}$ to the extended space.

\medskip
\noindent \textbf{Upper bound of $\nu(\fl)$.} An immediate implication of lifting in \cref{def:generallift} and \cref{lem:singulargapgeneral} is that the convergence rate of the semigroup $P_t$ is at most $\mathcal{O}(\sqrt{\lambda_O})$ (\cref{thm:lowerbound}), where $\lambda_O$ is the spectral gap of its collapsed semigroup \eqref{eq:spectralgaplo}. In other words, the lifting framework can provide at most a quadratic speedup for a given symmetric coercive semigroup, particularly when $\lambda_O$ is small (corresponding to slow convergence).

Before stating \cref{thm:lowerbound}, a generalization to \cref{thm:lower}, we recall that $- \fll$ is a positive operator. When $\fl$ is a lift of $\fll$, the condition \eqref{eq:liftgen} implies that for $x \in \dom(\fll)$ satisfying
\begin{equation*}
    \langle \fl x, \mf{S} \fl x \rangle_{\mh} = - \langle x, \fll x \rangle_{\mhh} = 0\,,
\end{equation*}
we must have both $x \in \ker(\fll)$ and $x \in \ker(\fl)$. This implies the following relation between the equilibrium subspace of the semigroups $P_{t,O}$ and its lifting $P_t$:  
\begin{align} \label{rela:ker}
     \ker(\fll) = \ker(\fl) \cap \dom(\fll)\,.
\end{align}


    \begin{theorem}[Upper bound] \label{thm:lowerbound}
        Let $P_t$ be a hypocoercive $C_0$-semigroup on $\mh$ with the decay function $r(t) = C e^{-\nu t}$ in \eqref{eq:exponentialconverge}, and $P_{t,O}$ be a symmetric coercive $C_0$-semigroup on $\mhh$ that satisfies \cref{assp:locoercive}. Suppose that $P_t$ is a lift of $P_{t, O}$ and 
        the constant $\w{s_{\rm m}} := \inf_{x \in \mhh^\perp \backslash \{0\}} \frac{\norm{\Pi_1 \mf{S} \Pi_1 x}_{\mhh}}{\norm{x}_{\mhh}}$ is positive, where $\Pi_1$ is the projection from $\mhh^\perp$ to $\overline{\ran(\fl|_{\mhh})}$. 
        Then, there holds
        \begin{align} \label{eq:upperrate}
            \nu \le (1 + \log C) \sqrt{\w{s_{\rm m}}^{-1}\lad_O}\,.
        \end{align}
\end{theorem}

\begin{remark}
   By the proof below, establishing the upper bound in \eqref{eq:upperrate} only requires
   \begin{align*}
       \l \fl x,  \mf{S} \fl x \r_{\mh} \le - \l x, \fll x \r_{\mhh}\,, \q \text{for any $x \in \dom(\fll)$,}
   \end{align*}
 and $\w{s_{\rm m}} > 0$. However, as we shall see in \cref{thm:lowerboundpt}, for the lower bound of the convergence rate of $P_t$, both conditions \eqref{eq:phpgen} and \eqref{eq:liftgen} in \cref{def:generallift} are necessary. 
\end{remark}

\begin{proof}
By \cref{lem:singulargapgeneral}, it suffices to estimate $s(\fl)$ by the spectral gap of $\fll$. We note 
\begin{equation} \label{eq:lowerp1gen}
        \begin{aligned}
            s(\fl) & = \inf\{\norm{\fl x}_{\mc{H}} \,;\ x \in \dom(\fl) \cap \mc{F}^\perp,\, \norm{x}_{\mc{H}} = 1\} \\
            & \le \inf\{\norm{\fl x}_{\mc{H}} \,;\ x \in \dom(\fll) \cap \mc{F}^\perp,\, \norm{x}_{\mc{H}} = 1\}\,,
        \end{aligned}
\end{equation}
by $\dom(\fll) \subset \dom(\fl)$.
Then we can write, for $x \in \dom(\fll)$
\begin{align*}
    \w{\mf{S}}^{-1/2} \w{\mf{S}}^{1/2} \fl x = \fl x\,,
\end{align*}
with $\w{\mf{S}} = \Pi_1 \mf{S} \Pi_1$, and derive using condition \eqref{eq:liftgen},  
    \begin{equation}\label{eq:lowerp2gen}
        \begin{aligned}
          \norm{\fl x}_{\mh} & = \norm{\w{\mf{S}}^{-1/2} \w{\mf{S}}^{1/2} \fl x}_{\mh} \\
          & \le \w{s_{\rm m}}^{-1/2} \norm{\w{\mf{S}}^{1/2} \fl x}_{\mh} = - \w{s_{\rm m}}^{-1/2} \l x, \fll x\r_{\mhh}^{1/2}\,.
        \end{aligned}
    \end{equation}
    Combining \eqref{eq:lowerp1gen} and \eqref{eq:lowerp2gen} implies 
    \begin{align*}
        s(\fl) &\le \inf\{\norm{\fl x}_{\mc{H}} \,;\ x \in \dom(\fll) \cap \mc{F}^\perp,\, \norm{x}_{\mc{H}} = 1\} \\
        & \le \w{s_{\rm m}}^{-1/2} \inf\{-  \l x, \fll x\r_{\mhh}^{1/2} \,;\ x \in \dom(\fll) \cap \mc{F}^\perp,\, \norm{x}_{\mc{H}} = 1\} \\
        & = \w{s_{\rm m}}^{-1/2} \inf\{-  \l x, \fll x\r_{\mhh}^{1/2} \,;\ x \in \dom(\fll) \cap \ker(\fll)^\perp,\, \norm{x}_{\mc{H}} = 1\}\,,
    \end{align*}
which, along with \eqref{eq:spectralgaplo}-\eqref{eq:spectralgaplo2}, gives $ s(\fl) \le \sqrt{\lad_O/\w{s_{\rm m}}}$. Here we have used, by \eqref{rela:ker}, 
\begin{align*}
    \dom(\fll) \cap \mc{F}^\perp & = \dom(\fll) \cap (\dom(\fll) \cap \ker(\fl))^\perp  = \dom(\fll) \cap \ker(\fll)^\perp \,.
\end{align*}
The proof is complete. 
\end{proof}

\noindent \textbf{Lower bound of $\nu(\fl)$.} Obtaining quantitative lower bounds for the convergence rate of $P_t$ requires more involved analysis and additional structural conditions (see Assumptions~\ref{assmp:decomposition} and~\ref{assump:est1}). Our approach builds on the variational framework with space-time Poincaré inequalities, first introduced by Albritton et al.~\cites{albritton2019variational} and subsequently extended to various hypocoercive models \cites{cao2023explicit,brigati2023construct,brigati2024explicit,li2024quantum}. Recent works by Eberle et al.~\cites{eberle2024non,eberle2025convergence} have streamlined this approach through the lens of second-order lifts; see also \cite{brigati2024hypocoercivity}. 

We first assume that $P_t$ admits a symmetric-anti-symmetric decomposition with certain properties (\cref{assmp:decomposition}). This class of generators yields hypocoercive semigroups $P_t$ (\cref{rem:hypocoerpt}), whose convergence properties will be analyzed.

\begin{assumption}[Decomposition] \label{assmp:decomposition}  
Let $\mf{P}$ be the orthogonal projection from $\mh$ to $\mhh$. Suppose that $(\ls,\dom(\ls))$ and $(\la,\dom(\la))$ are closed and densely defined operators on $\mh$ such that  
    \begin{align} \label{domainconst}
      \dom(\fl) \subset \dom(\ls)\cap \dom(\la)\,, 
      \tag{B0}
    \end{align}
    and on $\dom(\fl)$, 
    \begin{align*}
        \fl  = \la  + \gamma \ls \,,\q \gamma > 0\,.
    \end{align*}
   We assume
    \begin{enumerate} [label=(\roman*)]
        \item \label{condd1}  $\ls$ is symmetric and satisfies 
        \begin{equation} \label{eq:kernells}
            \mf{P} \ls x = 0\,,\q \text{for $x \in \dom(\fl)$.} \tag{B1}
        \end{equation}
        Moreover, $- \ls$ is coercive on $(\id - \mf{P}) (\mh)$: for some $\lad_S > 0$,
        \begin{align} \label{eq:lscoercive}
            \lad_S \norm{x - \mf{P} x}_{\mh}^2 \le - \l x, \ls x \r_{\mh}\,, \q \forall x \in \dom(\fl)\,.  \tag{B2}
        \end{align} 
        \item \label{condd2}  $\la$ is a lift of $\fll$, and it is anti-symmetric on $\dom(\fll)$:
        \begin{align} \label{eq:weakanti}
           \la^* x = - \la x\,, \q \forall x \in \dom(\fll)\,. \tag{B3}
        \end{align}
        \item \label{condd3} There holds 
        \begin{align} \label{eq:kernelconst}
            \ker(\fll) = \ker(\fl) \subset \mhh\,.   \tag{B4}   
        \end{align}
        Moreover, the equilibria $\mc{F} = \ker(\fl)$ of $P_t$ in \eqref{eq:mixsubspace} form a strict subspace of $\mhh$. 
    \end{enumerate}
\end{assumption}

\begin{remark}\label{rem:liftla}
Several remarks regarding conditions \ref{condd1}--\ref{condd3} in \cref{assmp:decomposition} are worth noting. For \ref{condd1}, it is easy to see that \eqref{eq:kernells} and \eqref{eq:lscoercive} imply
\begin{align} \label{kerls}
    \mhh \cap \dom(\ls) \subset \ker(\ls)\,, \q \dom(\fl) \cap \ker(\ls) \subset \mc{H}_0\,,
\end{align}
respectively, which places constraints on the kernel of $\ls$. For \ref{condd2}, under mild domain assumptions (e.g., $\mhh \subset \dom(\ls)$ and $\dom(\fl) = \dom(\la) \cap \dom(\ls)$), it can be verified that the lifting conditions \eqref{eq:phpgen} and \eqref{eq:liftgen} hold for $\la$ if and only if they hold for $\fl$. This implies that $\la$ is a lift of $\fll$ if and only if $\fl$ is a lift of $\fll$. For \ref{condd3}, for primitive (classical or quantum) Markov processes, $\ker(\fll) = \ker(\fl)$ is one-dimensional, and thus condition \ref{condd3} becomes trivial in this case. 

In particular, in the finite-dimensional setting, the relation \eqref{kerls} reduces to $\ker(\ls) = \mhh$, and then $\la$ is a lift of $\fll$ if and only if $\fl$ is a lift of $\fll$. It follows from \eqref{rela:ker} that 
\begin{equation*}
    \ker(\fll) = \ker(\fl) \cap \mhh = \ker(\fl)\,,    
\end{equation*}
namely, \eqref{eq:kernelconst} holds, as $\ker(\fl) \subset \mhh = \ker(\ls)$ by \cref{lem:kernelrela} below.
\end{remark}

The following lemma, extending the relation \eqref{eq:kerlamma} of quantum dynamics (see \cref{lem:symanti}), characterizes the equilibria of $P_t$ in terms of $\ker(\la)$ and $\ker(\ls)$, which implies that the equilibrium subspace $\mc{F}$ of $P_t$ is independent of $\gamma$. 

\begin{lemma} \label{lem:kernelrela}
    Let $P_t$ be a hypocoercive $C_0$-semigroups on $\mh$ satisfying \cref{assmp:decomposition}. Then, $x \in \dom (\fl)$ is an equilibrium, i.e., $\fl x = 0$, if and only if $\ls x = 0$ and $\la x = 0$. 
\end{lemma}

\begin{proof}
    If $\fl x = 0$, then $\Re \l x, (\la + \gamma \ls) x \r_{\mh} = 0$, which implies $$0 \le - \gamma \l x, \ls x \r_{\mh} = \Re \l x, \la x \r_{\mh} \le 0\,,$$
    by the dissipativity of $\la$ and $\ls$, 
    and thus $\ls x = \la x = 0$. The converse direction is obvious. 
\end{proof}

\begin{remark}\label{rem:hypocoerpt}
Similarly to \cref{lem:hypoqms}, the second part of condition \ref{condd3} in \cref{assmp:decomposition} is necessary for guaranteeing the hypocoercivity of $P_t$. Indeed, by \cref{lem:kernelrela}, we have $\mc{F} = \ker(\fl) \subset \ker(\ls) \cap \ker(\la)$. Recall \eqref{eq:lscoercive}:
for $x \in \dom(\fl)$, 
\begin{align} \label{eq:hypocoercive11}
    - \Re \l x, \mf{L} x \r_{\mh} \ge - \l x, \ls x \r_{\mh} \ge \lad_S \norm{x - \mf{P} x}_{\mh}^2\,.
\end{align} 
If $\mc{F}$ is not a strict subspace of $\mc{H}_O$, namely, $\mc{F} = \mc{H}_O$, it follows from \eqref{eq:hypocoercive11} that the Poincar\'e inequality \eqref{eq:coercive} holds: $- \Re \l x, \mf{L} x \r_{\mh} \ge \lad_S \norm{x - P_\infty x}^2_{\mh}$. 
Thus, we have the purely exponential convergence of $P_t$ with rate $\lad_S$, which contradicts the hypocoercivity. 
\end{remark}

To motivate the space-time Poincaré inequality, we carefully examine the coercivity condition \eqref{eq:coercive}. This inequality seeks to control $\|x\|_{\mh}^2$ by $- \Re \langle x, \fl x \rangle_{\mh}$ for $x \in \mc{F}^\perp \cap \dom(\fl)$. 
Using the orthogonal projection $\mf{P}$, we decompose the norm:
\begin{align} \label{eq:flowdecomposition}
    \|x\|_{\mh}^2 = \|\mf{P} x\|_{\mh}^2 + \|x - \mf{P} x\|_{\mh}^2\,.
\end{align}
The coercivity assumption \eqref{eq:lscoercive} immediately bounds the second term as in \eqref{eq:hypocoercive11}. However, estimating $\|\mf{P} x\|_{\mh}^2$ presents a fundamental difficulty: for $\mf{P} x \in \dom(\fll) \cap \dom(\fl)$, we have 
\begin{align*}
    -\Re \langle \mf{P} x, \fl \mf{P} x \rangle_{\mh} &= -\gamma \langle \mf{P} x, \ls \mf{P} x \rangle_{\mh} - \Re \langle \mf{P} x, \la \mf{P} x \rangle_{\mh} = 0\,,
\end{align*}
by \eqref{eq:kernells} and \eqref{eq:weakanti}, revealing the degeneracy of the dissipation of $P_t$ on $\mhh$. This also suggests why condition \ref{condd2} imposes anti-symmetry of $\la$ only on the dense subspace $\dom(\fll) \subset \mhh$: 
the coercivity of $\fl$ on $\mhh^\perp$ is already ensured by \eqref{eq:lscoercive}, making the specific properties of $\la$ on this subspace non-essential for the overall analysis.

In \cref{thm:flowpoincare}, we establish a flow Poincaré inequality inspired by \cite{eberle2025convergence}.
It is a simplified version of the space-time Poincaré inequalities employed in \cites{albritton2019variational,cao2023explicit,li2024quantum}\footnote{The main difference between the space-time Poincar\'e and flow Poincar\'e inequalities is that the former one holds for any $x_t \in L^2([0,T];\mh)$ while the latter one is only for $x_t = P_t x_0$.}. This inequality takes the standard form of Poincar\'e inequality but augmented with a time variable: there exists $\alpha_T > 0$ depending on $T > 0$ such that for any $x_0 \in \dom(\fl) \cap \mc{F}^\perp$,
\begin{align}\label{eq:flow_poincare}
    \alpha_T \int_0^T \|x_t\|_{\mc{H}}^2 \ud t \leq - \int_0^T \langle x_t, \ls x_t \rangle_{\mh} \ud t\,,  \quad \text{where } x_t := P_t x_0.
\end{align}
From a physical perspective, this captures an important fact that while the hypocoercive dynamics $P_t$ exhibits the degenerate dissipation on the state space at any fixed time, the coercivity could be restored when considering its time-averaged behavior.

We begin by defining the Bochner space $L^2([0,T];\mh)$ with the normalized inner product:
\begin{equation}\label{eq:time_avg_innerprod}
    \langle x_t, y_t \rangle_{T, \mh} := \frac{1}{T} \int_0^T \langle x_t, y_t \rangle_{\mh} \, \ud t \,,
\end{equation}
and its closed subspace $L^2([0,T];\mhh)$ consisting of curves valued in $\mhh$. For a closed, densely defined symmetric operator $\mf{A}$ on $\mh$, we introduce the sesquilinear form:
\begin{align*}
    \mc{E}_{\mf{A}}(x,y) := - \l x, \mf{A} y \r_{\mh}\,,\q \text{for $x \in \mh$, $y \in \dom(\mf{A})$\,,}
\end{align*}
and its time-augmented version: 
\begin{align*}
    \mc{E}_{T, \mf{A}}(x_t, y_t) := \int_0^T   \mc{E}_{\mf{A}}(x_t, y_t) \ud t\,, \q \text{for $x_t \in L^2([0,T]; \mh)$, $y_t \in L^2([0,T];\dom(\mf{A}))$\,.}
\end{align*}
If $x = y$, we simply write $\mc{E}_{\mf{A}}(x) = \mc{E}_{\mf{A}}(x,x)$ and $\mc{E}_{T, \mf{A}}(x_t) = \mc{E}_{T, \mf{A}}(x_t, x_t)$. 

\begin{theorem}[Flow Poincar\'e inequality] \label{thm:flowpoincare}
Let $P_t$ be a hypocoercive $C_0$-semigroup on $\mh$ with the equilibrium subspace $\mc{F}$ as defined in \eqref{eq:mixsubspace}. Suppose that Assumptions \ref{assp:locoercive}, \ref{assmp:decomposition}, and \ref{assump:est1} hold. Then, for any time horizon $T > 0$, and any $x_0 \in \mc{F}^\perp \cap \dom(\fl)$, there holds: 
\begin{align} \label{eq:tspoincare}
  \alpha_T  \norm{x_t}_{T, \mc{H}}^2 \le \mc{E}_{T, \ls}(x_t)\,, \q \alpha_T := \frac{1}{\gamma^2 C_0(T) + C_1(T)}\,, 
\end{align}
where $x_t: = P_t x_0$, and 
\begin{align*}
    & C_0(T): = 2 c_3^2 K_0^2\,, \\  & C_1(T): = 2 \left(s_{\rm M}^{1/2} \lad_S^{-1/2} c_4  + K_3 c_2  +  \lad_S^{-1/2} (K_1 c_1 + K_2 c_3) \right)^2 + \lad_S^{-1}\,,
\end{align*}
where the constants $\{c_i\}_{i = 1,2,3,4}$ are given in \cref{divergencelemma}, $s_{\rm M}$ is defined in \eqref{eq:eigss}, and constants $\{K_i\}_{i = 0,1,2,3}$ are given in \eqref{asspc0}--\eqref{asspc2} from \cref{assump:est1}.
\end{theorem}

\cref{assump:est1} is primarily technical and will be stated in \cref{subsec:flowpoincare} below for ease of exposition. Rather than providing general sufficient conditions for its verification, we find it more convenient to verify this assumption on a case-by-case basis in applications.

The proof of the flow Poincaré inequality builds on a decomposition analogous to \eqref{eq:flowdecomposition}, with the principal challenge remaining the control of the projected flow $\|\mf{P}x_t\|_{T,\mh}$ in the subspace $\mhh$. This requires a full use of the lifting structure (see \cref{lem:spacetime}) and technical a priori estimates for solutions to an abstract divergence equation (see \cref{divergencelemma}).  To maintain the flow of our presentation, we postpone the proof of \cref{thm:flowpoincare} to \cref{subsec:flowpoincare}. 
Once the flow Poincar\'e inequality is established, a standard Gr\"onwall-type argument applied to the time-averaged energy functional yields exponential decay in the $L^2([0,T];\mh)$ norm.

\begin{theorem}[Lower bound] \label{thm:lowerboundpt}
Under the same assumptions as in \cref{thm:flowpoincare}, it holds that for any observation period $T > 0$, and the initial condition $x_0 \in \mc{F}^\perp$, 
\begin{align} \label{eq:expconverge}
    \frac{1}{T}\int_t^{t + T} \norm{P_s x_0}_{\mh}^2 \ud s \le e^{- 2\nu t} \norm{x_0}_{\mh}^2\,,
\end{align}
with the convergence rate: 
\begin{equation} \label{eq:ratel2}
    \nu = \frac{\gamma}{\gamma^2 C_0(T) + C_1(T)}\,,
\end{equation}
where $C_0(T)$ and $C_1(T)$ are the flow Poincar\'e constants in \cref{thm:flowpoincare}. 
Moreover, the convergence rate estimate \eqref{eq:ratel2} is maximized at 
\begin{align} \label{eq:convergmax}
    \gamma_{\rm max} = \sqrt{C_1(T)/C_0(T)} \q \text{with rate}\q \nu_{\max} = \frac{1}{2\sqrt{C_0(T) C_1(T)}}\,.
\end{align}
\end{theorem}

\begin{corollary}\label{rem:constantprefactor}
    Under the same assumptions as in \cref{thm:flowpoincare}, it holds that for any $T > 0$ and $x_0 \in \mc{F}^\perp$, 
    \begin{align} \label{eq:expcoverct}
    \norm{P_t x_0}_{\mh} \le C e^{-\nu t} \norm{x_0}_{\mh}\,,
\end{align}
where $C = e^{\nu T}$, and $\nu$ satisfies \eqref{eq:ratel2}. Furthermore, for any $\gamma > 0$ and $T > 0$, we have 
\begin{equation*}
    C = e^{\nu T} = \mc{O}(1)\,,
\end{equation*}
if $K_0 K_3 = \Theta(1)$ holds, where  $K_0, K_3$ are the constants from \eqref{asspc0}--\eqref{asspc2}. 
\end{corollary}
\begin{proof}
    The $L^2$ convergence \eqref{eq:expcoverct} follows from the time-averaged one 
\eqref{eq:expconverge} and 
\begin{align*}
   \norm{P_{t + T} x_0}_{\mh}^2 &\le \frac{1}{T}\int_t^{t + T} \norm{P_s x_0}_{\mh}^2 \ud s\,. 
\end{align*}
To see $C = e^{\nu T} = \mc{O}(1)$, we compute 
    \begin{align*}
        C \leq \exp(\nu_{\max} T) \leq \exp\left( \frac{T}{4 c_2 c_3 K_0 K_3} \right) = \mc{O}(1),
    \end{align*}
    where the second inequality is by $\sqrt{C_0(T)} = \sqrt{2} c_3 K_0$ and $\sqrt{C_1(T)} \ge \sqrt{2} c_2 K_3$, and the third one is by $T/c_3 = \mc{O}(1)$ and $c_2 = \Theta(1)$ from \eqref{eq:diverest} in \cref{divergencelemma}.  
\end{proof}

\begin{proof}[Proof of \cref{thm:lowerboundpt}]
Noting $P_t (\dom(\fl)) \subset \dom(\fl)$, and $\dom(\fl) \subset \dom(\ls)\cap \dom(\la)$ by \eqref{domainconst}, as well as the dissipativity of $\la$ by \cref{lem:dissipative}, we can compute   
    \begin{align*}
        \frac{\rd}{\rd t} \int_t^{t +T} \norm{P_s x_0}_{\mh}^2 \ud s & = \norm{P_{t + T} x_0}_{\mh}^2 - \norm{P_t x_0}_{\mh}^2 \\
        & = \int_t^{t + T} \frac{\rd}{\rd s} \norm{P_s x_0}_{\mh}^2 \ud s \\
        & \le 2 \gamma \int_t^{t + T} \l P_s x_0, \ls P_s x_0 \r_{\mh}\,.
    \end{align*}
By the flow Poincar\'e inequality in \cref{thm:flowpoincare}, we find 
\begin{align*}
    \frac{\rd}{\rd t} \int_t^{t +T} \norm{P_s x_0}_{\mh}^2 \ud s \le - \frac{2 \gamma}{\gamma^2 C_0(T) + C_1(T)} \int_t^{t + T} \norm{P_s x_0}_{\mh}^2 \ud s\,,
\end{align*}
which implies \eqref{eq:expconverge} by Gr\"onwall's inequality and $\norm{P_t x_0}_{\mh} \le \norm{x_0}_{\mh}$. 
Then maximizing the function $\gamma/(\gamma^2 C_0(T) + C_1(T))$ in $\gamma$ gives \eqref{eq:convergmax}. 
\end{proof}

\smallskip

\noindent \textbf{Discussion on optimal lifts.} Similarly to the discussion at the end of \cref{sec:upperlift}, under the assumptions required for \cref{thm:lowerbound,thm:lowerboundpt}, at the particular $\gamma_{\max}$ and $T = \Theta(\lad_O^{-1/2})$, we obtain the $L^2$ rate estimation for $P_t = \exp(t \fl)$ as  
\begin{align} \label{eq:estlgamma}  
  \Omega\left( \frac{\sqrt{\lad_O}}{K_0 \left(s_{\rm M}^{1/2} \lad_S^{-1/2}  + K_3  +  \lad_S^{-1/2} (K_1  + K_2 \lad_O^{-1/2}) + \lad_S^{-1/2} \right)} \right) \le \nu(\mc{L}_{\gamma_{\max}}) \le \mc{O}\left(\sqrt{\frac{\lad_O}{\w{s_{\rm m}}}}\right)\,,  
\end{align}  
with the help of \eqref{eq:diverest} in \cref{divergencelemma}.  
If the upper and lower bounds in \eqref{eq:estlgamma} match asymptotically, we can say that $\fl$ is an \emph{optimal lift} of $\fll$. Suppose the spectra of the operators $\mf{S}$ and $\ls$ are both of order one. It follows that $\w{s_{\rm m}} = s_{\rm M} = \Theta(1)$, $\lad_S = \Theta(1)$ and $K_0 = K_3 = \Theta(1)$ by \cref{rem:estki}. In this case, the estimation \eqref{eq:estlgamma} simplifies to  
\begin{align} \label{eq:estlgammasim}  
  \Omega\left( \frac{\sqrt{\lad_O}}{1  +   K_1  + K_2 \lad_O^{-1/2}} \right) \le \nu(\mc{L}_{\gamma_{\max}}) \le \mc{O}(\sqrt{\lad_O})\,.  
\end{align}  
Thus, as in \eqref{eq:keyest}, if one can prove that $K_1 + K_2 \lad_O^{-1/2} = \mc{O}(1)$, the optimal lift follows.  

\subsection{Flow Poincar\'e inequality} \label{subsec:flowpoincare}
The main aim of this section is to prove \cref{thm:flowpoincare}. We start with an auxiliary lemma, in analogy with \cite{eberle2024non}*{Lemma 18} and
\cite{eberle2025convergence}*{Lemma 16}. 
We define the operator $(\mc{A},\dom(\mc{A}))$ on $L^2([0,T];\mh)$:
\begin{align} \label{def:operaotral2}
    \mc{A} x_t := - \p_t x_t + \la x_t\,,
\end{align}
with domain
\begin{multline*}
    \dom(\mc{A}) := \big\{ x_t \in H^1([0,T];\mh)\,; \ 
    \text{$x_t \in \dom(\la)$  for a.e. $t \in [0,T]$ and $\la x_t \in L^2([0,T]; \mh)$}
\big\}\,.    
\end{multline*}

\begin{lemma} \label{lem:spacetime}
    Suppose $\la$ is a lift of $\fll$. Then, for any 
    $x_t, y_t, z_t \in L^2([0,T]; \mhh)$ such that $x_t \in \dom(\mc{A})$, $y_t \in \dom(\fll)$ for a.e. $t\in [0,T]$, and $\fll y_t \in L^2([0,T]; \mhh)$, 
    there holds
    \begin{align} \label{eq:spacetime}
        \l \mc{A} x_t, z_t + \mf{S} \la y_t \r_{T, \mh} = - \l \p_t x_t, z_t \r_{T,\mhh} + \mc{E}_{T, \fll}(x_t, y_t)\,.
    \end{align}    
    \end{lemma}
    
    \begin{proof}
    Recalling the lifting conditions in \cref{def:generallift}, we find
    \begin{align*}
        \l \p_t x_t,  \mf{S} \la y_t\r_{\mh} = 0 = \l \la x_t, z_t\r_{\mh}\,,\q \l \la x_t, \mf{S} \la y_t\r_{\mh} = - \l x_t, \fll y_t \r_{\mh}\,,
    \end{align*}
      for a.e. $t \in [0,T]$. This readily gives \eqref{eq:spacetime}. 
    \end{proof}



We next introduce the abstract divergence lemma, extended from \cite{eberle2024space}*{Theorem 5} and 
\cite{eberle2025convergence}*{Lemma 15}; see also 
\cites{eberle2024non,li2024quantum}. In the PDE setting (i.e., $\fl$ is a diffusion process generator), the divergence lemma is equivalent to Lions' Lemma \cites{brigati2023construct,brigati2024explicit}. Let $D_O$ be a dense subspace of $\dom(\fll)$  with respect to $\norm{\dd}_{\mhh}$. We define the Sobolev-type spaces $H^1_{\fll}$ and $H^2_{\fll}$  by the closure of $\{x_t \in C^\infty([0,\infty]; D_O)\,;\ x_0 = x_1 = 0\}$ 
under the norm 
\begin{align*}
    \norm{x}_{1,\fll} := \norm{x}_{T, \mhh} + \mc{E}_{T, \fll}(x_t)^{1/2} + \norm{\p_t x_t}_{T, \mhh}\,,
\end{align*}
and the norm 
\begin{align*}
    \norm{x}_{2,\fll} := \norm{x}_{1,\fll} + \norm{\fll x}_{T, \mhh} + \norm{\p_{tt} x}_{T, \mhh} + \mc{E}_{T,\fll}(\p_t x)^{1/2}\,,
\end{align*}
respectively. We also introduce the subspace of $L^2([0,T];\mhh)$: 
\begin{multline*}
     L^2_\perp([0,T];\mhh)  = \bigl\{x_t \in L^2([0,T];\mhh)\,; \l w_t, x_t\r = 0 \\ \text{for any $w_t = w_0$ a.e. $t$ with $w_0 \in \ker(\fll)$}\bigr\}.     
\end{multline*}

\begin{lemma} \label{divergencelemma}
Let $P_{t,O} = \exp(t \fll)$ be a symmetric contraction $C_0$-semigroup satisfying \cref{assp:locoercive}. For any $T > 0$ and $x_t \in L^2_\perp([0,T]; \mhh)$, there exists a solution $(z_t, y_t) \in H^1_{\fll} \times H^2_{\fll}$ to the abstract divergence equation:
    \begin{align} \label{diverequation}
        \p_t z_t - \fll y_t = x_t  
    \end{align}
    that satisfies the estimates:
\begin{align}  \label{eq:diver1}
    \norm{\fll y_t}_{T, \mhh} \le c_1 \norm{x_t}_{T, \mhh}\,, \quad  \sqrt{\mc{E}_{T, \fll}(z_t)} \le c_2 \norm{x_t}_{T, \mhh}\,,
\end{align}
and 
\begin{align} \label{eq:diver2}
     \sqrt{\mc{E}_{T, \fll}(y_t)} \le c_3 \norm{x_t}_{T,\mhh}\,, \quad  \sqrt{\mc{E}_{T, \fll}(\p_t y_t)} \le c_4 \norm{x_t}_{T,\mhh}\,, 
\end{align}
where 
\begin{align} \label{eq:diverest}
    c_1 = c_2 = \Theta(1)\,,\q c_3 = \Theta\left(T + \frac{1}{\sqrt{\lad_O}} \right)\,, \q c_4 = \Theta\left(1 + \frac{1}{T \sqrt{\lad_O}}\right)\,.
\end{align}
\end{lemma}

The constants involved in the asymptotic notation $\Theta(\dd)$ in \eqref{eq:diverest} could be made explicit.
The proof of \cref{divergencelemma} is technical, and we defer it to \cref{app:diverlema} for ease of exposition. 

\medskip 

We now state the technical assumption for \cref{thm:flowpoincare} and then prove the general flow Poincar\'e inequality.  
\begin{assumption} \label{assump:est1}
Define the linear subspace 
\begin{align*}
    \mc{C}_T:= \{x_t = P_t x_0\,;\ x_0 \in \mc{F}^\perp \cap \dom(\fl)\,,\, 0 \le  t \le T \} \subset L^2([0,T]; \mh)\,. 
\end{align*} 
There exists a dense subspace $C_d$ of $\mc{F}^\perp \cap \dom(\fl)$ with respect to $\norm{\dd
}_{\mc{H}}$, with $\mc{C}_{d,T}$ defined by $\mc{C}_{d,T}:= \{x_t = P_t x_0\,;\ x_0 \in C_d \,,\, 0 \le  t \le T \}$\footnote{By the contraction of $P_t$ and density of $C_d$, the space $\mc{C}_{d,T}$ is also dense in $\mc{C}_T$ for the norm $\norm{x_t}_{T,\mh}$}, such that the following holds
\begin{enumerate} [label=(c\arabic*),start=0]
    \item \label{asspmm0} $x_t, \mf{P}x_t \in \dom (\mc{A})$ for $x_t \in \mc{C}_{d,T}$.  
    \item  \label{asspmm1}$\mc{C}_{d,T}$ is dense in $\mc{C}_T$ with respect to the norm $\norm{x_t}_{T,\mh} + \mc{E}_{T,\ls}(x_t)$.
\end{enumerate}
Moreover, there exist $K_i > 0$, $i = 0, 1,2,3$ such that for $x \in \dom(\fl)$, $y, z \in \dom(\fll)$,  
\begin{equation} \label{asspc0}
     -  \l \ls x, \mf{S}\la y \r_{\mh} \le  K_0 \sqrt{\mc{E}_{\ls}(x) \mc{E}_{\fll}(y)}\,, \tag{C0}
\end{equation} 
\begin{align} \label{asspc1}
     \l \la (\mf{P} x - x), \mf{S} \la y \r_{\mh} 
   \le & \left( K_1 \norm{\fll y}_{\mhh} + K_2 \sqrt{\mc{E}_{\fll}(y)} \right) \norm{\mf{P} x - x}_{\mh}\,, \tag{C1}
\end{align}
and 
\begin{align} \label{asspc2}
    - \l \mf{P} x - x, \la z \r_{\mh} 
   \le & K_3 \sqrt{\mc{E}_{\ls}(x) \mc{E}_{\fll}(z)}\,. \tag{C2}
\end{align}
\end{assumption}

\begin{remark} \label{rem:estki}
   In applications, the inequalities \eqref{asspc0}--\eqref{asspc2} can typically be verified as follows (here we will discuss on a formal level for simplicity). 
For \eqref{asspc0}, we compute  
\begin{align*}  
    -\l \ls x, \mf{S} \la y \r_{\mh} &= \l \mf{S}^{-1/2} (-\ls) x, \mf{S}^{1/2} \la y \r_{\mh} \\  
    &\le \norm{\mf{S}^{-1/2}(-\ls)^{1/2}}_{\mh \to \mh} \sqrt{\mc{E}_{\ls}(x) \mc{E}_{\fll}(y)}\,,
\end{align*}  
from the Cauchy-Schwarz inequality. As in \eqref{eq:constk1k2m}, the inequality \eqref{asspc1} reduces to establishing:
\begin{align}   \label{eq:simpleassp3}
    \norm{(\id - \mf{P}) \la^* \mf{S} \la y}_{\mh} \le K_1 \norm{\fll y}_{\mhh} + K_2 \sqrt{\mc{E}_{\fll}(y)}\,.
\end{align}   
For \eqref{asspc2}, we similarly estimate  
\begin{align*}  
    -\l \mf{P} x - x, \la z \r_{\mh} &= \l \mf{S}^{-1/2}(-\ls)^{-1/2} (-\ls)^{1/2}( x- \mf{P} x), \mf{S}^{1/2} \la z \r_{\mh} \\  
    &\le \norm{\mf{S}^{-1/2}(-\ls)^{-1/2}}_{\mh \to \mh} \sqrt{\mc{E}_{\ls}(x) \mc{E}_{\fll}(z)}\,. 
\end{align*}
\end{remark}

\begin{proof}[Proof of \cref{thm:flowpoincare}]
By the density assumption \ref{asspmm1}, it suffices to consider $x_0 \in C_d$ and prove \eqref{eq:tspoincare}. 
By the orthogonal projection $\mf{P}$, we write, for $x_t = P_t x_0$, 
\begin{align} \label{eq:orthodecomp}
    \norm{x_t}_{T,\mh}^2 = \norm{\mf{P} x_t}_{T,\mh}^2 + \norm{(\id - \mf{P})x_t}_{T,\mh}^2\,.
\end{align}
Note that $\mc{F}^\perp \cap \dom(\fl)$ is $P_t$-invariant, and then $\mf{P}x_t \in \mc{F}^\perp \cap \mhh$ holds by condition~\ref{condd3} in \cref{assmp:decomposition}, implying 
 $\mf{P}x_t \in L^2_\perp([0,T];\mhh)$. 
Let $(z_t, y_t)$ be the solution to $\mf{P}x_t = \p_t z_t - \fll y_t$ with the estimates given in \cref{divergencelemma}. We compute 
\begin{align}
    \norm{\mf{P}x_t}_{T, \mhh}^2 & = \l \mf{P} x_t, \p_t z_t - \fll y_t \r_{T,\mhh} \notag \\
    & =  \mc{E}_{T, \fll}(\mf{P} x_t,y_t) - \l \p_t \mf{P} x_t, z_t \r_{T,\mhh} \notag \\
    & = \l \mc{A}(\mf{P} x_t), z_t + \mf{S} \la y_t \r_{T, \mhh} \notag \\
    & = \underbrace{\l \mc{A} x_t, z_t + \mf{S}\la y_t \r_{T, \mh}}_{({\rm I})}  +  \underbrace{\l \mc{A}(\mf{P} x_t - x_t), z_t + \mf{S}\la y_t \r_{T, \mh}}_{(\rm II)} \,, \label{eq:decompositionpx}
\end{align}
where the second equality uses the integration by parts with $z_0 = z_T = 0$, and the third one is by \cref{lem:spacetime} and the assumption \ref{asspmm0}.

We next estimate the terms $({\rm I})$ and $({\rm II})$ in \eqref{eq:decompositionpx}, respectively. 
We start with the first term. By noting $\p_t x_t = \fl x_t = (\la + \gamma \ls)x_t$, we have 
\begin{align*}
    \mc{A} x_t = - \gamma \ls x_t\,,
\end{align*} 
and then, by assumption \eqref{eq:kernells}, 
\begin{align*}
    \l \mc{A} x_t, z_t \r_{T, \mh} = \l - \gamma \ls x_t, z_t \r_{T, \mh} = 0\,.
\end{align*}
Therefore, we derive, by \eqref{asspc0},
\begin{align} \label{auxeqi}
    {({\rm I})} & = - \gamma  \l \ls x_t, \mf{S}\la y_t \r_{T, \mh} 
     \le \gamma K_0 \sqrt{\mc{E}_{T,\ls}(x_t) \mc{E}_{T, \fll}(y_t)}\,.
\end{align}
Then, using the a priori estimate \eqref{eq:diver2} in \cref{divergencelemma}, we find 
\begin{align*}
    {({\rm I})} \le \gamma c_3 K_0 \sqrt{\mc{E}_{T,\ls}(x_t)} \norm{\mf{P} x_t}_{T, \mhh}\,.
\end{align*}

We proceed to estimate the term $({\rm II})$ in \eqref{eq:decompositionpx}, which is further splitted into two parts: 
\begin{align}  \label{auxeqii}
 ({\rm II}) = \underbrace{\l -\p_t(\mf{P} x_t - x_t), z_t + \mf{S}\la y_t \r_{T, \mh}}_{(\rm III)}  + \underbrace{\l \la (\mf{P} x_t - x_t), z_t + \mf{S}\la y_t \r_{T, \mh}}_{(\rm IV)}\,,
\end{align}
by $\mc{A} = - \p_t + \la$. For term $({\rm III})$, by integration by parts with $z_0,z_T, y_0, y_T = 0$, $\mf{P} x_t - x_t \perp \p_t z_t$, 
and Cauchy's inequality, we have 
\begin{align*}
    ({\rm III}) & =  \l  \mf{P} x_t - x_t, \p_t z_t +  \mf{S} \la \p_t y_t \r_{T, \mh} \\
     & =  \l  \mf{S}^{1/2}  (\mf{P} x_t - x_t), \mf{S}^{1/2} \la \p_t y_t \r_{T, \mh} \\
     & \le s_{\rm M}^{1/2} \sqrt{\mc{E}_{T,\fll}(\p_t y_t)} \norm{\mf{P} x_t - x_t}_{T, \mh}\,,
\end{align*}
where the constant $s_{\rm M}$ is given in \eqref{eq:eigss}. Then by \eqref{eq:diver2} in \cref{divergencelemma}, and the coercivity \eqref{eq:lscoercive} in \cref{assmp:decomposition}, it follows that 
\begin{align}  \label{auxeqiii}
    ({\rm III}) &\le c_4 s_{\rm M}^{1/2} \norm{\mf{P} x_t}_{T, \mhh} \norm{\mf{P} x_t - x_t}_{T, \mh} \notag \\
& \le  c_4 s_{\rm M}^{1/2} \lad_S^{-1/2} \norm{\mf{P} x_t}_{T, \mhh} \sqrt{\mc{E}_{T, \ls}(x_t)}\,.
\end{align}

For the term $({\rm IV})$, we have, by the (partial) anti-symmetry of $\la$ in \eqref{eq:weakanti}, 
\begin{align} \label{auxeqiv}
    ({\rm IV}) 
   = & - \l \mf{P} x_t - x_t, \la z_t \r_{T, \mh} + \l \la (\mf{P} x_t - x_t), \mf{S} \la y_t \r_{T, \mh}\,.
\end{align}
The first term is directly estimated by \eqref{asspc2}, 
while the second one is estimated as follows: by assumptions \eqref{asspc1} and \eqref{eq:lscoercive}, 
\begin{align*}
    & \l \la (\mf{P} x_t - x_t), \mf{S} \la y_t \r_{T, \mh} \\
   \le & \left( K_1 \norm{\fll y_t}_{T,\mhh} + K_2 \sqrt{\mc{E}_{T, \fll}(y_t)} \right) \norm{\mf{P} x_t - x_t}_{T, \mh}  \\
   \le & \lad_S^{-1/2} \sqrt{\mc{E}_{T,\ls}(x_t)}\left( K_1 \norm{\fll y_t}_{T,\mhh} + K_2 \sqrt{\mc{E}_{T, \fll}(y_t)} \right).
\end{align*}
Then, thanks to \cref{divergencelemma}, we have from \eqref{auxeqiv}:
\begin{align}  \label{auxeqv}
    ({\rm IV})  
    \le & \sqrt{\mc{E}_{T,\ls}(x_t)}  \left(K_3 \sqrt{\mc{E}_{T,\fll}(z_t)} + \lad_S^{-1/2} \left(K_1 \norm{\fll y_t}_{T,\mhh} + K_2 \sqrt{\mc{E}_{T,\fll}(y_t)}\right) \right) \notag \\
    \le &  \sqrt{\mc{E}_{T,\ls}(x_t)}  \left(K_3 c_2 + \lad_S^{-1/2} \left(K_1 c_1 + K_2 c_3\right) \right) \norm{\mf{P} x_t}_{T, \mhh}\,.
\end{align}

Therefore, collecting \eqref{auxeqi},  \eqref{auxeqii},  \eqref{auxeqiii}, and  \eqref{auxeqv}, we conclude
\begin{multline*}
    \norm{\mf{P} x_t}_{T, \mhh}^2 \le \gamma c_3 K_0 \sqrt{\mc{E}_{T,\ls}(x_t)} \norm{\mf{P} x_t}_{T, \mhh} + c_4 s_{\rm M}^{1/2} \lad_S^{-1/2} \norm{\mf{P} x_t}_{T, \mhh} \sqrt{\mc{E}_{T,\ls}(x_t)} \\
   + \sqrt{\mc{E}_{T,\ls}(x_t)}  \left(K_3 c_2 + \lad_S^{-1/2} \left(K_1 c_1 + K_2 c_3\right) \right) \norm{\mf{P} x_t}_{T, \mhh}\,,
\end{multline*}
and hence  
\begin{equation}  \label{auxeqpx}
    \norm{\mf{P} x_t}_{T, \mhh} \le \left(\gamma c_3 K_0  +  s_{\rm M}^{1/2} \lad_S^{-1/2} c_4  + K_3 c_2 + \lad_S^{-1/2} \left(K_1 c_1 + K_2 c_3 \right)\right)  \sqrt{\mc{E}_{T,\ls}(x_t)}\,.
\end{equation}
The proof is complete by \eqref{eq:orthodecomp}, \eqref{auxeqpx}, and \eqref{eq:lscoercive}. 
\end{proof}

\section{Examples and Applications} \label{sec:example}

In this section, we investigate optimal lifts for various detailed balanced classical and quantum Markov processes, achieving a square root reduction in the $L^2$ relaxation time.

\subsection{Lifting reversible diffusion processes} \label{sec:liftclassical}
We first apply the lifting framework introduced in \cref{sec:lifthilbert} to classical reversible diffusion processes, recovering the results presented in the recent works by Eberle et al.~\cites{eberle2024non,eberle2025convergence}. In what follows, for a given probability $\nu$ on a measurable space $\mc{X}$, we define the $\nu$-induced inner product for functions $f, g$ on $\mc{X}$ as
\begin{equation*}
\langle f, g \rangle_{L^2(\nu)} := \int_{\mc{X}} \overline{f(x)} g(x) \, \mathrm{d}\nu(x)\,.
\end{equation*}
The completion of this inner product space yields the Hilbert space denoted by $L^2(\mc{X},\nu)$.

Following the notation in \cite{eberle2025convergence}, we consider Markov diffusion processes $P_t$ and $\hat{P}_t$ with invariant measures $\mu$ and $\hat{\mu}$, defined on the state spaces $\mc{S}$ and $\hat{\mc{S}} = \mc{S} \times \mc{V}$, respectively, where $P_t$ is reversible. 
These processes are associated with generators $(L, \dom(L))$ and $(\hat{L}, \dom(\hat{L}))$, and the measure $\mu$ on $\mc{S}$ is given by $\mu = \hat{\mu} \circ \pi^{-1}$, with the canonical projection $\pi(x, v) = x$ from $\hat{\mc{S}}$ to $\mc{S}$. To facilitate the discussion, it is convenient to disintegrate the measure $\hat{\mu}$ as
\begin{align} \label{eq:decompomuh}
    \hat{\mu}(\rd x \, \rd v) = \mu(\rd x) \kappa_x(\rd v)\,,
\end{align}
where $\kappa_x$ is a conditional probability on $\mc{V}$ given $x \in \mc{S}$. 

To align with the lifting framework outlined in \cref{sec:lifthilbert}, we note that the Hilbert space $L^2(\mc{S}, \mu)$ can be naturally viewed as a subspace of $L^2(\hat{\mc{S}}, \hat{\mu})$ with the induced norm. Additionally, the Markov processes $P_t$ and $\hat{P}_t$ define contraction $C^0$-semigroups on the spaces $\mhh = L^2(\mc{S},\mu)$ and $\mh = L^2(\hs,\hat{\mu})$, respectively. 
Indeed, by extending functions $f \in L^2(\mc{S},\mu)$ to $\hat{\mc{S}}$ via the projection $\pi$ (i.e., $f \mapsto f \circ \pi$), we obtain from \eqref{eq:decompomuh} that 
\begin{align*}
    \langle f \circ \pi, f \circ \pi \rangle_{L^2(\hat{\mu})} &= \int_{\mc{S} \times \mc{V}} \overline{f(x)} g(x) \, \hat{\mu}(\rd x \, \rd v) 
    = \int_{\mc{S}} \overline{f(x)} g(x) \, \mu(\rd x) = \langle f, g \rangle_{L^2(\mu)}\,.
\end{align*}
We also define the orthogonal projection (i.e., conditional expectation) by
\begin{align} \label{eq:condexpvel}
    \Pi_v f(x, v) = \int_{\mc{V}} f(x, v) \, \kappa_x(\rd v): L^2(\hat{\mc{S}}, \hat{\mu}) \to L^2(\mc{S}, \mu)\,.
\end{align}
We are now prepared to reformulate the notion of lifting (\cref{def:generallift}) within the context of diffusion processes:
\begin{definition} \label{def:liftdiffusion}
    Let $P_t = \exp(tL)$ and $\hat{P}_t = \exp(t \hat{L})$ be the Markov diffusion processes defined as above. We say that $\hat{P}_t$ is a lift of $P_t$ if
    it holds that $\dom(L) \subset \dom(\hat{L})$, and 
        \begin{align} \label{eq:dilift1}
            \l f \circ \pi, \hat{L} (g \circ \pi)  \r_{L^2(\hat{\mu})} = 0\,,\q \forall f \in L^2(\mc{S},\mu)\,,\ g \circ \pi \in \dom(\hl)\,,
        \end{align}
    and for a positive bounded operator $\mf{S}: L^2(\mc{S}, \mu)^\perp \to L^2(\mc{S}, \mu)^\perp$, 
    \begin{align} \label{eq:dilift2}
        \l \hl(f \circ \pi), \mf{S}\, \hl (g \circ \pi)  \r_{L^2(\hat{\mu})} = - \l f, L g  \r_{L^2(\mu)}\,, \q \forall f \circ \pi \in \dom(\hl)\,,\ g \in \dom (L)\,.
    \end{align}
\end{definition}
Our \cref{def:liftdiffusion} coincides with \cite{eberle2024non}*{Definition 1}, except for the presence of the operator $\mf{S}$ in our formulation and minor differences in domain specifications. We next focus on the lifts of overdamped Langevin diffusion to showcase the choice of $\mf{S}$. We first note that the orthogonal complement $L^2(\mc{S}, \mu)^\perp$ can be characterized by
\begin{align*}
    L^2(\mc{S}, \mu)^\perp = \left\{f(x,v) \in L^2(\hs,\hm)\,;\ \Pi_v f(x,v) = \int_{\mc{V}} f(x, v) \, \kappa_x(\rd v) = 0 \right\}\,.
\end{align*}
When $\kappa_x = \kappa$ is independent of $x$ (i.e., $\hm (\rd x \, \rd v) = \mu(\rd x) \kappa (\rd v)$), any positive bounded operator $\mf{S}_v$ on the zero-mean subspace: 
$$L^2_0(\mc{V},\ka) = \{g(v) \in L^2(\mc{V},\ka)\,;\ \int_{\mc{V}} g(v) \, \kappa(\rd v) = 0\}$$ 
naturally extends to a positive bounded operator on 
$L^2(\mc{S}, \mu)^\perp$ by tensoring with identity 
\begin{align*}
    \mf{S} = \text{id}_{L^2(\mc{S},\mu)} \otimes \mf{S}_v\,.    
\end{align*}

Let $U \in C^1(\mathbb{R}^d)$ be a potential satisfying $U(x) \to \infty$ as $|x| \to \infty$ and $\int \exp(-U(x)) \, dx < \infty$, and define the associated Gibbs measure $\mu(\rd x) \propto \exp(- U(x)) \ud x$. The overdamped Langevin generator on $L^2(\R^d, \mu)$ is defined by 
\begin{align} \label{genold}
    L f =  -  \nabla_x U \cdot \nabla_x f +  \Delta_x f = - \nabla_x^\star \nabla_x f\,,
\end{align}
for $f \in C_0^\infty(\mathbb{R}^d)$, where $\na^\star$ is the adjoint of $\na$ in the Hilbert space $L^2(\R^d, \mu)$. Following the discussion in \cite{eberle2024non}, the following Markov semigroups could be regarded as lifts of overdamped Langevin dynamics in the sense of \cref{def:liftdiffusion} with flexible choices of $\mf{S}$:

\begin{enumerate}[wide]
    \item \emph{Deterministic Hamiltonian flow} associated with the Hamiltonian $H(x,v) = U(x) + \frac{1}{2}|v|^2$ on the state space $\hs = \R^d \times \R^d$ is given by the ODE: 
    \begin{align*}
        \rd x_t = \na_v H(x_t,v_t) \ud t = v_t \ud t\,,\q \rd v_t = - \na_x H(x_t,v_t) \ud t = - \na_x U(x_t) \ud t\,. 
    \end{align*}
    Let $p_0$ be an initial distribution on $(x_0,v_0)$, and denote by $p_t$ the pushforward measure induced by the flow $(x_0, v_0) \to (x_t, v_t)$. This defines a Markov process with the invariant measure $\hm (\rd x \, \rd v) = \mu(\rd x) \kappa (\rd v) \propto \exp( - H) \rd x \, \rd v$ and the generator 
    \begin{align} \label{genhamflow}
        \hl f (x,v) = v \dd \na_x f(x,v) - \na_x U(x) \dd \na_v f(x,v)\,.
    \end{align}
    To see $\hl$ in \eqref{genhamflow} is a lift of $L$ in \eqref{genold}, by approximation, it suffices to verify conditions \eqref{eq:dilift1}--\eqref{eq:dilift2} for $f(x),g(x)\in C_0^\infty$. Indeed, a direct computation gives 
    \begin{align*}
        \l f \circ \pi, \hat{L} (g \circ \pi)  \r_{L^2(\hat{\mu})} = \int_{\R^d \times \R^d}  \overline{f(x)}\, v \dd \na_x g(x) \, \mu(\rd x) \kappa (\rd v) = 0 \,, 
    \end{align*}
    by $\kappa(\rd v) \propto e^{-\frac{|v|^2}{2}}\, \rd v$ and $\int_{\R^d} v \, \ka(\rd v) = 0$, and 
    \begin{align*}
         \l \hl(f \circ \pi), \mf{S}\, \hl (g \circ \pi)  \r_{L^2(\hat{\mu})} & = \int_{\R^{2d}} \overline{v \dd \na _x f(x)} \, \mf{S} \left(v \dd \na _x g(x)\right) \, \hm(\rd x) \\
         & = \int_{\R^d} \overline{\na _x f(x)}^\top \left(\int_{\R^d} v  (\mf{S}_v (v))^\top \ka(\rd v) \right) \na_x g(x) \, \mu(\rd x) \\
         & = \int_{\R^d} \overline{\na _x f(x)}^\top \na_x g(x) \mu(\rd x) =  - \l f, L g \r_{L^2(\mu)}\,,
    \end{align*}
    if and only if 
    \begin{align} \label{oldliftcond}
        \int_{\R^d} v  (\mf{S}_v (v))^\top \ka(\rd v) = I_{d \times d}\,.
    \end{align}
    where $I_{d \times d}$ is the identity matrix on $\R^d$, and $\mf{S}$ is defined by    
    $\mf{S} := \text{id}_{L^2(\R^d,\mu)} \otimes \mf{S}_v$ for a positive bounded operator $\mf{S}_v$ on $L^2_0(\mc{V},\ka)$. Recalling that $\ka$ is a standard Gaussian, it is clear that \cref{oldliftcond} holds for any $\mf{S}$ such that $\mf{S}_v(v) = v$. For instance, one could take $\mf{S}_v$ as the identity operator as in \cite{eberle2024non}, or the inverse of the negative Ornstein-Uhlenbeck generator:
    \begin{align} \label{genoup}
        \mf{S}_v = (- L_{\rm OU})^{-1}\,,\q L_{\rm OU} = - v \dd \na_v + \Delta_v\,,
    \end{align}
    which is positive and bounded on $L_0^2(\R^d,\ka)$, or the negative full refreshment generator:
    \begin{align} \label{genrefresh}
        \mf{S}_v = - R\,,\q R = \Pi_v - \id\,,
    \end{align}
    where $\Pi_v$ is defined as in \cref{eq:condexpvel}.
    \item \emph{Underdamped Langevin dynamics} (ULD) corresponds to the SDE on $(x,v) \in \R^d \times \R^d$:
    \begin{align*}
        \rd x_t = v_t \rd t\,,\q \rd v_t = - \na U(x_t)  - \gamma v_t \rd t + \sqrt{2 \gamma} \ud B_t\,, 
    \end{align*}
    where $B_t$ is a standard Brownian motion on $\R^d$ and $\gamma > 0$ is the friction parameter. Its generator is given by $\hl_\gamma = \hl + \gamma L_{\rm OU}$ with invariant measure $\hm = \mu \otimes \ka \propto e^{- U(x) - \frac{1}{2}|v|^2}$, where $\hl$ and $L_{\rm OU}$ are given in \eqref{genhamflow} and \eqref{genoup}, respectively. Thanks to $L_{\rm OU}(f \circ \pi) = 0$ for all $f(x)$, all the above discussion for deterministic Hamiltonian flow holds for underdamped Langevin processes, and hence $\hl_\gamma$ is a lift of $L$ for any $\gamma > 0$ with $\mf{S} = \id$, $(- L_{\rm OU})^{-1}$, or $- R$. 
    
    \item \emph{Randomized Hamiltonian Monte Carlo} (RHMC) simulates Hamiltonian dynamics for a random duration $T \sim \exp(\gamma)$, where $\gamma$ is the refresh rate. Then the velocity (momentum) is fully resampled from a Gaussian distribution $\kappa = \mathcal{N}(0, I_d)$, preventing correlations from accumulating over iterations. The associated Markov process has the generator $\hl_\gamma = \hl + \gamma R$ and leaves invariant the measure $\hm \propto e^{- U(x) - \frac{1}{2}|v|^2}$, with $\hl$ and $R$ given in \eqref{genhamflow} and \eqref{genrefresh}, respectively. Again, since $R(f\circ \pi) = 0$ for any $f(x)$, by the discussion for Hamiltonian flow, $\hl_\gamma$ for RHMC is also a lift of $L$ for any $\gamma > 0$. 
\end{enumerate}

Note that the generator \eqref{genhamflow} for the Hamiltonian flow corresponds precisely to the case $\gamma = 0$ in the generators $\hl_\gamma$ for both ULD and RHMC. The damping parameter (or the refresh rate) $\gamma > 0$ plays a crucial role in achieving the optimal convergence acceleration through lifting, as shown in \cref{thm:uppermatrix,thm:lowerboundpt}, which demonstrate that a careful selection of $\gamma$ leads to the maximal convergence rate. Specifically, assuming $\na^2 U(x) \ge - K \lad(L_{\rm OU}) I_{d \times d}$ for some $K \ge 0$, where $\lad(L_{\rm OU})$ denotes the spectral gap of $L_{\rm OU}$, one can show that for both ULD and RHMC, setting $\gamma = \Theta(\sqrt{(1 + K) \lad(L_{\rm OU})})$ yields the maximal $L^2$ convergence rate $\nu = \Omega (\sqrt{\lad(L_{\rm OU})}/\sqrt{1+ K}$ of $\{\hl_\gamma\}_{\gamma \ge 0}$, 
establishing the optimal lifts in these two cases; see \cite{cao2023explicit} and \cite{eberle2024space}*{Section 5.5} for precise statements. 

Furthermore, many piecewise-deterministic Markov processes, such as the Bouncy Particle Sampler, Event Chain Monte Carlo, and the Zig-Zag process, can also be interpreted as lifts of overdamped Langevin dynamics. For additional applications, we refer readers to \cite{eberle2024space}, which extends this framework to reversible diffusions on Riemannian manifolds with boundary reflections, and to \cite{eberle2025convergence}, where detailed analyses of sampling algorithms and processes with non-trivial boundary behavior are provided.

\subsection{Lifting finite Markov chains via quantum dynamics}  \label{sec:liftfinitemarkov}
As discussed in the Introduction, low-dimensional Markov chains typically exhibit diffusive mixing behavior, characterized by a time scaling as $\Theta(n^2)$ to traverse a distance $n$. A canonical example is the simple symmetric random walk on the chain $\mathbb{Z}/n\mathbb{Z}$, whose mixing time is $\Theta(n^2)$. The seminal work \cite{diaconis2000analysis} analyzed 
a lifted random walk on the product space $(x,v) \in \mathbb{Z}/(n\mathbb{Z}) \times \{+1, -1\}$ that at each step flips its velocity with probability $\frac{1}{n}$ and 
moves in the $x$-state space with velocity $v$. It was shown that this lifted chain achieves a quadratic speedup with mixing time $\Theta(n)$. 

This section focuses on the lifting of reversible Markov chains on the state space $\{1, 2, \dots, n\}$ with a uniform stationary measure using quantum Markov semigroups. We shall establish the existence of an optimal (quantum) lift for the simple random walk on the chain, thereby complementing the results of \cite{diaconis2000analysis}.
Specifically, for a reversible classical Markov semigroup $\lo$ on the finite set $\{1,2,\ldots,n\}$ with the unique uniform stationary distribution,
building on the frameworks developed in \cref{sec:liftingmatrix,sec:lifthilbert}, we construct a primitive QMS generator $\mc{L}_\gamma  = \lh + \gamma \ld$ with the maximally mixed state $\frac{\mi}{n}$ as the unique invariant state, which would satisfy Conditions~\ref{asspA}--\ref{assump:overdp} and serve as a lift of $\lo$. 

Following the framework in \cite{li2024quantum}*{Section 5.1}, consider a symmetric matrix $A \in \mathbb{C}^{n \times n}$ with distinct eigenvalues $\{\kappa_i\}_{i=1}^n \subset \mathbb{R}$ and spectral decomposition $A = \sum_{i=1}^n \kappa_i P_i$, where $P_i = \ket{i}\bra{i}$ is the projection onto the eigenspace for $\kappa_i$. We define the dephasing generator
\begin{align} \label{eq:singlejump}
   \ld(X) := - [A, [A, X]] = - \sum_{i,j} (\kappa_i - \kappa_j)^2 P_i X P_j\,,
\end{align}
which is self-adjoint on $\mc{B}(\mathbb{C}^n)$. The operator $\ld$ has eigenvalues $-(\kappa_i - \kappa_j)^2$ with corresponding eigenspaces $P_i \mc{B}(\mathbb{C}^n) P_j$. Consequently, its operator norm and spectral gap are given by
\begin{align} \label{eq:ep1norm}
 \|\ld\|_{(2,\frac{\mi}{n}) \to (2,\frac{\mi}{n})} = \max_{i,j} |\kappa_i - \kappa_j|^2, \quad \lad_S = \min_{i\neq j} |\kappa_i - \kappa_j|^2,
\end{align}
and its kernel consists of matrices commuting with $A$:
\begin{equation} \label{exp:kernel1}
   \ker(\ld) = \Bigl\{X \in \mc{B}(\mathbb{C}^n) \,\Big|\, X = \sum_{j=1}^n x_j \ket{j}\bra{j}, \ x_j \in \mathbb{C} \Bigr\},
\end{equation}
i.e., the diagonal matrices in the basis $\{\ket{i}\bra{j}\}_{i,j=1}^n$.
We next define $\lh$ by $\lh = \mc{L}_H = i[H, \cdot]$ the coherent generator induced by a Hamiltonian $$H = \sum_{i,j} h_{ij} \ket{i}\bra{j}\,.$$ Then \cref{asspA} is automatically satisfied. By \cref{eq:kerlamma}, the Lindbladian $\mc{L}_\gamma = \mc{L}_H + \gamma \ld$ is primitive if and only if
\begin{align} \label{eq:primicond}
     \ker(\mc{L}_H) \cap \ker(\ld) = \operatorname{Span}\{\mi\}.
\end{align}
For $X = \sum_j x_j \ket{j}\bra{j} \in \ker(\ld)$, the condition $X \in \ker(\mc{L}_H)$ reduces to
\begin{equation} \label{eq:ep1commu}
   (x_i - x_j) h_{ij} = 0 \quad \forall i,j.
\end{equation}
Thus, \eqref{eq:primicond} holds when the only solution to \eqref{eq:ep1commu} is $x_j \equiv c$ for some $c \in \mathbb{C}$. This is equivalent to the graph connectivity condition: for any pair of indices $\ell \neq m$, there exists a path $(\ell_1, \ldots, \ell_k)$ with $\ell_1 = \ell$, $\ell_k = m$, and $h_{\ell_j \ell_{j+1}} \neq 0$ for all $j$. 

In summary, the generator $\mc{L}_\gamma = \mc{L}_H + \gamma \ld$ is primitive if and only if $H$ is irreducible (i.e., the graph associated with its off-diagonal elements is connected). Moreover, in this case, \cref{asspB} follows since $\dim \ker(\ld) = n > 1 = \dim \ker \mc{L}_\gamma$. Finally, \cref{assump:PHP} always holds: for any $X = \sum_j x_j \ket{j}\bra{j} \in \ker(\ld)$ and self-adjoint $H$, we have
\begin{align*}
    \langle X, \mc{L}_H X \rangle_{2,\frac{\mi}{n}} = i \langle X, [H, X] \rangle_{2, \frac{\mi}{n}} = 0,
\end{align*}
which implies $\es \mc{L}_H \es = 0$, where $\es$ denotes the projection onto $\ker(\ld)$.

For our analysis, we define $\lo$ through the alternative formulation \eqref{def:generallo} with $\mc{S} = \id$ (the primary one \eqref{eq:odlimit_generator} would require more involved arguments). To derive the explicit form of $\lo$, we compute, for $X = \sum_j x_j \ket{j}\bra{j}$ and $Y = \sum_j y_j \ket{j}\bra{j}$ in $\ker(\ld)$:

\begin{equation}\label{def:loclass}
    \begin{aligned}
        \langle [H, Y], [H, X] \rangle_{2,\frac{\mi}{n}} 
        &= \sum_{i,j} y_i^* x_j \langle [H, \ket{i}\bra{i}], [H, \ket{j}\bra{j}] \rangle_{2,\frac{\mi}{n}} \\
        &= \frac{1}{n}\sum_{i,j} y_i^* x_j \left(2 \bra{i} H^2 \ket{j}\delta_{ij} - 2 |h_{ij}|^2 \right) \\
        &= \frac{2}{n}\sum_i y_i^* x_i \sum_k |h_{ik}|^2 - \frac{2}{n}\sum_{i,j} y_i^* x_j |h_{ij}|^2 \\
        &= -\frac{1}{n} \langle y, \hat{H} x \rangle\,,
    \end{aligned}
\end{equation}
where the third equality follows from
\begin{align*}
   \bra{i} H^2 \ket{i} = \sum_{k=1}^n |h_{ik}|^2\,,
\end{align*}
and $\hat{H}$ is a $Q$-matrix (i.e., non-negative off-diagonals with zero row sums) defined by
\begin{equation} \label{def:matrix}
    \hat{H}_{ij} := 
    \begin{cases}
        2|h_{ij}|^2 & \text{for } i \neq j\,, \\
        -2\sum_{l \neq i} |h_{il}|^2 & \text{for } i = j\,.
    \end{cases}
\end{equation}
The irreducibility of $H$ ensures that $\hat{H}$ is also irreducible and symmetric. Consequently, $\exp(t \hat{H})$ generates an ergodic continuous-time Markov chain with uniform stationary measure $\nu$. Equipping $\C^n$ with the $\nu$-weighted inner product $\langle \cdot, \cdot \rangle_{\ell^2(\nu)}$, we recognize from \eqref{def:loclass} that
\begin{align*}
    \langle [H, Y], [H, X] \rangle_{2,\frac{\mi}{n}} = -\langle x, \hat{H} y \rangle_{\ell^2(\nu)},
\end{align*}
which establishes $\lo = \hat{H}$ as the generator of a primitive symmetric QMS on the commutative algebra \eqref{exp:kernel1}, verifying \cref{assump:overdp}. Therefore, following the arguments in \cref{sec:liftingmatrix}, we have that $\mc{L}_\gamma$ constitutes a lift of $\lo$ in the sense of \cref{def:generallift}, and it admits convergence rate estimates analogous to those in \cref{thm:lower,thm:uppermatrix}. 

Further, conversely, given a symmetric ergodic Markov semigroup $\exp(tQ)$ on the state space $\{1,\ldots,n\}$, we can construct its quantum lifted dynamics as above. Specifically, define the Hamiltonian $H = \sum_{i,j} h_{ij} \ket{i}\bra{j}$ by 
\begin{align} \label{eq:hamq}
    h_{ij} = \begin{cases}
        \sqrt{Q_{ij}/2}\,, & \text{if $i \neq j$\,,} \\
        0\,, & \text{if $i = j$\,,}
    \end{cases}
\end{align}
and let $\mc{L}_\gamma = \mc{L}_H + \gamma \ld$ be the Lindbladian with $\ld$ given in \eqref{eq:singlejump} and $\mc{L}_H = i[H, \dd]$. This QMS $\exp(t \mc{L}_\gamma)$ is a lift of $\exp(t Q)$ with convergence guarantees (\cref{thm:lower,thm:uppermatrix}). We now summarize the above discussions into the following result. 

\begin{proposition}
   Let $\exp(tQ)$ be a symmetric ergodic Markov semigroup on $\{1,\ldots,n\}$, and let $\mc{L}_\gamma = \mc{L}_H + \gamma \ld$ ($\gamma > 0$) be the Lindbladian defined through \eqref{eq:singlejump} and \eqref{eq:hamq}, where $A$ has $n$ distinct eigenvalues $\{\kappa_i\}_{i=1}^n$ with $\kappa_i = \Theta(1)$.
    Then, it holds that
    \begin{itemize}
        \item  $\mc{L}_\gamma$ is primitive and hypocoercive with the maximally mixed state $\frac{\mi}{n}$ as the unique invariant state. Moreover, $\exp(t \mc{L}_\gamma)$ is a lift of $\exp(t Q)$ in the sense of \cref{def:generallift}.
        \item $\mc{L}_\gamma$ has the maximal 
     $L^2$ convergence rate estimate at a particular $\gamma_{\max}$:
    \begin{align} \label{eq:l2esti}
        \Omega\left( \frac{\sqrt{\lad_Q}}{1  +   K_1  + K_2 \lad_Q^{-1/2}} \right) \leq \nu(\mc{L}_\gamma) \leq \mc{O}(\sqrt{\lad_Q})\,,
    \end{align}
    where $\lad_Q$ denotes the spectral gap of $Q$, and the constants $K_1, K_2 > 0$ satisfy  
    \begin{align} \label{eq:keyestforopt}
          \norm{[H,[H, Y]]}_{2,\frac{\mi}{n}} \le K_1 \norm{Q y}_{\ell^2(\nu)} + K_2 \sqrt{- \l y, Q y\r_{\ell^2(\nu)}}\,,
    \end{align}
    for $Y = \sum_j y_j \ket{j}\bra{j} \in \ker(\ld)$. 
    \end{itemize}
\end{proposition}
The convergence rate estimate \eqref{eq:l2esti} follows from arguments analogous to those for \eqref{eq:estlgammasim}, while the key inequality \eqref{eq:keyestforopt} represents a simplified version of \eqref{eq:simpleassp3}. To obtain optimal constants $(K_1, K_2)$, we reformulate \eqref{eq:keyestforopt} in terms of $H$ and $Q$.  We compute:
\begin{align*}
    [H, [H, Y]] &= \sum_j y_j [H, [H, P_j]] \\ &= \sum_j x_j \left(H^2 P_j - 2 H P_j H + P_j H^2  \right),
\end{align*}
for $Y = \sum_j y_j P_j$ with $P_j = \ket{j}\bra{j}$, and then 
\begin{align*}
   \norm{[H, [H, Y]]}_{2,\frac{\mi}{n}}^2 = \frac{1}{n}\sum_{i,j} x_i^* x_j M_{ij}\,,
\end{align*}
where, by noting $(h_{ij})$ in \eqref{eq:hamq} is real symmetric, the positive semidefinite matrix $M$ has entries:
\begin{align} \label{def:matrixm}
    M_{ij} 
  = & 2 \bra{i} H^4 \ket{j} \d_{ij} - 8 \bra{i} H^3 \ket{j}  h_{ij} + 6 \bra{i}H^2\ket{j}^2\,.
\end{align}
To establish \eqref{eq:keyestforopt}, it suffices to find positive constants $C_1, C_2$ such that:
\begin{align}\label{eq:matrixbound}
    M \preceq C_1 Q^2 - C_2 Q,
\end{align}
which immediately gives the desired inequality with $K_1 = \sqrt{C_1}$ and $K_2 = \sqrt{C_2}$. 

It is important to note that verifying \eqref{eq:matrixbound} only requires  $Q$, as the definition of  $M$ in \eqref{def:matrixm} depends solely on  $Q$ through \eqref{eq:hamq}. To achieve the optimal (quantum) lift with a quadratic speedup for the mixing of $\exp(t Q)$, it is necessary that \eqref{eq:matrixbound} holds with  
\begin{equation} \label{keyrela}
    \sqrt{C_1} + \sqrt{C_2 / \lambda_Q} = \Theta(1)\,.    
\end{equation}
An interesting open question is to characterize the class of $Q$-matrices for which this property holds. For specific examples with explicit forms of $Q$, this may be verified through direct computation.
For instance, as we show next, in the case of a simple random walk on the chain as in \cite{diaconis2000analysis}, the above construction indeed yields an optimal lift. 
We consider the simple nearest-neighbor random walk on $\{1, \ldots, n\}$, defined by the transition matrix  $P(x, y)$, where  $P(x, y) = 1/2$ for $y = x \pm 1$, and $ P(1, 1) = P(n, n) = 1/2$. The corresponding continuous-time Markov chain is given by $ \exp(t Q)$ with $Q = P - I$. This Markov chain is reversible and ergodic, with the uniform distribution as its stationary measure. Then, for this model, a straightforward computation shows \eqref{eq:matrixbound} with $C_1 = \frac{3}{2}$ and $C_2 = 0$:
\begin{align*}
    M \preceq \frac{3}{2} Q^2\,,
\end{align*}
where 
\begin{equation*}
    M = \begin{pmatrix}
  \frac{5}{8} & -1 & \frac{3}{8} & 0 & \cdots & 0 & 0 \\ - 1 & \frac{17}{8} & - \frac{3}{2} & \frac{3}{8} & \cdots & 0 & 0 \\ \frac{3}{8} & - \frac{3}{2} & \frac{9}{4} & - \frac{3}{2} & \cdots & 0 & 0 \\ 0 & \frac{3}{8} & - \frac{3}{2} & \frac{9}{4} & \cdots & 0 & 0 \\ \vdots & \vdots & \vdots & \vdots & \ddots & \vdots & \vdots \\ 0 & 0 & 0 & 0 & \cdots & \frac{17}{8} & - 1 \\ 0 & 0 & 0 & 0 & \cdots & - 1 & \frac{5}{8} 
\end{pmatrix}\,, \q Q^2 = 
\begin{pmatrix}
  \frac{1}{2} & - \frac{3}{4} & \frac{1}{4} & 0 & \cdots & 0 & 0 \\ - \frac{3}{4} & \frac{3}{2} & - 1 & \frac{1}{4} & \cdots & 0 & 0 \\ \frac{1}{4} & - 1 & \frac{3}{2} & -1 & \cdots & 0 & 0 \\ 0 & \frac{1}{4} & - 1 & \frac{3}{2} & \cdots & 0 & 0 \\ \vdots & \vdots & \vdots & \vdots & \ddots & \vdots & \vdots \\ 0 & 0 & 0 & 0 & \cdots & \frac{3}{2} & - \frac{3}{4} \\ 0 & 0 & 0 & 0 & \cdots & - \frac{3}{4} & \frac{1}{2} 
\end{pmatrix}.
\end{equation*}
It follows that \cref{keyrela} holds, and $\exp(t \mc{L}_\gamma)$ for some $\gamma > 0$ has the sharp $L^2$ convergence rate $\nu = \Theta(\sqrt{\lad_Q})$ and hence is an optimal lift of $\exp(t Q)$. 

\subsection{Lifting quantum Markov processes} \label{sec:appbipart} 
In this section, we proceed to study the lifting of detailed balanced quantum Markov semigroups, following the framework outlined in \cref{sec:liftingmatrix}. 
Let $\exp(t \mathcal{L}_O)$ be a detailed balanced QMS on a von Neumann algebra $\mc{M}_B \subset \mathcal{B}(\maf{H}_B)$. Without loss of generality, we consider the lifted processes (forming a class of hypocoercive QMS) $\exp(t \mathcal{L}_\gamma)$ of the form \eqref{eq:decom}, acting on $ \mathcal{M}_{AB}:= \mathcal{M}_A \otimes \mathcal{M}_B$, where $\mathcal{M}_A = \mathcal{B}(\maf{H}_A)$ represents the extended system and $\mathcal{M}_B$ is the primary system. Here both $\maf{H}_A$ and $\maf{H}_B$ are finite-dimensional Hilbert spaces. 
Throughout this section, we use the subscripts $A$,  $B$, and  $AB$ to indicate the subsystem over which the Hilbert-Schmidt (HS) inner product is defined. For example, $\langle \cdot, \cdot \rangle_A$ denotes the HS inner product on $\mathcal{M}_A$.

Specifically, let $\mc{L}^A$ be a primitive QMS generator acting on the subsystem $A$ with the unique full-rank invariant state $\si^A \in \mc{D}_+(\mc{M}_A)$. We define the generator $\ld$ in \eqref{eq:decom} by extending $\mc{L}^A$ to the composite system $\mc{M}_{AB}$ as follows:
\begin{align} \label{def:lsbip}
    \ld = \mathcal{L}^A \otimes \id^B\,,
\end{align}
where $\id^B$ is the identity map on $\mc{M}_B$. We define $\lh$ in  \eqref{eq:decom} as a coherent term:
\begin{align} \label{Hamiltonianterm}
    \lh(X) = \mc{L}_H(X) := i [H, X]\,,
\end{align}
where $H \in \mc{B}(\maf{H}_A \otimes \maf{H}_B)$ (not necessarily belonging to $\mc{M}_{AB}$) is a traceless Hamiltonian, i.e., $\tr(H) = 0$. Let $H^A = \tr_B(H)$ and $H^B = \tr_A(H)$ be the reduced subsystem Hamiltonians. Then, $H$ can be decomposed as: 
\begin{align} \label{eq:hamil}
    H = H^A \otimes \mi^B + \mi^A \otimes H^B + H^{AB}\,,\q \tr_A (H^{AB}) = 0 = \tr_B (H^{AB})\,,
\end{align}
where $H^{AB}$ consists exclusively of interaction terms in the composite system. 
This framework also naturally encompasses various boundary-driven spin chain models that appear in nonequilibrium statistical mechanics
\cites{prosen2012comments,landi2022nonequilibrium,tupkary2023searching}.
To avoid potential confusion with the subsystem Lindbladian  $\mathcal{L}^A$, we will use $\mathcal{L}_H$ to refer to the term in \eqref{Hamiltonianterm}, rather than $\lh$. Thus, the model under consideration in this section can be written as:
\begin{align} \label{defhypoqms}
    \mathcal{L}_\gamma = \mathcal{L}_H + \gamma \mathcal{L}_S = i [H, \cdot] + \mathcal{L}^A \otimes \mathrm{id}^B\,.
\end{align}
Next, in \cref{sec:characterlift}, we first characterize the conditions under which 
$\mathcal{L}_\gamma$ in \eqref{defhypoqms} can be regarded as a lift of some detailed balanced Lindbladian $\lo$ supported on $\mc{M}_B$, using the results of \cref{sec:liftingmatrix}. Conversely, in \cref{sec:constructlift}, we showcase how to construct (optimal) lifts $\mathcal{L}_\gamma$ for a given detailed balanced generator $ \mathcal{L}_O$. 

\subsubsection{Characterizing lifting conditions} \label{sec:characterlift} 
Building on the results in \cref{sec:liftingmatrix}, it suffices to verify Conditions \ref{asspA}--\ref{assump:overdp} to ensure that $\mc{L}_\gamma$ serves as a lift of $\lo$ defined as in \eqref{eq:odlimit_generator} by \cref{def:2ndlift}. This also allows us to apply the convergence rate estimates from \cref{thm:lower,thm:uppermatrix} to the lifted process $\exp(t \mc{L}_\gamma)$. The main result of this section is that, for bipartite systems, many of these conditions either simplify significantly or become unnecessary. We start by translating \cref{asspA} into the following one:
\begin{conditionnp}{\ref{asspA}'}{Reformulation of \cref{asspA}}\label{asspA2}
For a full-rank state $\si \in \mc{D}_+(\mc{M}_{AB})$, it holds that $[\si, H] = 0$ for the Hamiltonian $H$ \eqref{eq:hamil} and $\ld = \mathcal{L}^A \otimes \id^B$ is $\si$-KMS detailed balanced.
\end{conditionnp}


\noindent
By \cref{lem:symanti}, $\si \in \mc{D}_+(\mc{M}_{AB})$ in \cref{asspA2} is a full-rank invariant state of $\exp(t \mc{L}^\dag_\gamma)$ for any $\gamma > 0$. To see the equivalence between \cref{asspA} and \cref{asspA2}, it suffices to note that $\mc{L}_H(X) = i [H, X]$ is anti-self-adjoint for the KMS inner product $\l \dd, \dd\r_{\si,1/2}$ if and only if $[H,\si] = 0$. 
Under above assumptions, we show that the invariant state $\si \in \mc{D}_+(\mc{M}_{AB})$ necessarily admits a product structure, by generalizing 
\cite{carlen2025stationary}*{Theorem 1} as follows.  

\begin{proposition} \label{thm:uniquefixed}
Let $\ld$ be defined in \eqref{def:lsbip}, where $\mc{L}^A$ is a primitive QMS generator on the subsystem $A$ with invariant state $\si^A \in \mc{D}_+(\mc{M}_A)$. Suppose that \cref{asspA2} holds. Then, 
\begin{itemize}
    \item There exists a state $\si^B \in \mc{D}_+(\mc{M}_B)$ such that 
    \begin{align} \label{eq:productstate}
        \si = \si^A \otimes \si^B\,.
    \end{align} 
    Moreover, $\mc{L}^A$ is $\si^A$-KMS detailed balanced. 
    \item It holds that 
    \begin{align} \label{eq:cond2}
        \left[\si^A, H^A \right] = 0\,,\q \left[\si^B, H^B\right] = 0\,,\q \left[\si, H^{AB}\right] = 0\,,
    \end{align}
    where Hamiltonians $H^A$, $H^B$, and $H^{AB}$ are given in \eqref{eq:hamil}.
\end{itemize}
\end{proposition}

\begin{remark}
It is worth emphasizing that if $\mc{L}_\gamma$ is assumed to be primitive, then the (unique) invariant state $\si$ of $\mc{L}_\gamma$ in \eqref{eq:productstate} is not a Gibbs state associated with the Hamiltonian $H$. In fact, if $\si \propto e^{-\beta H}$ is of the product form, then necessarily $H^{AB} = 0$ holds. It follows that $\si^A \otimes \si^B$ is an invariant state of $\mc{L}_\gamma$ for any $\si^B$ on $B$ such that $[\si^B, H^B] = 0$, which contradicts with the primitivity of $\mc{L}_\gamma$. 
\end{remark}

\begin{proof}[Proof of \cref{thm:uniquefixed}]
The proof is similar to that of \cite{carlen2025stationary}*{Theorem 1}. We consider the positive semidefinite operator $\mc{L}^A(\mc{L}^A)^\dag$ with kernel of dimension one and spanned by $\si^A$, due to the primitivity of $\mc{L}^A$. Suppose that it admits the eigendecomposition: for $j \ge 0$, 
 \begin{align*}
    \mc{L}^A(\mc{L}^A)^\dag \rho_j = \lad_j \rho_j\,,
 \end{align*}
 where $\lad_0 = 0$ and $\lad_j > 0$ for $j \ge 1$, and $\{\rho_j\}_{j \ge 0}$ with $\rho_0 \propto \si^A$ are orthonormal basis of $\mc{M}_A = \mc{B}(\maf{H}_A)$ with respect to the HS inner product. It allows us to decompose $\si$ as follows: 
 \begin{align*}
    \si = \sum_{j \ge 0} \rho_j \otimes W_j\,.
 \end{align*}
 Note that 
 \begin{align} \label{auxeq1ld}
   0 = \ld^\dag (\si) = (\mc{L}^A)^\dag \otimes \id^B (\si) = \sum_{j \ge 0}  (\mc{L}^A)^\dag  \rho_j \otimes W_j\,,
 \end{align}
 which implies, by denoting $X_k = (\mc{L}^A)^\dag  \rho_k$,
 \begin{align*}
    0 & = \tr_A\left( (X_k^* \otimes \mi) \ld^\dag (\si) \right) \\ & = \sum_{j \ge 0} \tr_A \left(\rho_k^* \mc{L}^A (\mc{L}^A)^\dag  \rho_j \right) \otimes W_j =  \sum_{j \ge 0} \lad_j \d_{jk} W_j = \lad_k W_k\,.
 \end{align*}
 It follows that $W_j = 0$ for $j \ge 1$ as $\lad_j > 0$, and hence $\si$ is a product state \eqref{eq:productstate}. Then, by the detailed balance of $\mathcal{L}^A \otimes \id^B$, we have, for $X^A, Y^A \in \mc{M}_A$, 
 \begin{align*}
     \left\l \mc{L}^A(X^A) \otimes \mi^B, Y^A \otimes \mi^B \right\r_{\si,1/2} 
     & = \left\l X^A \otimes \mi^B, \mc{L}^A(Y^A)\otimes  \mi^B \right\r_{\si,1/2}\,,
 \end{align*}
 which, by \cref{eq:productstate}, further yields 
 \begin{align*}
     \left\l \mc{L}^A(X^A), Y^A \right\r_{\si^A,1/2} 
     & = \left\l X^A, \mc{L}^A(Y^A)\right\r_{\si^A,1/2}\,.
 \end{align*}
 That is, $\mc{L}^A$ is $\si^A$-KMS detailed balanced. The rest is to prove \eqref{eq:cond2} by using $[H, \si] = [H, \si^A \otimes \si^B] = 0$. Indeed, letting $\si^A = e^{V^A}$ and $\si^B = e^{V^B}$, we have that $[H, \log \si] = 0$ with \eqref{eq:hamil} gives 
 \begin{align*}
    \left[V^A, H^A\right] \otimes \mi^B + \mi^A \otimes \left[V^B, H^B \right] + \left[V^A \otimes \mi^B + \mi^A \otimes V^B, H^{AB} \right] = 0\,,
 \end{align*}
By taking the partial trace over the system $A$, it follows that 
\begin{align*}
    \tr_A\left(\mi^A\right) \left[V^B, H^B\right] + \left[V^B, \tr_A H^{AB}\right] =  \tr_A\left(\mi^A\right) \left[V^B, H^B\right] = 0\,,
\end{align*}
since $\tr_A ([V^A \otimes \mi^B, H^{AB}]) = 0$ and $\tr_A H^{AB} = 0$. Similarly, we have $[V^A, H^A] = 0$. Therefore, we have proved $[\si^A, H^A] = 0$ and $[\si^B, H^B] = 0$, and then $[\si, H^{AB}] = 0$ follows from $[\si, H] = 0$ immediately. The proof is complete. 
\end{proof}

We proceed to characterize Conditions~\ref{asspB} and~\ref{assump:PHP} for bipartite quantum systems, which are crucial for establishing the overdamped limit of the QMS $\exp(t\mc{L}_\gamma)$. Thanks to the $\si^A$-KMS detailed balance of $\mc{L}^A$ proved in \cref{thm:uniquefixed}, let $\mc{L}^A$ admit the spectral decomposition: 
\begin{align*}
   - \mc{L}^A Z_j^A = \lad_j Z_j^A\,, \q j \ge 0\,, 
\end{align*}
where $\lad_0 = 0$ with $Z_0^A = \mi^A$, and $\lad_j > 0$ for $j \ge 1$ with $\tr(Z_j^A) = 0$, and 
\begin{align} \label{eq:orthogonal}
    \left\l Z_k^A, Z_j^A \right\r_{\si^A,1/2} = \d_{jk}\,.
\end{align}
The eigenvectors $Z_j^A$  may not be self-adjoint, and the Hermitian-preserving property of $\mc{L}^A$ implies that $(Z_j^A)^*$ are also eigenvectors:
$$- \mc{L}^A (Z_j^A)^* = \lad_j (Z_j^A)^*\,. $$ 
Moreover, the orthonormality persists:
\begin{align*}
    \left\langle (Z_k^{A})^*, (Z_j^{A})^* \right\rangle_{\si^A,1/2} 
    &= \tr\left( Z_k^{A} \sqrt{\si^{A}} (Z_j^{A})^* \sqrt{\si^{A}} \right) \\
    &= \left\langle Z_j^{A}, Z_k^{A} \right\rangle_{\si^A,1/2} = \delta_{jk}\,.
\end{align*}
\begin{remark}
 Suppose that $\mc{L}^A$ is GNS detailed balanced with respect to $\si^A$, and let $\{Z_j^A\}$ be orthonormal with respect to GNS inner product, i.e., $\l Z_k^A, Z_j^A \r_{\si^A,1} = \d_{jk}$. In general, the orthonormality of $\{Z_j^A\}$ does not extend to $\{(Z_j^A)^*\}$, that is, $\l (Z_k^A)^*, (Z_j^A)^* \r_{\si^A,1} \neq c_j \d_{jk}$ for $c_j > 0$. Here we provide a very simple example. Define
\begin{align*}
    \sigma = \begin{pmatrix}
        \frac{1}{3} & 0 \\
        0 & \frac{2}{3}
        \end{pmatrix}, \quad
        Z_1 = \begin{pmatrix}
        0 & 1 \\
        1 & 0
        \end{pmatrix}, \quad
        Z_2 = \sqrt{2}\begin{pmatrix}
        0 & \frac{1}{2} \\
        - 1 & 0
        \end{pmatrix}.
\end{align*}
It is direct to check that $\tr(Z_1 \si) = 0 = \tr(Z_2 \si)$ and $\l Z_k, Z_j\r_{\si,1} = \d_{kj}$, while 
\begin{align*}
    \l Z_1^*, Z_2^*\r_{\si,1} = - \frac{\sqrt{2}}{2} = \l Z_2^*, Z_1^*\r_{\si,1}\,, \q \norm{Z_1^*}_{\si,1}^2 = 1\,,\q  \norm{Z_2^*}_{\si,1}^2 = \frac{3}{2}\,.
\end{align*}
\end{remark}



We next expand the interacting Hamiltonian with respect to $\{Z_j^A\}$: 
\begin{equation} \label{eq:decomHab}
    H^{AB} = \sum_{j \ge 0} Z_j^A \otimes G_j^B\,,
\end{equation}
where $\{G_j^B\}_{j \ge 0} \subset \mc{B}(\maf{H}_B)$. This allows us to characterize \cref{assump:PHP} as follows. 

\begin{conditionnp}{\ref{assump:PHP}'}{Reformulation of \cref{assump:PHP} under \cref{asspA2}} \label{condeq:php}
    For $H \in \mc{B}(\maf{H}_A \otimes \maf{H}_B)$ and $\{G_j^B\}_{j \ge 0} \subset \mc{B}(\maf{H}_B)$ given in \eqref{eq:hamil} and \eqref{eq:decomHab}, respectively, it holds that 
        \begin{align*}
            \left[H^B + G_0^B, X^B\right] = 0\,,\ \forall X^B \in \mc{M}_B\,.
        \end{align*}
      i.e., $H^B + G_0^B$ is in the commutant of $\mc{M}_B$.
\end{conditionnp}

The equivalence between \cref{assump:PHP} and \cref{condeq:php} is established by the following lemma.

\begin{lemma} \label{lem:eqcondc}
Under the assumptions of \cref{thm:uniquefixed},
it holds that 
\begin{equation} \label{auxeqcc}
     \es X  = \mi^A \otimes \tr_A \left( (\si^A \otimes \mi^B) X\right), \q \text{for}\ X \in \mc{M}_{AB}\,,
\end{equation}
and $\es \mc{L}_H \es = 0$ is equivalent to that $H^B + G_0^B \in \mc{M}_B'$ belongs to the commutant of $\mc{M}_B$. 
Here $\es$ is the conditional expectation from $\mc{M}_{AB}$ to $\mc{F}(\mc{L}_S)$ with respect to $\si$. 
\end{lemma}

\begin{proof}
    We first compute, for any $\mi^A \otimes Y^B \in \mc{F}(\ld)$ and $X \in \mc{M}_{AB}$,
    \begin{align} \label{auxeqaa}
    \left\l \mi^A \otimes Y^B, X \right\r_{\si, 1/2} & = \left\l \mi^A \otimes Y^B, \mi^A \otimes X^B \right\r_{\si^A \otimes \si^B, 1/2}  = \left\l Y^B, X^B \right\r_{\si^B, 1/2}\,,
    \end{align}
    where $X^B$ is defined by 
    \begin{align} \label{defxbproj}
        \mi^A \otimes X^B = \es X \in \mc{F}(\ld)\,.    
    \end{align}
    We can also compute 
    \begin{equation} \label{auxeqbb}
        \left\l \mi^A \otimes Y^B, \Gamma_{\si} X \right\r = \left\l Y^B, \tr_A(\Gamma_{\si} X) \right\r = \left\l Y^B,  \tr_A (\si^A \otimes \mi^B X) \right\r_{\si^B, 1/2}\,,
    \end{equation}
    where we have used
    \begin{align*}
        \tr_A(\Gamma_{\si} X) & = \tr_A \left( \sqrt{\si^A} \otimes \sqrt{\si^B} X \sqrt{\si^A} \otimes \sqrt{\si^B}\right)\\
        & = \sqrt{\si^B} \tr_A \left( \left(\si^A \otimes \mi^B \right) X \right) \sqrt{\si^B}\,.
    \end{align*}
    Comparing \eqref{auxeqaa} and \eqref{auxeqbb} gives \eqref{auxeqcc}. Letting $X^B$ be given as in \eqref{defxbproj}, one can then compute that $\es \mc{L}_H \es = 0$ is equivalent to  
    \begin{align*}
       0 = \tr_A \left(\left(\si^A \otimes \mi^B \right) \mc{L}_H \left(\mi^A \otimes X^B \right) \right) & = \left[H^B, X^B\right] + \tr_A\left(\left(\si^A \otimes \mi^B\right)\left[H^{AB}, \mi^A \otimes X^B \right]\right) \\
       & = \left[H^B, X^B\right] + \left[G_0^B, X^B\right]\,,
    \end{align*}
    thanks to \eqref{eq:hamil} and \eqref{eq:decomHab} with $\tr(\si^A Z_j^A) = 0$ for $j \ge 1$ and $Z_0^A = \mi^A$. 
\end{proof}

Further, \cref{asspB} can be simplified as follows, by \cref{lem:primitlgamma} below. 

\begin{conditionnp}{\ref{asspB}'}{Reformulation of \cref{asspB} under Conditions~\ref{asspA2} and~\ref{condeq:php}} \label{asspB2}
    For $G_j^B$ defined by  \eqref{eq:decomHab}, there holds $G_j^B \neq 0$ for some $j \ge 1$. 
\end{conditionnp}

\begin{lemma}\label{lem:primitlgamma}
Let $\ld$ be given in \eqref{def:lsbip} for a primitive QMS generator $\mc{L}^A$. Suppose that \cref{asspA2} holds. Then, it holds that
\begin{align*}
   \mc{F}(\ld) = \ker(\mc{L}^A) \otimes \mc{M}_B = \mi^A  \otimes \mc{M}_B\,,
\end{align*}
and 
\begin{equation*}
    \mc{F}(\mc{L}_\gamma) =  \mc{F}(\ld)\cap \{X \in \mc{M}_{AB}\,;\ [H, X] = 0\}\,.
\end{equation*}
    Moreover, suppose that \cref{condeq:php} holds. Then, 
    \begin{align*}
       \mc{F}(\mc{L}_\gamma) = \mi^A  \otimes \left\{X^B \in \mc{M}_B\,;\ \left[G_j^B, X^B\right] = 0 \ \text{for all $j \ge 1$}  \right\}\,.
    \end{align*}
    The condition $\dim \mc{F}(\ld) > \dim \mc{F}(\mc{L}_\gamma)$ is equivalent to that $\{G_j^B\}_{j \ge 1}$ are not all zero. 
\end{lemma}

\begin{proof}
    It suffices to apply \cref{lem:symanti} and find 
    \begin{align}
            \mc{F}(\mc{L}_\gamma) = \left\{X = \mi^A \otimes X^B\,;\ [H, X] = 0\,, \ X^B \in \mc{M}_B \right\}\,. 
    \end{align}
    Using the decompositions \eqref{eq:hamil} and \eqref{eq:decomHab}, along with \cref{condeq:php}, we compute:
    \begin{align*}
    0 = \left[H, \mi^A \otimes X^B\right] = \sum_{j \geq 1} Z_j^A \otimes \left[G_j^B, X^B\right]\,.
    \end{align*}
    The linear independence of $Z_j^A$ then implies $[G_j^B, X^B] = 0$ for all $j \ge 1$.  It follows that $\dim \mc{F}(\ld) > \dim \mc{F}(\mc{L}_\gamma)$ if and only if some $G_j^B$ is non-trivial. 
\end{proof}

We next focus on the characterization of \cref{assump:overdp}.  Recall from \eqref{eq:odlimit_generator} that the overdamped limit generator (denoting $\lh$ as $\mc{L}_H$ for clarity) takes the form:
\begin{equation} \label{eq:LO_abstract}
    \mc{L}_O = - (\mc{L}_H \es)^\star(-\mc{L}_S)^{-1}\mc{L}_H \es\,,
\end{equation}
which is defined on the subalgebra $ \mc{F}(\ld) = \mi^A  \otimes \mc{M}_B$ of $\mc{M}_{AB}$. For notational simplicity, in what follows, we shall identify $\exp(t \lo)$ with its induced dynamics $\exp(t \hat{\mc{L}_O})$ defined on $\mc{M}_B$ through the relation $\exp(t \lo) = \id^A \otimes \exp(t \hat{\mc{L}_O})$. 
Under \cref{assump:PHP} (equivalently, \cref{condeq:php}), we first compute $\mc{L}_O$ in an explicit form. Let $X \in \mc{M}_{AB}$ and write $\es X = \mi^A \otimes X^B \in \mc{F}(\ld)$ with $X^B \in \mc{M}_B$. We then compute 
\begin{align*}
    \mc{L}_H \es X = i[H, \es X] & =  i \mi^A \otimes \left[H^B, X^{B}\right] + i \sum_{j \ge 0} Z_j^A \otimes \left[G_j^B, X^B\right] \\
     & = i \sum_{j \ge 1} Z_j^A \otimes \left[G_j^B, X^B\right]\,,
\end{align*}
by \eqref{eq:hamil}, \eqref{eq:decomHab}, and \cref{condeq:php}. It follows that 
\begin{equation*}
    (-\ld)^{-1}  \mc{L}_H \es X =  i \sum_{j \ge 1} \frac{1}{\lad_j} Z_j^A \otimes \left[G_j^B, X^B\right]\,,
\end{equation*}
and hence, by writing $\es Y = \mi^A \otimes Y^B$ for $Y \in \mc{M}_{AB}$, 
\begin{equation}\label{aux:diriformsim}
    \begin{aligned}
        - \l Y, \mc{L}_O X\r_{\si,1/2} & =  - \left\l \mc{L}_H \es Y, (-\ld)^{-1}  \mc{L}_H \es X \right\r_{\si,1/2} \\
     & = \sum_{k \ge 1} \sum_{j \ge 1} \frac{1}{\lad_j} \left\l Z_k^A \otimes \left[G_k^B, Y^B\right], Z_j^A \otimes \left[G_j^B, X^B\right] \right\r_{\si,1/2}\,.        
    \end{aligned}
\end{equation}
Thanks to the orthogonality of $Z_j^k$ and $\si = \si^A \otimes \si^B$, we have 
\begin{align*}
    \left\l Z_k^A \otimes \left[G_k^B, Y^B \right], Z_j^A \otimes \left[G_j^B, X^B \right] \right\r_{\si,1/2} & = \left\l Z_k^A, Z_j^A \right\r_{\si^A,1/2}  \left\l \left[G_k^B, Y^B\right],  \left[G_j^B, X^B \right] \right\r_{\si^B,1/2} \\
    & = \d_{jk}  \left\l \left[G_k^B, Y^B \right],  \left[G_j^B, X^B \right] \right\r_{\si^B,1/2}\,,
\end{align*}
and then one can simplify \eqref{aux:diriformsim} as: 
\begin{align} \label{eq:dirichletlo}
    - \l Y, \mc{L}_O X\r_{\si,1/2} 
    & =  \sum_{j \ge 1} \frac{1}{\lad_j}  \left\l \left[G_j^B, Y^B \right],  \left[G_j^B, X^B\right] \right\r_{\si^B,1/2}\,.        
\end{align}
We summarize the above discussion into the following lemma, and then \cref{assump:overdp} can be simplified to \cref{assump:overdp2} below. 

\begin{lemma}\label{lem:overdam}
    Let $\mc{L}_\gamma = \mc{L}_H + \gamma \ld$, with $\mc{L}_H$ in \eqref{Hamiltonianterm} and $\ld$ in \eqref{def:lsbip}, be the QMS generator given as above. Suppose that Conditions~\ref{asspA2}, \ref{asspB2}, and~\ref{condeq:php} hold. Then, we have 
\begin{enumerate}[label=(\roman*)]
    \item  The generator $\lo$ of the overdamped limit of $\exp(t \mc{L}_\gamma)$ is given by
    \begin{align} \label{logenerator}
        \lo(X^B) = - \sum_{j \ge 1} \frac{1}{\lad_j} \Gamma_{\si^B}^{-1}\left[(G_j^B)^*, \Gamma_{\si^B}\left(\left[G_j^B, X^B \right]\right) \right],
    \end{align}
    where the operators $G_j^B$ are from the decomposition \eqref{eq:decomHab}, and $\Gamma_{\si^B}$ is defined as in \eqref{def:weighting}.  
    \item $\lo$ is self-adjoint with respect to the KMS inner product $\l \dd, \dd \r_{\si^B,1/2}$ on $\mc{M}_B$ with invariant state $\si^B$. The fixed-point subspace of $\exp(t \lo)$ is 
    \begin{align} \label{lokernelany}
        \ker(\lo) = \left\{X^B \in \mc{M}_B\,;\ \left[G_j^B, X^B\right] = 0 \ \text{for all $j \ge 1$}  \right\}\,.
    \end{align}
    Thus, $\si^B$ is the unique invariant state of $\exp(t \lo^\dag)$ if and only if 
    \begin{align}  \label{lokernel}
        \ker(\lo) = {\rm Span}\{\mi^B\}\,.
    \end{align}
    \item The fixed-point algebra of $\exp(t \mc{L}_\gamma)$ satisfies
    \begin{align} \label{kernelff}
        \mc{F}(\mc{L}_\gamma) = \mi^A \otimes \ker(\lo)\,.
    \end{align}
    Thus, $\mc{L}_\gamma$ is primitive if and only if \cref{lokernel} holds.
\end{enumerate}     
\end{lemma}

\begin{proof}
    The sesquilinear form \eqref{eq:dirichletlo} fully determines the generator $\lo$ and directly implies the formulation \eqref{logenerator}. Since the form \eqref{eq:dirichletlo} is symmetric, it follows that $\lo$ is self-adjoint. From \eqref{logenerator}, it is straightforward to verify that $\lo^\dag(\si^B) = 0$, implying that $\si^B$ is an invariant state of $\exp(t \lo^\dag)$. The uniqueness of the invariant state, i.e., $\dim \ker(\lo^\dag) = 1$, is equivalent to $\dim \ker(\lo) = 1$, due to $\dim \ker(\lo^\dag) = \dim \ker(\lo)$. Moreover, $\lo(X^B) = 0$ if and only if $-\l X^B, \lo X^B \r_{\si^B,1/2} = 0$, which is further equivalent to $[G_j^B, X^B] = 0$ for all $j$. Finally, \cref{kernelff} is by \cref{lem:primitlgamma} and \eqref{lokernelany}. Further, $\mc{L}_\gamma$ is primitive if and only if $\ker(\mc{L}_\gamma) = {\rm Span}\{\mi\}$, that is, \eqref{lokernel} holds. 
\end{proof}

It is convenient to introduce the noncommutative partial derivative by
\begin{equation} \label{eq:partialderivative}
    \p_j (X^B) = \left[ \lad_j^{-1/2} G_j^B, X^B\right]\q \text{with}\q  \p_j^\dag (X^B) = \left[\lad_j^{-1/2} (G_j^B)^*, X^B\right]\,.
\end{equation}
Recall the formulation of the adjoint operator with respect to the KMS inner product \eqref{eq:adjointkms}. We can derive the adjoint of $\p_j$ with respect to $\l \dd, \dd  \r_{\si^B,1/2}$:  
\begin{align} \label{eq:adjpjkms}
    \p_{j}^\star (X^B) = \Gamma_{\si^B}^{-1} \circ \p_j^\dag \circ \Gamma_{\si^B} (X^B) = \Gamma_{\si^B}^{-1} \left[\lad_j^{-1/2} (G_j^B)^*, \Gamma_{\si^B} X^B \right]\,,
\end{align}
and then $\lo$ in \eqref{logenerator} can be readily written into the following compact form: 
\begin{align} \label{eq:genlocp}
    \lo (X^B) = - \sum_{j \ge 1} \p_{j}^\star \p_j (X^B)\,.
\end{align}
 
\begin{conditionnp}{\ref{assump:overdp}'}{Simplified \cref{assump:overdp}}\label{assump:overdp2}
    The generator $\lo$ of the overdamped limit \eqref{eq:genlocp} generates a quantum Markov semigroup.
\end{conditionnp}


To ensure that \cref{assump:overdp2} holds, we need to identify additional constraints on the operators $\{G_j^B\}_{j \geq 1}$. 
It was shown in \cite{vernooij2023derivations}*{Theorem 2.5} that for any $\si$-KMS detailed balanced QMS generator $\mc{L}$, there exist $\{V_j\}_j = \{V_j^*\}_j \subset \bh$ such that its Dirichlet form can be written as: 
\begin{equation}\label{eq:quadraticdiri}
    \mc{E}_{\mc{L}}(X,Y) = \sum_{j} \l [V_j, X], [V_j, Y] \r_{\si,1/2}\,.    
\end{equation}
However, not all such quadratic forms correspond to QMS generators; see \cites{albeverio1977dirichlet,goldstein1993beurling,cipriani1997dirichlet} for the theory of noncommutative Dirichlet forms. In the case of $\mc
M_B = \mc{B}(\maf{H}_B)$,
by a careful examination of the proof of \cite{vernooij2023derivations}*{Theorem 2.5}, we have the sufficient and necessary conditions on $\{V_j\}$ for \eqref{eq:quadraticdiri} to represent a detailed balanced QMS generator, which implies \cref{lem:kmscondd} below. Recall that any $\sigma$-KMS detailed balanced generator $\mc{L}$ admits the canonical form:
\begin{equation} \label{eq:KMS_generator}
    \mc{L}(X) = (\id + \Delta_\si^{1/2})^{-1}(\Psi(\mi)) X + X (\id + \Delta_\si^{-1/2})^{-1}(\Psi(\mi)) - \Psi(X)\,, \q \forall X \in \bh\,,
\end{equation}
where $\Psi$ is completely positive and KMS detailed balanced \cites{fagnola2007generators,amorim2021complete}. We also introduce the $\mf{W}$-transform for superoperators $\Phi$:
\begin{align} \label{defw}
    \mf{W}(\Phi) := \Delta_\si^{-1/4} \circ \Phi \circ \Delta_\si^{1/4} + \Delta_\si^{1/4} \circ \Phi \circ \Delta_\si^{-1/4}\,.
\end{align}

\begin{lemma}\label{lem:kmscondd}
Suppose that $\mc{M}_B = \mc{B}(\maf{H}_B)$. Given operators $\{G_j^B\}_{j \ge 1}$, $\lo$ in \eqref{logenerator} is a $\si^B$-KMS detailed balanced QMS generator if there holds $\big\{\lad_j^{-1/2} G_j^B \big\} = \big\{\lad_j^{-1/2} (G_j^B)^* \big\}$, and   
    $\mf{W}(\Xi)$ is completely positive, where $\Xi$ is a completely positive map defined by $\Xi(X) = \sum_j \lad_j^{-1}  (\w{G}_j^B)^* X \w{G}_j^B $ with $\w{G}_j^B = \Delta_{\si^B}^{1/4} G_j^B$. In this case, the generator $\lo$ is of the form \eqref{eq:KMS_generator} with $\Psi$ given by $\mf{W}(\Xi)$. 
\end{lemma}

The proof is almost identical to that of \cite{vernooij2023derivations}*{Theorem 2.5} and is thus omitted. One can further write $\mf{W}(\Xi)$ into the Kraus form $\mf{W}[\Xi](X) = \sum_j L_j^* X L_j$, which gives the jump operators $L_j$ in the GKSL representation \eqref{eq:lindbladform} of $\lo$. This establishes a direct correspondence between the jumps $L_j$ and the operators $G_j^B$ from the decomposition of $H^{AB}$. 

In the case of GNS detailed balance, the correspondence between the jumps of $\lo$ and $G_j^B$ in the form \eqref{eq:dirichletlo} would be much more explicit. Specifically, suppose that $\lo$ is a $\si^B$-GNS detailed balanced QMS generator on $\mc{M}_B \subset \mc{B}(\maf{H}_B)$ with $\si^B \in \mc{D}_+(\mc{M}_B)$. Then, it admits the following canonical form \cite{wirth2024christensen}*{Corollary 5.4} (see also \cites{alicki1976detailed,carlen2017gradient}):
\begin{align} \label{eq:structure}
    \lo(X^B) = \sum_{j = 1}^{J_O} \left(e^{-\omega_j/2} V_j^*[X^B, V_j] + e^{\omega_j/2}[V_j, X^B] V_j^*\right)\,, \q X^B \in \mc{M}_B\,,
\end{align}
where $\{V_j\} \subset \mc{B}(\maf{H}_B)$ are eigenvectors of $\Delta_{\si^B}$:
\begin{align} \label{eq:eigmodular}
 \Delta_{\si^B}(V_j) =  e^{-\omega_j} V_j \,,\quad \l V_j, V_k \r = c_j\d_{j,k}\,, \q \tr(V_j) = 0\,.
\end{align}
For $\ww_j = 0$, one can take the corresponding $V_j$ to be self-adjoint. Additionally, for each index $1 \le j \le J_O$, there exists a conjugate index $1 \le  j' \le J_O$ such that 
\begin{align} \label{eq:adjoint_index}
 V_j^* = V_{j'}\,,\quad \ww_j = - \ww_{j'}\,.
\end{align}
Here $\ww_j \in \R$ are called the Bohr frequencies of the generator $\lo$. Then, by \eqref{eq:structure}--\eqref{eq:adjoint_index}, a direct computation gives 
\begin{align} \label{dirigns}
    \mc{E}_{\lo}(Y^B, X^B) = \sum_{j = 1}^{J_O} \left\l \left[V_j, Y^B\right], \left[V_j, X^B \right] \right\r_{\si^B,1/2}\,,\q \forall X^B, Y^B \in \mc{M}_B\,.
\end{align}
Comparing \eqref{eq:dirichletlo} and \eqref{dirigns}, we readily find that if $G_j^B = \lad_j V_j$ and the conditions \eqref{eq:eigmodular} and \eqref{eq:adjoint_index} hold for $G_j^B$, then the symmetric non-negative bilinear form \eqref{eq:dirichletlo} defines a $\si^B$-GNS detailed balanced Lindbladian. Note that the constraint $\tr(G_j^B) = 0$ in \eqref{eq:eigmodular} always holds. Indeed, by \eqref{eq:hamil} and \eqref{eq:decomHab}, we have $\tr_B(H^{AB}) = \sum_{j \ge 0} \tr(G_j^B) Z_j^A = 0$, implying  $\tr(G_j^B)$ since $\{Z_j^A\}$ is a basis of $\mc{B}(\maf{H}_A)$. 



\subsubsection{Constructing optimal lifts}\label{sec:constructlift}
We have characterized Conditions \ref{asspA}--\ref{assump:overdp} for hypocoercive QMS on bipartite systems. Suppose that $\lo$ is a given $\si^B$-GNS detailed balanced Lindbladian on $\mc{M}_B \subset \mc{B}(\maf{H}_B)$ with $\si^B \in \mc{D}_+(\mc{M}_B)$ as a (possibly non-unique) invariant state. Without loss of generality, let $\lo$ admits the canonical form \eqref{eq:structure} with jumps $G_j^B$. We next turn to the construction of (optimal) lifts $\mc{L}_\gamma$ on the extended space $\mc{M}_{AB} = \mc{B}(\maf{H}_A) \otimes \mc{M}_B$ for $\lo$ in the sense of \cref{def:2ndlift}, where the Hilbert space $\maf{H}_A$ is to be determined. The main results of this section are \cref{thm:liftgns,thm:optimalliftgns}. 

In our construction, we take $\mc{L}^A$ as the generator of the depolarizing semigroup, defined by  
\begin{align} \label{eq:depolargen}
    \mc{L}^A(X^A) = - X^A + \tr(\si^A X^A)\mi^A\,,\q \text{for $X^A \in \mc{B}(\maf{H}_A)$\,,}
\end{align}
where $\si^A \in \mc{D}_+(\maf{H}_A)$ is a full-rank quantum state. One may view $\mc{L}^A$ as the quantum analog of the full refreshment generator in \eqref{genrefresh}. It is also clear that $\mc{L}^A$ is primitive and $\si^A$-GNS detailed balanced, with two eigenvalues $0$ and $-1$. The eigenspace of $\mc{L}^A$ associated with the eigenvalue $-1$ is given by 
\begin{align*}
    \ker(-1 - \mc{L}^A) = \left\{X^A \in \mc{B}(\maf{H}_{A})\,;\ \tr(\si^A X^A) = 0\right\}\,.
\end{align*}
Here, both $\maf{H}_A$ and $\si^A$ are to be constructed.

We next construct the Hamiltonian $H$ of the form \eqref{eq:hamil} such that Conditions~\ref{asspA2}, \ref{asspB2}, and~\ref{condeq:php} are satisfied for the hypocoercive Lindbladian: 
\begin{align*}
    \mc{L}_\gamma = \mc{L}_H + \gamma \mc{L}_S = i [H, \dd] + \mc{L}^A \otimes \id^B\,,
\end{align*}
and the given generator $\lo$ is the overdamped limit of $\mc{L}_\gamma$. This ensures that $\mc{L}_\gamma$ is a lift of $\lo$ and allows the application of \cref{thm:lower,thm:uppermatrix} to estimate the convergence of $\exp(t \mc{L}\gamma)$. Without loss of generality, we consider the following ansatz of $H$, by letting $H^A = H^B = 0$ in \eqref{eq:hamil} and $G_0^B = 0$ in \eqref{eq:decomHab}, 
\begin{align} \label{eq:ansatzh}
    H  = H^{AB} = \sum_{j \ge 1} Z_j^A \otimes G_j^B\,,
\end{align}
where $G_j^B$ are the jump operators associated with the given $\lo$, satisfying \eqref{eq:eigmodular} and \eqref{eq:adjoint_index}, and $Z_j^A$ are orthonormal eigenvectors of $\mc{L}^A$ to be determined, and satisfying the conditions:
\begin{align} \label{eq:orthozj}
\left\l Z_k^A, Z_j^A \right\r_{\si^A,1/2} = \delta_{jk}\,, \quad \tr(\si^A Z_j^A) = 0\,.
\end{align}
According to the ansatz \eqref{eq:ansatzh}, Conditions~\ref{asspB2} and~\ref{condeq:php} hold straightforwardly. 
It remains to ensure that $H$ is constructed such that \cref{asspA2} holds. Once this is established, it is also easy to see from the discussions in \cref{sec:characterlift} that the overdamped limit of $\mc{L}_\gamma$ gives $\lo$, since the eigenvalues of $-\mc{L}^A$ in \eqref{eq:depolargen} satisfy $\lad_0 = 0$ and $\lad_j = 1$ for all $j \ge 1$. Then, $\mc{L}_\gamma$ is a lift of $\lo$ with convergence guarantees as desired.

Specifically, \cref{asspA2} requires 
\begin{align} \label{eq:asspa2}
    [H, \si^A \otimes \si^B]  = 0\,,
\end{align}
which implies, by \eqref{eq:eigmodular} and \eqref{eq:ansatzh},
\begin{align*}
    \sum_{j \ge 1} Z_j^A \otimes G_j^B  =  \sum_{j \ge 1} \Delta_{\si^A} \left(Z_j^A\right)   \otimes \Delta_{\si^B} \left(G_j^B\right) = \sum_{j \ge 1} e^{-\ww_j} \Delta_{\si^A} \left(Z_j^A\right)  \otimes G_j^B\,.
\end{align*}
It follows the equivalence between \eqref{eq:asspa2} and the following constraints on $Z_j^A$: 
\begin{align} \label{constraintzj}
    \Delta_{\si^A}\left( Z_j^A \right) =  e^{\ww_j} Z_j^A\,,
\end{align}
since $\{G_j^B\}$ are linearly independent. Moreover, by \eqref{eq:adjoint_index}, the self-adjointness 
\begin{align*}
    H  = \sum_{j \ge 1} Z_j^A \otimes G_j^B = H^* = \sum_{j \ge 1} (Z_j^A)^* \otimes G_{j'}^B
\end{align*}
adds an additional constraint on $Z_j^A$:
\begin{align} \label{constraintzj2}
    (Z_j^A)^* = Z_{j'}^A\,,
\end{align}
where $j'$ is the conjugate index of $j$ from \eqref{eq:adjoint_index}. 
To summarize, finding the lifts of $\lo$ requires constructing operators $\si^A$ and $Z_j^A$ that satisfy the constraints \eqref{eq:orthozj}, \eqref{constraintzj}, and \eqref{constraintzj2}.

To achieve this, we consider the disjoint decomposition of the indices $1 \leq j \leq J_O$:
\begin{align*}
    \mc{J}_{O,0} := \{j\,; \ \ww_j = 0\}\,,\q \mc{J}_{O,+} := \{j\,; \ \ww_j > 0\}\,,\q  \mc{J}_{O,-} := \{j\,; \ \ww_j < 0\}\,,
\end{align*}
where $\ww_j$ are the Bohr frequencies of $\lo$ in \eqref{eq:structure}. By the condition \eqref{eq:adjoint_index}, the index set $ \mc{J}_{O,-}$ consists of the conjugate indices of $ \mc{J}_{O,+}$, i.e., $\mc{J}_{O,-} = \{j'\,; \ j \in  \mc{J}_{O,+}\}$. Without loss of generality, one can assume that there exist integers $J_0$ and $J_+$ such that 
\begin{align*}
    \mc{J}_{O,0} = \{1 \le j \le J_0\}\,,\q \mc{J}_{O,+} = \{J_0 + 1 \le j \le J_0 + J_+\}\,,
\end{align*}
and 
\begin{align*}
    \mc{J}_{O,-} = \{J_0 + J_+ + 1 \le j \le J_0 + 2 J_+ = J_O\}\,.
\end{align*}
Then, let $\ket{\vp_\ell}$ be orthonormal eigenbasis of $\si^A$ with eigenvalues $1 > \mu_\ell > 0$:
\begin{align*}
    \si^A \ket{\vp_\ell} = \mu_\ell \ket{\vp_\ell}\,. 
\end{align*}
We are now ready to construct $Z_j^A$ as follows.
\begin{itemize}
    \item For $j \in \mc{J}_{O,0} = \{1 \le j \le J_0\}$, the operators $G_j^B$ are self-adjoint with $j' = j$,  implying $(Z_j^A)^* = Z_j^A$. We hence define 
    \begin{align} \label{defzj1}
        Z_j^A = \frac{1}{\sqrt{\mu_{2j- 1} + \mu_{2j}}} \left(\ket{\vp_{2j}}\bra{\vp_{2j}} - \ket{\vp_{2j-1}}\bra{\vp_{2j-1}}\right).
    \end{align}
    \item For $j \in \mc{J}_{O,+} = \{J_0 + 1 \le j \le J_0 + J_+\}$ with the conjugate index $j' \in \mc{J}_{O,-}$, we define 
    \begin{align} \label{defzj2}
        Z_j^A = \frac{1}{(\mu_{2j-1}\mu_{2j})^{1/4}} \ket{\vp_{2j-1}}\bra{\vp_{2j}}\,,\q Z_{j'}^A = \frac{1}{(\mu_{2j-1}\mu_{2j})^{1/4}} \ket{\vp_{2j}}\bra{\vp_{2j-1}}\,,
    \end{align}
    where eigenvalues $\mu_{2j-1}, \mu_{2j}$ are set to satisfy $$\mu_{2j - 1}/\mu_{2j} = e^{\ww_j}\,,$$ such that \eqref{constraintzj} holds. 
\end{itemize}
It is easy to see from the above constructions \eqref{defzj1} and \eqref{defzj2} that 
\begin{align*}
    \maf{H}_A = {\rm Span}\{\ket{\vp_\ell}\,;\  1 \le \ell \le  2 (J_0 + J_+)\}\,,\q \si^A = \sum_{\ell = 1}^{2 (J_0 + J_+)} \mu_\ell \ket{\vp_\ell}\bra{\vp_\ell}\,,
\end{align*}
and the conditions \eqref{eq:orthozj}, \eqref{constraintzj}, and \eqref{constraintzj2} hold for $\{Z_j^A\}$. Therefore, we have proved:
\begin{proposition} \label{thm:liftgns}
Let $\mc{M}$ be a finite-dimensional von Neumann algebra and $\si \in \mc{D}_+(\mc{M})$ be a full-rank state. There exist second-order lifts for any given $\si$-GNS detailed balanced quantum Markov semigroup $\exp(t \mc{L}_O)$ on $\mc{M}$ in the sense of \cref{def:2ndlift}.
\end{proposition}


We next show that the above construction actually gives the optimal lifts (see \cref{def:optimallift}) in the case where $\si^B$ is the maximally mixed state, i.e., $$\si^B \propto \mi^B\,,$$
if $\mc{L}_O$ satisfies the so-called $\kappa$-intertwining condition for some $\kappa \in \R$ \cites{carlen2017gradient,wirth2021curvature}.

\begin{definition} [Intertwining condition] \label{def:intercondi}
    Suppose that $\mc{P}_t$ is a $\si$-KMS detailed balanced with generator $\mc{L}$. There exists operators $\{V_j\} = \{V_j^*\}$ such that
    \begin{align*}
        \mc{L} = - \sum_{j \in \mc{J}} \p_j^\star \p_j 
    \end{align*}
    where $\p_j(\dd) := [V_j, \dd]$ with the adjoint
    $\p_j^\star$ with respect to the KMS inner product. We say $\mc{P}_t$ satisfies the $\kappa$-intertwining condition for some $\kappa \in \mathbb{R}$ if for any $j \in \mc{J}$,
    \begin{equation} \label{def:intertwining}
        \p_j \mc{P}_t = e^{-\kappa t} \mc{P}_t \p_j\,,\q \text{equivalently, }\    \partial_j \mathcal{L} = \mathcal{L} \partial_j - \kappa \partial_j\,. 
    \end{equation}
    \end{definition}

Such an intertwining condition \eqref{def:intertwining} was first used by Carlen and Maas \cites{carlen2017gradient,carlen2020non} for proving entropic Ricci curvature lower bounds and the geodesic convexity of the relative entropy in the quantum Wasserstein space, and it was further developed in \cites{wirth2021curvature,wirth2021complete,munch2024intertwining} for  dimension-dependent quantum
functional inequalities and gradient estimates. 

    \begin{lemma}\label{lem:intertcond}
        Suppose that $\mc{P}_t$ is a $\si$-KMS detailed balanced with generator $\mc{L} = - \sum_{j \in \mc{J}} \p_j^\star \p_j$ that satisfies the $\kappa$-intertwining condition for $\kappa \in \R$. Then there holds 
        \begin{align*}
            \sum_{i,j}\norm{\p_i \p_j X}_{\si,1/2}^2 \le \max\{0,-\kappa\} \mc{E}_{\mc{L}}(X)  + \norm{\mc{L} X}_{\si,1/2}^2\,.
        \end{align*}
    \end{lemma}
    \begin{proof}
    The proof follows from a direct computation:
    \begin{align*}
        \sum_{i,j}\norm{\p_i \p_j X}_{\si,1/2}^2  & = \sum_{i,j} \l \p_i \p_j X, \p_i \p_j X  \r_{\si,1/2} \\
        & = \sum_{i,j} \l \p_{i}^\star \p_i \p_j X, \p_j X  \r_{\si,1/2} = \sum_{j} \l - \mc{L} \p_j X,  \p_j X  \r_{\si,1/2}\\
        &  = \sum_{j} \l (- \kappa \p_j -  \p_j \mc{L}) X,  \p_j X  \r_{\si,1/2} \\
        &  = - \kappa \mc{E}_{\mc{L}_\si}(X) + \norm{\mc{L}_\si X}_{\si,1/2}^2\,. \qedhere
    \end{align*}
    \end{proof}


Following the discussion at the end of \cref{sec:upperlift}, to make $\mc{L}_\gamma$ become an optimal lift of $\lo$, one needs to ensure that the constants $K_1$ and $K_2$ in \eqref{eq:constk1k2mm} satisfy the condition \eqref{eq:keyest}. 
The remaining task now is to estimate $K_1$ and $K_2$ for our construction. For this, noting that in the case of $\si^B \propto \mi^B$, it holds that $\ww_j = 0$ and $G_j^B$ are self-adjoint for all $j$. Moreover, without loss of generality, we let $\mu_\ell = (2 J_0)^{-1}$ for $1 \le  \ell \le 2J_0$, and then have  
\begin{align} \label{eq:sigmasym}
    \maf{H}_A = {\rm Span}\{\ket{\vp_\ell}\,;\  1 \le \ell \le  2 J_0 \}\,,\q \si^A = \frac{1}{2 J_0} \sum_{\ell = 1}^{2 J_0}  \ket{\vp_\ell}\bra{\vp_\ell}\,.
\end{align}
We now compute 
\begin{equation} \label{auxeq:k1k2}
    \begin{aligned}
        \lh^\dag (-\ld)^{-1} \lh Y  =  \sum_{i, j \ge 1}  & [Z_i^A \otimes G_i^B, Z_j^A \otimes [G_j^B, Y^B]] \\
        = \sum_{i, j \ge 1} &  Z_i^A Z_j^A \otimes [G_i^B, [G_j^B, Y^B]] \\
        & +  \left(Z_i^A Z_j^A - Z_j^A Z_i^A \right) \otimes [G_j^B, Y^B] G_i^B\,,
    \end{aligned}
\end{equation}
for $Y = \mi^A \otimes Y^B \in \mc{F}(\mc{L}_S)$. 
It is easy to see that the second term in \eqref{auxeq:k1k2} vanishes:
\begin{align*}
   [Z_i^A, Z_j^A] \otimes \p_j (Y^B) G_i^B = 0\,,
\end{align*}
since $\mc{J}_{O,+} = \mc{J}_{O,-} = \emptyset$, and   
$[Z_i^A, Z_j^A] = 0$ by \eqref{defzj1}. It suffices to estimate the first term in \eqref{auxeq:k1k2}. By \cref{lem:intertcond}, we have 
\begin{equation} \label{eqq:optimal}
    \begin{aligned}
        \norm{\lh^\dag (-\ld)^{-1} \lh Y}_{\si,1/2}^2 & =  \Big\|\sum_{i, j \ge 1}  Z_i^A Z_j^A \otimes \p_i \p_j (Y^B) \Big\|_{\si,1/2}^2 \\ & = \sum_{k,l \ge 1} \sum_{i,j \ge 1} \left\l Z_k^A Z_l^A, Z_i^A Z_j^A \right\r_{\si^A,1/2} \left\l \p_k \p_l (Y^B),  \p_i \p_j (Y^B) \right\r_{\si^B,1/2}\\ & \le \sum_{j \ge 1} \norm{Z_{j'}^A Z_j^A}_{\si^A,1/2}^2 \norm{\p_{j'} \p_j (Y^B)}_{\si^B,1/2}^2 \\
        & \le J_0\left( \norm{\mc{L} X}_{\si,1/2}^2 + \max\{0,-\kappa\} \mc{E}_{\mc{L}}(X) \right)\,.
    \end{aligned}
\end{equation}
where we have used, thanks to the constructions \eqref{defzj1} and \eqref{defzj2},
\begin{align*}
    \l Z_k^A Z_l^A, Z_i^A Z_j^A \r_{\si^A,1/2} \propto \d_{ki} \d_{l j} \d_{j'i}\,,
\end{align*}
and 
\begin{align*}
    \norm{Z_{j'}^A Z_j^A}_{\si^A,1/2}^2 = \begin{cases}
        \frac{1}{\mu_{2j-1} + \mu_{2j}} & j \in \mc{J}_{O,0} \\
        \frac{1}{\mu_{2j-1}} & j \in \mc{J}_{O,+} \\
        \frac{1}{\mu_{2j'}} & j \in \mc{J}_{O,-} 
    \end{cases}\,.
\end{align*}
with $\mc{J}_{O,+} = \mc{J}_{O,-} = \emptyset$ and $\mu_{\ell} = (2 J_0)^{-1}$ for all $ \ell $ by \eqref{eq:sigmasym}. 
It follows from \eqref{eqq:optimal} that 
\begin{align*}  
    \norm{(\id - \es)\lh^\dag (-\ld)^{-1} \lh Y}_{2,\si} 
  \le \sqrt{J_0} \left( \norm{\mc{L} X}_{\si,1/2} + \sqrt{\max\{0,-\kappa\} \mc{E}_{\mc{L}}(X)} \right)\,,
\end{align*}
that is, 
\begin{equation*}
    K_1 = \sqrt{J_0}\,,\q K_2 = \sqrt{J_0 \max\{0,-\kappa\}}\,.
\end{equation*}
According to \eqref{eq:rateestmatrix2} and \eqref{eq:keyest}, as well as \cref{thm:lower}, we can conclude:
\begin{theorem} \label{thm:optimalliftgns}
Let $\mc{M}$ be a finite-dimensional von Neumann algebra. For a symmetric QMS generator $\lo$ satisfying the $\kappa$-intertwining condition for $\kappa \in \R$, there exists a second-order lift $\mc{L}$ with convergence rate:
    \begin{align*}
         \Omega\left(\frac{\sqrt{\lad_O}}{  1 +  \sqrt{\max\{0,-\kappa\}\lad_O^{-1}}}\right) \le \nu(\mc{L}) \le \mc{O}\left(\sqrt{\lad_O}\right)\,.
    \end{align*}
    Such lift is optimal if $\max\{0,-\kappa\}\lad_O^{-1} = \mc{O}(1)$, which always holds if $\kappa \ge 0$. 
\end{theorem} 

\begin{remark}
  For a general full-rank state $\sigma$, it holds that $\mathcal{J}_{O,+} \neq \emptyset$ and  $[Z_i^A, Z_j^A] \neq 0$ for some pairs of $(i,j)$. Thus, estimating the second term in \eqref{auxeq:k1k2} is necessary to obtain the desired constants $K_1$ and $K_2$. This task forms the main challenge in extending \cref{thm:liftgns} to general  $\sigma$-GNS detailed balanced  $\mathcal{L}_O$, beyond the case where  $\sigma \propto \mi$. Addressing this challenge may require a more refined or clever construction of the lifted dynamics. 
\end{remark}

We conclude this section by presenting a few simple examples of symmetric QMS that satisfy the intertwining conditions in \eqref{def:intertwining}. 

\begin{example}[Depolarizing semigroups]
Let $\mc{N}$ be a von Neumann subalgebra of $\mc{M}$ and $E_{\mc{N}}$ be the associated conditional expectation. The general depolarizing semigroup is generated by $\mc{L} = E_{\mc{N}} - \id$. Suppose that it can be written as $\mc{L} = - \sum_j \p_j^\dag \p_j$. Then, since $\mc{L} E_{\mc{N}} = 0$, we have $\p_j E_{\mc{N}} = 0 = E_{\mc{N}} \p_j$ for all $j$. It follows that $\p_j \mc{L} = -\p_j = \mc{L} \p_j$, that is, $\mc{L}$ satisfies the $\kappa$-intertwining condition with $\kappa = 0$. 
\end{example}

\begin{example}[Schur multipliers]
A Schur multiplier $\mc{L}$ over the matrix algebra $\mc{B}(\C^n)$ is a linear operator defined by $\mc{L} E_{ij} = a_{ij} E_{ij}$, where $a_{ij} \in \C$ and $\{E_{ij}\}$ are the matrix units. Then $\exp(t \mc{L})$ defines a symmetric QMS if and only if there exists a Hilbert space $\mc{H}$ and elements $\{a(j)\}_{j = 1}^n \in \mc{H}$ such that $a_{ij} = - \norm{a(i) - a(j)}^2$. Then, let $\{e_k\}_{k = 1}^d$ be  an orthonormal basis of $\mc{H}$. Define 
\begin{align*}
    V_k = \sum_{j = 1}^n \l a(j), e_k \r E_{jj}\,.
\end{align*}
which are self-adjoint. By a direct computation, one can check that, for $X \in \mc{B}(\C^n)$,
\begin{align*}
    \mc{L}(X) = - \sum_{k = 1}^d \p_k^2 (X)\q \text{with} \q \p_k(X) = [V_k, X]\,.
\end{align*}
and further, $\p_j \mc{L} = \mc{L} \p_j$, implying that $\mc{L}$ satisfies the $\kappa$-intertwining condition for $\kappa = 0$.     
\end{example}

For more interesting and and complicated examples, we refer interested readers to recent works \cites{carlen2017gradient, wirth2021curvature, wirth2021complete, gao2024coarse}, which include cases such as Fermi and Bose Ornstein-Uhlenbeck semigroups \cite{carlen2017gradient} and quantum Markov semigroups on discrete group von Neumann algebras \cite{wirth2021curvature}*{Section 6.2} (see also \cite{gao2024coarse}*{Section 5.2}). 

\section{Concluding remarks}
Inspired by \cites{eberle2024non,eberle2025convergence}, in this work, we have developed a second-order lifting framework for quantum Markov semigroups and uncovered its connections to the overdamped limit. This allows us to establish the convergence guarantees of lifted hypocoercive dynamics from both above and below, while also characterizing the limitations of accelerating detailed balanced quantum Markov processes via lifting. An abstract lifting framework for symmetric contraction $C_0$-semigroups has also been presented, which unifies our discussion of quantum dynamics with the results of \cites{eberle2024non,eberle2025convergence} for classical diffusions. To illustrate our theoretical findings, we have provided new examples of optimal lifts for certain symmetric finite Markov chains and quantum Markov semigroups. 
We conclude with a few intriguing open questions:  
\begin{itemize}[wide]  
    \item The second-order lift can be formally defined as the reverse of the overdamped limit, where the structural \cref{assump:PHP} is utilized. In light of \cref{rem:first_order}, we are curious whether a first-order lifting framework, similar to \cite{chen1999lifting}, exists for discrete-time quantum Markov evolutions.  
    \item As discussed, demonstrating that the lift construction in \cref{sec:liftfinitemarkov} is optimal requires verifying the matrix inequality \eqref{eq:matrixbound} alongside \eqref{keyrela}. It would be interesting to characterize the class of symmetric $Q$-matrices for which this property holds.  
    \item Following \cite{li2024quantum}*{Section 5.2}, it is not difficult to construct the lift of a general reversible finite Markov chain with respect to a non-uniform measure. However, finding the optimal lift remains a significant challenge.  
    \item Designing optimal lifts for symmetric quantum Markov semigroups without relying on the $\kappa$-intertwining condition would be very interesting. Extending this to general $\si$-detailed balanced quantum Markov dynamics is another exciting direction to explore.  
\end{itemize}  

\subsection*{Acknowledgment} This work is supported in part by 
National Key R$\&$D Program of China Grant No. 2024YFA1016000 (B.L.) and by National Science Foundation via award DMS-2309378 (J.L.).

\begin{appendix}

\section{Conditional expectation} \label{app:quantumce}

In this appendix, we briefly recall some basics of noncommutative conditional expectation. Let $\mathcal{N} \subset \mathcal{M} \subset \bh$ be finite-dimensional von Neumann subalgebras. 

\begin{definition}[\cites{ohya2004quantum,benatti2009dynamics}]
Let $\mathcal{N} \subset \mathcal{M}$ be finite-dimensional von Neumann algebras, and let $\si$ be a full-rank state on $\mc{M}$. A linear mapping $E_{\mc{N}}: \mc{M} \to \mc{N}$ is a \emph{conditional expectation} if it is completely positive and unital, and there holds 
    \begin{align} \label{cecdon1}
        E_{\mc{N}}(AXB) = A E_{\mc{N}}(X) B, \quad \forall X \in \mc{M},\ A, B \in \mc{N}\,.
    \end{align}
Further, such a mapping $E_{\mc{N}}$ is called a \emph{conditional expectation with respect to $\si$} if it is trace-preserving with respect to $\si$:
    \begin{align} \label{cecdon2}
        \tr(\si E_{\mc{N}}(X)) = \tr(\si X), \quad \forall X \in \mc{M}\,.
    \end{align}
\end{definition}

Note that the condition \eqref{cecdon2} is equivalent to the invariance $E_{\mc{N}}^\dag(\si) = \si$. Additionally, the following useful properties hold for a conditional expectation \cite{takesaki2003theory}.

\begin{lemma} \label{lem:properce}
Let $E_{\mc{N}} : \mathcal{M} \to \mathcal{N}$ a conditional expectation with respect to $\sigma$. Then $E_{\mc{N}}$ is self-adjoint with respect to the $s$-inner product $\l \dd ,\dd \r_{\si,s}$ for any $s \in [0,1]$, where the inner product $\l \dd ,\dd \r_{\si,s}$ is defined as in \eqref{def:s-inner}. It follows that $E_{\mc{N}}$ commutes with the modular automorphism group:
\begin{align*}
    \Delta_\si^{is} \circ E_{\mc{N}} = E_{\mc{N}} \circ  \Delta_\si^{is}\,,\q \forall s \in \R\,,
\end{align*}
where $\Delta_\si$ is the modular operator in \eqref{def:modularoperator}.
\end{lemma}

\begin{corollary}[Uniqueness]\label{coro:uniquence}
    Let $E_{\mc{N}} : \mathcal{M} \to \mathcal{N}$ be a conditional expectation with respect to $\sigma$. Then $E_{\mc{N}}$ is the orthogonal projection to $\mc{N}$ with respect to the $L^2$ norms $\norm{\dd}_{\si,s}$ for any $s \in [0,1]$. Consequently, $E_{\mc{N}}$ is uniquely determined by $\mc{N}$ and $\sigma$.
\end{corollary}
\begin{proof}
For a full-rank state $\si$ and any \(X \in \mathcal{M}\), there holds 
\begin{align*}
     \langle E_{\mathcal{N}}(X), X - E_{\mathcal{N}}(X) \rangle_{\sigma, s}
= & \langle \Gamma_{\sigma, s} \circ E_{\mathcal{N}}(X), X - E_{\mathcal{N}}(X) \rangle \\
= & \langle E_{\mathcal{N}}^\dag \circ \Gamma_{\sigma, s}(X), X - E_{\mathcal{N}}(X) \rangle \\
= & \langle \Gamma_{\sigma, s}(X), E_{\mathcal{N}}(X - E_{\mathcal{N}}(X)) \rangle = 0\,, 
\end{align*}
where $\Gamma_{\si,s}(X) := \si^{1 - s} X \si^s$. Here in the second equality, we have used the self-adjointness of $E_{\mc{N}}$ with respect to $\l \dd ,\dd \r_{\si,s}$, while the third equality is by $E_{\mc{N}}(X) =X$ for any $X \in \mc{N}$. 
\end{proof}

The existence of conditional expectation $E_{\mc{N}}$ with respect to a full-rank state $\si$ is not always guaranteed. The following lemma, by \cite{takesaki1972conditional}*{Theorem IX.4.2} (see also \cite{takesaki2003theory}), provides the necessary and sufficient conditions to ensure the existence of $E_{\mc{N}}$. 

\begin{lemma}\label{lem:extencece}
The existence of conditional expectation $E_{\mc{N}}: \mathcal{M} \to \mathcal{N}$ with respect to a full-rank state $\sigma$ is equivalent to the invariance of $\mc{N}$ 
under the modular automorphism group generated by $\si$: $\Delta_\si^{is} (\mc{N}) = \mc{N}$ for any $s \in \R$.
\end{lemma}

\section{\texorpdfstring{Auxiliary proofs of \cref{sec:liftingmatrix}}{Auxiliary proofs of Section 2}} \label{appA}

\begin{proof}[Proof of \cref{lem:relatingmix}]
    We recall, for $\rho \in \mc{D}(\mc{M})$ and $\si \in \mc{D}_+(\mc{M})$, 
        \begin{align*}
            \norm{\rho - \si}_{\tr}^2 \le \tr[(\rho - \si) \Gamma_{\si}^{-1}(\rho - \si)]\,,
        \end{align*}
        which implies, by letting $ X := \Gamma_{\si}^{-1} \rho$ and $\mc{P}_t^\star = \Gamma_{\si}^{-1} \circ \mc{P}^\dag \circ \Gamma_{\si}$,
        \begin{align*}
            \norm{\mc{P}_t^\dag(\rho) - \si}_{\tr} \le \norm{\mc{P}_t^\star X - \mi}_{2,\si} = \sup_{Y \in \mc{M}\backslash\{0\}} \frac{\l Y, \mc{P}_t^\star X - \mi \r_{\si,1/2}}{\norm{Y}_{2,\si}}\,.
        \end{align*}
        We compute 
        \begin{align*}
            \l Y, \mc{P}_t^\star X - \mi \r_{\si,1/2} = \l \mc{P}_t Y,  X  \r_{\si,1/2} - \tr(\si Y)  =  \l \mc{P}_t Y - \tr(\si Y) \mi,  X  \r_{\si,1/2}\,.
        \end{align*}
        It follows that when $t \ge t_{\rm rel} (2 - \frac{1}{2}\log(\si_{\min}))$, 
        \begin{align*}
            \norm{\mc{P}_t^\dag(\rho) - \si}_{\tr} & \le \sup_{Y \in \mc{M}\backslash\{0\}} \frac{ \l \mc{P}_t Y - \tr(\si Y) \mi,  X  \r_{\si,1/2}}{\norm{Y - \tr(\si Y) \mi}_{2,\si}} \\
                & \le \si_{\min}^{-1/2} e^{- \lfloor t/t_{\rm rel} \rfloor} \le e^{-1}\,,
        \end{align*}
        where the second inequality is by $\sup_{\rho \in \mc{D}(\mc{M})}\norm{\Gamma_{\si}^{-1} \rho}_{2,\si}^2 = \si_{\min}^{-1}$. 
    \end{proof}

    \begin{proof}[Proof of \cref{prop:overdampedmx}]
    We write $ X^\veps(t) = X_0(t) + \veps X_1(t) + \veps^2 X_2(t) + r(t)$ and plug it into \eqref{eq:asym2}. It follows that, by \eqref{eq:0order}-\eqref{eq:norder},
    \begin{align*}
        \frac{\rd}{\rd t} r(t) = & \frac{\rd}{\rd t} (X^\veps(t) - X_0(t) - \veps X_1(t) - \veps^2 X_2(t)) \\
            =  & (\veps^{-1}\lh + \veps^{-2} \ld) r(t) + \lh X_1(t) + \ep \lh X_2(t) + \ld X_2(t)  \\
        & - \frac{\rd}{\rd t} X_0(t) - \veps  \frac{\rd}{\rd t} X_1(t) - \veps^2 \frac{\rd}{\rd t} X_2(t) \\
        = & (\veps^{-1}\lh + \veps^{-2} \ld) r(t) + \ep \underbrace{\left(\lh X_2(t) -  \frac{\rd}{\rd t} X_1(t) - \veps \frac{\rd}{\rd t} X_2(t)\right)}_{q(t)}\,.
    \end{align*}
    By the variation-of-constants formula, we derive 
    \begin{align*}
        r(t) = e^{t (\veps^{-1}\lh + \veps^{-2} \ld)} r(t) + \ep \int_0^t e^{(t-s)(\veps^{-1}\lh + \veps^{-2} \ld)} q(s) \ud s\,.
    \end{align*}
    Noting $e^{t (\veps^{-1}\lh + \veps^{-2} \ld)}$ is a contraction semigroup for $\norm{\dd}_{2,\si}$ and $r(0) = - \ep X_1(0) - \ep^2 X_2(0)$, we conclude, for a fixed $t > 0$, 
      \begin{align*}
        \norm{r(t)}_{2,\si} & \le \ep \left(\norm{X_1(0) + \ep X_2(0)}_{2,\si} + t \sup_{0 \le s \le t} \norm{q(t)}_{2,\si}\right)\,.  \qedhere
      \end{align*}
\end{proof}

\begin{proof}[Proof of \cref{thm:lower}]
It suffices to apply \cref{lem:singulargap} to $\mc{L}_\gamma$ and estimate $s(\mc{L}_\gamma)$ from above in terms of $\lad(\mc{L}_O)$. For this, we note
\begin{equation} \label{eq:lowerp1}
    \begin{aligned}
        s(\mc{L}_\gamma) = \inf_{X \in \mc{M} \backslash\{0\}} \frac{\norm{\mc{L}_\gamma X}_{2,\si}}{\norm{X - E_{\mc{F}}X}_{2,\si}} &\le \inf_{X \in \mc{M} \backslash\{0\}} \frac{\norm{\mc{L}_\gamma \es X}_{2,\si}}{\norm{\es X - E_O X}_{2,\si}} \\ &= \inf_{X \in \mc{M} \backslash\{0\}} \frac{\norm{\lh \es X}_{2,\si}}{\norm{\es X - E_O X}_{2,\si}}\,,
    \end{aligned}
\end{equation}
where in the first inequality, we have used, by \cref{lem:symanti,lem:proplo},
\begin{equation*}
    E_S E_{\mc{F}} = E_{\mc{F}} = E_O\,,
\end{equation*}
and the second equality is by $\ld \es = 0$.  Then, thanks to \cref{assump:PHP}, there holds 
        \begin{align*}
            \w{\mc{S}}^{-1/2} \w{\mc{S}}^{1/2} \lh \es = \lh \es\,,
        \end{align*}
       with $\w{\mc{S}}: = \Pi_1 (-  \ld )^{-1} \Pi_1 $ 
        and thus 
        \begin{equation}\label{eq:lowerp2}
            \begin{aligned}
                \norm{\lh \es X}_{2,\si} & = \norm{\w{\mc{S}}^{-1/2} \w{\mc{S}}^{1/2} \lh \es X}_{2,\si} \\
                & \le \norm{\w{\mc{S}}^{-1/2}}_{(2,\si) \to (2,\si)}  \norm{\w{\mc{S}}^{1/2} \lh \es X}_{2,\si}\,.
            \end{aligned}
        \end{equation}
       We also have 
        \begin{align*}
            \norm{\w{\mc{S}}^{1/2} \lh \es X}_{2,\si}^2 = \l X,  (\lh \es)^\star (-\ld)^{-1} \lh \es X\r_{\si,1/2} = \mc{E}_{\lo}(X,X)\,.
        \end{align*}
        It then follows from \eqref{eq:lowerp1} and \eqref{eq:lowerp2} that 
        \begin{align*}
            s(\mc{L}_\gamma) &\le \inf_{X \in \mc{M} \backslash\{0\}} \frac{\norm{\lh \es X}_{2,\si}}{\norm{\es X -E_O X}_{2,\si}} \\
                & \le \sqrt{\w{s_{\rm m}}^{-1}}  \inf_{X \in \mc{F}(\ld) \backslash\{0\}} \frac{\sqrt{\mc{E}_{\lo}(X,X)}}{\norm{X - E_O X}_{2,\si}} \\
                & = \sqrt{\w{s_{\rm m}}^{-1} \lad_O}\,, 
        \end{align*}
        where the last step is by \eqref{eq:spgaplado}. The proof is completed by \cref{lem:singulargap}.
        \end{proof}

\section{\texorpdfstring{Proof of Divergence \cref{divergencelemma}}{Proof of Divergence Lemma}} \label{app:diverlema}

\begin{proof}[Proof of \cref{divergencelemma}]
The proof is generalized from \cite{eberle2024space}*{Theorem 5} and \cite{eberle2025convergence}*{Theorem 15}, as well as \cite{li2024quantum}*{Lemma 4.7}. 
Since $\fll$ is negative self-adjoint with discrete spectrum, there exists an orthonormal basis $\{e_k\}_{k \ge 0}$ of $\mc{H}_O$ such that $\fll e_k  = - \mu_k^2 e_k$ with $\mu_k \ge 0$ and $\mu_k^2$ arranged in an increasing order, 
where $\mu_k = 0$ and $e_k \in \ker(\fll)$ for $0 \le k \le K_0 = \dim(\ker(\fll))$. The spectral gap $\lad_O$ of $\fll$ can be then characterized by $\lad_O = \inf\{\mu_k^2\,;\ k > K_0\}$. 

We introduce the time-augmented basis: 
\begin{equation}\label{divereqbasis}
\begin{aligned}
      & H_k^a = (t - (T - t)) e_k\,,\q  0 \le k \le K_0\,, \\
    & H_k^a = \big(e^{- \mu_k t} - e^{- \mu_k (T-t)}\big) e_k\,,\q  k > K_0\,, \\
    & H_k^s = \big(e^{- \mu_k t} + e^{- \mu_k (T-t)}\big) e_k\,,\q  k > K_0\,,
\end{aligned}
\end{equation}
which gives an orthogonal basis of the subspace:
\begin{align*}
\mathbb{H} = \{x_t \in L_\perp^2([0,T]; \mc{H}_O)\,; \ x_t \in H^2([0,T]; \mc{H}_0)\,,\  x_t \in \dom(\fll)\ \text{for a.e.\,$t$,}\  \partial_{tt} x_t +  \fll x_t = 0\}\,.  
\end{align*}
It allows us to further decompose $L_\perp^2([0, T]; \mc{H}_O)$ into symmetric and antisymmetric modes with low and high energies: 
\begin{align}\label{spacedecomp}
    L_\perp^2([0, T]; \mc{H}_O) = \underbrace{\mathbb{H}_{l,a} \oplus \mathbb{H}_{l,s} \oplus \mathbb{H}_{h,a} \oplus \mathbb{H}_{h,s}}_{= \mathbb{H}} \oplus \mathbb{H}^{\perp}\,,
\end{align}
where 
\begin{equation} \label{spacedecompo}
    \begin{aligned}
          & \mathbb{H}_{l,a} = {\rm Span} \left\{ H_k^a\,; \mu_k^2 \leq \frac{4}{T^2} \right\}\,,\q  \mathbb{H}_{l,s} = {\rm Span} \left\{ H_k^s\,; \mu_k^2 \leq \frac{4}{T^2} \right\}\,,\\
    & \mathbb{H}_{h,a} = \overline{{\rm Span} \left\{H_k^a\,; \mu_k^2 > \frac{4}{T^2} \right\}}\,,\q  \mathbb{H}_{h,s} = \overline{{\rm Span} \left\{ H_k^s\,; \mu_k^2 > \frac{4}{T^2} \right\}}\,.
    \end{aligned}
\end{equation}
Thanks to the orthogonal decomposition \eqref{spacedecomp}, it suffices to construct the solutions \((z_t, y_t)\) to the divergence equation \eqref{diverequation}, satisfying the desired estimates \eqref{eq:diver1}--\eqref{eq:diver2}, separately for \(x_t\) within the subspaces specified in \eqref{spacedecomp}.

\medskip 

\noindent \underline{\emph{Case I: $x_t \in \mathbb{H}^\perp$.}} We define the second-order operator $\w{\fll} = \p_{tt} + \fll$ on $L^2([0,T]; \mc{H}_O)$ with Neumann boundary condition in $t$ and the domain given by 
\begin{multline}
  \dom\left(\w{\fll}\right) = \Big\{x_t \in L^2([0,T]; \mc{H}_O)\,;\ x_t \in H^2([0,T]; \mc{H}_0)\,,\  x_t \in \dom(\fll)\ \text{for a.e.\,$t$,}\\  \p_t x_t|_{t = 0, T} = 0\,,\  \w{\fll}(x_t) \in L^2([0,T]; \mc{H}_O) \Big\}\,.     
\end{multline}
Note that 
\begin{align*}
    \ker\left(\w{\fll}\right) = \{x_t \in L^2([0,T]\,;\ 
    x_t = x_0 \text{ for a.e.\,$t$},\ x_0 \in \ker(\fll)\}\,.
\end{align*}
We denote by $\mf{P}_O$ and $\w{\mf{P}}_O$ the orthogonal projections to $\ker(\fll)$ and 
$\ker\Big(\w{\fll}\Big)$, respectively, which satisfy 
\begin{align*}
    \w{\mf{P}}_O x_t = \frac{1}{T} \int_0^T \mf{P}_O x_t \ud t\,.
\end{align*}
Then, by tensorization similarly to \cite{li2024quantum}*{Lemma 4.5}, we have the Poincar\'e inequality for $\w{\fll}$:
\begin{align*}
    \|x_t - \w{\mf{P}}_O x_t\|_{T,\mc{H}_O}^2 & =  \|x_t - \mf{P}_O x_t\|_{T,\mc{H}_O}^2 + \| \mf{P}_O x_t - \w{\mf{P}}_O x_t\|_{T,\mc{H}_O}^2 \\
    & \le \frac{1}{\lad_O} \mc{E}_{T,\fll}(x_t) +  \frac{T^2}{\pi^2} \norm{\p_t \mf{P}_O x_t}_{T, \mc{H}_O}^2 \\
    & \le \max\left\{\frac{1}{\lad_O}, \frac{T^2}{\pi^2} \right\} \l x_t, - \w{\fll} x_t\r_{T,\mc{H}_O} \,.
\end{align*}
It follows that $- \w{\fll}: \dom(\w{\fll}) \cap L^2_\perp([0,T]; \mc{H}_O) \to L^2_\perp([0,T]; \mc{H}_O)$ admits a bounded inverse:  
\begin{align*}
    \mf{G}: L^2_\perp([0,T]; \mc{H}_O) \to \dom(\w{\fll}) \cap L^2_\perp([0,T]; \mc{H}_O)
\end{align*}
with operator norm:
\begin{align*}
    \norm{\mf{G}}_{L^2_\perp([0,T]; \mc{H}_O) \to L^2_\perp([0,T]; \mc{H}_O)} \le \max\left\{\frac{1}{\lad_O}, \frac{T^2}{\pi^2} \right\}\,.
\end{align*}
We define $y_t = \mf{G} x_t$ and $z_t = - \p_t y_t$ that solves the divergence equation \eqref{diverequation} with 
$z_t|_{t = 0, T} = 0$, since $\mf{G} x_t$ satisfies Neumann boundary condition in $t$. To see that $y_t$ also satisfies Dirichlet  boundary condition, we take $w_t \in \mathbb{H}$, then the integration by parts give 
\begin{align*}
0 = \l w_t, x_t\r_{T, \mc{H}_O} = \l w_t, -\w{\fll} y_t\r_{T, \mc{H}_O} =  \l  -\w{\fll} w_t,  y_t\r_{T, \mc{H}_O} + \l \p_t w_t, y_t \r_{\mc{H}_O}|_{t = 0}^{t = T}\,,
\end{align*}
which implies $y_0 = y_T = 0$ since $\w{\fll} w_t = 0$ and $w_t$ at $t = 0, T$ are arbitrary. We next derive the estimates \eqref{eq:diver1}--\eqref{eq:diver2}. For this, following \cite{eberle2025convergence}, note that the operators $\p_{tt}$ and $\fll$ on $L^2_\perp([0,T]; \mc{H}_O) \to \dom(\w{\fll})$ commute and admit discrete spectrum. It holds that both the operator norms of $\fll \mf{G}$ and $\p_{tt} \mf{G}$ are bounded by one. This allows us to conclude that in the case of $x_t \in \mathbb{H}^\perp$, we have
\begin{align*}  
    & \norm{\fll y_t}_{T, \mc{H}_O} \le  \norm{\fll \mf{G} x_t}_{T, \mc{H}_O} \le  \norm{x_t}_{T, \mc{H}_O}\,,\\
    & \mc{E}_{T,\fll}(z_t) = \mc{E}_{T,\fll}(\p_t y_t) = \l \p_{tt} \mf{G} x_t, \fll \mf{G} x_t \r_{T, \mc{H}_O} \le \norm{x_t}_{T, \mc{H}_O}^2\,, \\
    & \mc{E}_{T,\fll}(y_t) = - \l \mf{G} x_t, \fll \mf{G} x_t \r_{T, \mc{H}_O} \le \max\left\{\frac{1}{\lad_O}, \frac{T^2}{\pi^2} \right\} \norm{x_t}_{T, \mc{H}_O}^2\,.
\end{align*}
That is, the estimates \eqref{eq:diver1}--\eqref{eq:diver2} hold with 
\begin{align} \label{case1}
    c_1 = c_2 =c_4 = 1\,,\q c_3 = \max\left\{\frac{1}{\sqrt{\lad_O}}, \frac{T}{\pi} \right\}\,.
\end{align}

\medskip 

\noindent \underline{\emph{Case II: $x_t \in \mathbb{H}_{l,a}$.}} In this case, we define $z_t = \int_0^t x_s \ud s$ and $y_t = 0$, which solves \eqref{diverequation} and satisfies $z_t|_{t = 0, T} = 0 = y_t|_{t = 0, T}$. 
We estimate 
\begin{align*}
\norm{z_t}_{T,\mc{H}_O}^2 & = \frac{1}{T} \int_0^T \left\l \int_0^t x_s \ud s,  \int_0^t x_s \ud s \right\r_{\mc{H}_0} \ud t \\
& \le \frac{1}{T} \int_0^T t \int_0^t\left\l x_s, x_s \right\r_{\mc{H}_0} \ud s \ud t \\
& \le \int_0^T t \ud t \norm{x_t}_{T,\mc{H}_0}^2 = \frac{T^2}{2} \norm{x_t}_{T,\mc{H}_0}^2\,,
\end{align*}
where in the second line we have used the convexity of $\norm{\dd}_{\mc{H}_O}^2$.  Similarly, we have 
\begin{align*}
     \norm{\fll z_t}^2_{T, \mc{H}_O} & \le \frac{1}{T} \int_0^T t \int_0^t \norm{\fll x_s}_{\mc{H}_O}^2 \ud s \ud t \\
     & \underset{\eqref{spacedecompo}}{\le}  \frac{1}{T} \int_0^T t \int_0^t \frac{16}{T^4}\norm{ x_s}_{\mc{H}_O}^2 \ud s \ud t \le \frac{8}{T^2} \norm{x_t}_{T,\mc{H}_O}^2\,.
\end{align*}
It follows that by the Cauchy-Schwarz inequality, 
\begin{align} \label{auxeqdiver}
    \mc{E}_{T, \fll}(z_t) \le 2 \norm{x_t}_{T,\mc{H}_O}^2\,.
\end{align}
Thus, we can conclude, for $x_t \in \mathbb{H}_{l,a}$,
\begin{align} \label{case2}
    c_1 = c_3 = c_4 = 0\,,\q c_2 = \sqrt{2}\,.
\end{align}

\medskip 

\noindent \underline{\emph{Case III: $x_t \in \mathbb{H}_{l,s}$.}} Following \cite{eberle2024space}*{Theorem 15}, we consider the decomposition $x_t = x_t^{(0)} + x_t^{(1)}$ with $x_t^{(0)} = x_0 \cos(2 \pi t/ T)$ and $x_t^{(1)} = x_t - x_t^{(0)}$. By definition, there holds 
\begin{align*}
    \int_0^T x_t^{(0)} \ud t = 0\,,\q  x_t^{(1)}|_{t = 0, T} =  0\,.
\end{align*}
We then define $z_t := \int_0^t x_t^{(0)} \ud s$ that satisfies the Dirichlet boundary condition on $[0,T]$, and $y_t: = \mf{G}_O x_t^{(1)}$, where $\mf{G}_O$ is defined by $(- \fll)^{-1}$ on the finite-dimensional space 
${\rm Span}\{e_k\,; 0 < \mu_k^2 \le 2/T\}$. 
One can directly check \eqref{diverequation}:
\begin{equation*}
    \p_t z_t - \fll y_t = x_t^{(0)} + x_t^{(1)} = x_t\,.
\end{equation*}
We next estimate 
\begin{equation}\label{diverauxeq3}
    \begin{aligned}
           \norm{x_t^{(0)}}_{T,\mc{H}_O}^2 = \frac{T}{2}\norm{x_0}_{\mc{H}_O}^2 & = \frac{T}{2} \sum_{0 < \mu_k \le \frac{2}{T}} b_k^2 (1 + e^{- \mu_k T})^2 \\ 
    & \le 2 T \sum_{0 < \mu_k \le \frac{2}{T}} b_k^2 \le \frac{e^2}{2}  \norm{x_t}_{T, \mc{H}_O}^2\,,
    \end{aligned}
\end{equation}
by expanding $x_t$ into the basis $H_k^s$: $x_t = \sum_{0 < \mu_k \le 2/T} b_k H_k^s$ with 
\begin{equation}\label{diverauxeq4}
    \begin{aligned}
          \norm{x_t}_{T, \mc{H}_O}^2 & = \sum_{ 0 < \mu_k \leq \frac{2}{T}} b_{k}^{2} \int_{0}^{T} \left( e^{-\mu_{k}t} + e^{-\mu_{k}(T-t)} \right)^{2} \ud t \\
    & \geq \sum_{ 0 < \mu_k \leq \frac{2}{T}}  T b_{k}^{2} \min_{t \in [0,T]} \left( e^{-\mu_{k}t} + e^{- \mu_{k}(T-t)} \right)^{2}  \\
    & = \sum_{ 0 < \mu_k \leq \frac{2}{T}} 4 e^{- \mu_k T} T b_{k}^{2} \ge 4 e^{-2} T \sum_{ 0 < \mu_k \leq \frac{2}{T}} b_{k}^{2}\,.
    \end{aligned}
\end{equation}
It readily follows that 
\begin{align} \label{auxeqdiver2}
 \norm{\fll y_t}_{T, \mc{H}_O}^2 = \norm{x_t^{(1)}}_{T,\mhh}^2 \le (1 + e/\sqrt{2})^2 \norm{x_t}_{T,\mhh}^2\,.
\end{align}
By arguments similar to \eqref{auxeqdiver}, we have 
\begin{align*}
     \mc{E}_{T, \fll}(z_t) \le 2 \norm{x_t^{(0)}}_{T,\mc{H}_O}^2 \le e^2 \norm{x_t}_{T,\mc{H}_O}^2\,.
\end{align*}
We proceed estimate, by Poincar\'e inequality of $\fll$ and \eqref{auxeqdiver2},
\begin{align*}
    \mc{E}_{T,\fll}(y_t) \le \lad_O^{-1}  \norm{\fll y_t}_{T, \mc{H}_O}^2 \le \lad_O^{-1} (1 + e/\sqrt{2})^2 \norm{x_t}_{T,\mhh}^2\,,
\end{align*}
and 
\begin{align*}
    \mc{E}_{T,\fll}(\p_t y_t) \le \lad_O^{-1}  \norm{\p_t x_t^{(1)}}_{T, \mc{H}_O}^2\,,
\end{align*}
where, by estimates similar to \eqref{diverauxeq3},
\begin{align*}
    \norm{\p_t x_t^{(1)}}_{T, \mc{H}_O} &\le \norm{\p_t x_t^{(0)}}_{T, \mc{H}_O} + \norm{\p_t x_t}_{T, \mc{H}_O} \\
    & \le \frac{\sqrt{2} \pi e}{T} \norm{x_t}_{T,\mhh} + \frac{2}{T} \norm{x_t}_{T,\mhh}\,.
\end{align*}
Thus, we can conclude, for $x_t \in \mathbb{H}_{l,s}$,
\begin{align} \label{case3}
    c_1 = 1 + \frac{e}{\sqrt{2}}\,,\ c_2 = e\,, \ c_3 = \frac{1 + e/\sqrt{2}}{\sqrt{\lad_O}}\,,\  c_4 = \frac{\sqrt{2} \pi e + 2}{\sqrt{\lad_O} T}\,.
\end{align}

\medskip 

\noindent \underline{\emph{Case IV: $x_t \in \mathbb{H}_{h,a}$.}} It suffices to construct the solution $(z_t,y_t)$ to \eqref{diverequation} for each $x_t = H_k^a = \big(e^{- \mu_k t} - e^{- \mu_k (T-t)}\big) e_k$ with $\mu_k \ge 2/T$. We denote $u_k(t) = e^{- \mu_k t} - e^{- \mu_k (T-t)}$ and consider the following ansatz for $(z_t,y_t)$:
\begin{align} \label{diveransatz}
    z_t = v_k(t) e_k\,,\q y_t = \frac{1}{\mu_k^2} w_k(t) e_k\,,
\end{align}
with $v_k(0) = v_k(T) = w_k(0) = w_k(T) = 0$. It follows that they satisfy 
\begin{align*}
    \p_t z_t - \fll y_t = \left(v'_k(t) + w_k(t) \right) e_k = u_k(t) e_k\,,
\end{align*}
equivalently, $u_k(t) = v'_k(t) + w_k(t)$. The construction of functions $v_k$ and $w_k$ for achieving the optimal estimate is quite technical and was provided in \cite{eberle2024space}, which we sketch below for completeness. We let 
\[
\varphi_k(t) = (\mu_k t - 1)^2 \chi_{[0, \mu_k^{-1}]}(t) \in C^1([0,1])\,,
\]
satisfying $\vp_k(\frac{T}{2}) = 0$. Then, for $t \in [0, \frac{T}{2}]$, we define 
\begin{align*}
    v_k(t) = \varphi_k(t) \int_0^t u_k(s) \ud s\,,\q w_k(t) = u_k(t) - \dot{v}_k(t) = (1 - \varphi_k(t))u_k(t) - \dot{\varphi}_k(t) \int_0^t u_k(s) \ud s\,.
\end{align*}
and for $t \in [\frac{T}{2}, T]$, we set 
\begin{align*}
    v_k(t) = - \varphi_k(T - t) \int_t^T u_k(s) \ud s\,,\q w_k(t) = u_k(t) - \dot{v}_k(t)\,.
\end{align*}
Here $v_k$ and $w_k$ constructed above are continuous and piecewise $C^1$ with $ v_k(0) = v_k(T) = 0$ and 
\begin{align*}
    w_k(0) = u_k(0) - \dot{v}_k(0) = 0, \quad w_k(T) = u_k(t) - \dot{v}_k(T) = 0\,.
\end{align*}
It is also easy to see that 
\begin{align*}
    v_k\left(\frac{T}{2}\right) = \dot{v}_k\left(\frac{T}{2}\right) = 0\,,\q \dot{v}_k(0) = u_k(0), \quad \dot{v}_k(T) = u_k(T)\,.
\end{align*}
To derive the estimates \eqref{eq:diver1} and \eqref{eq:diver2}, we compute 
\begin{align*}
    \norm{\fll y_t}_{T, \mhh}^2 = \frac{1}{T} \int_0^T w_k^2(t) \ud t\,,\q \mc{E}_{T, \fll}(z_t) = \frac{\mu_k^2}{T}\int_0^T v_k^2(t) \ud t\,, 
\end{align*}
and 
\begin{align*}
    \mc{E}_{T,\fll}(y_t) = \frac{1}{\mu_k^2 T}\int_0^T w_k^2(t) \ud t\,, \q \mc{E}_{T,\fll}(\p_t y_t) = \frac{1}{\mu_k^2 T}\int_0^T \dot{w}_k^2(t) \ud t
\end{align*}
as well as $\norm{x_t}_{T,\mhh}^2 = \frac{1}{T}\int_0^T u_k^2(t) \ud t$. Estimating these one-dimensional integrals gives \eqref{eq:diver1}--\eqref{eq:diver2} with constants:
\begin{equation}\label{case4}
    c_1 = 1 + \frac{1}{\sqrt{3}}, \quad c_2 = \frac{1}{\sqrt{30}}, \quad c_3 = \frac{T}{2} \left( 1 + \frac{1}{\sqrt{3}} \right), \quad c_4 = 8\,. 
\end{equation}
For simplicity, we omit the details and direct readers to \cite{eberle2024space} for further computation.

\medskip

\noindent \underline{\emph{Case V: $x_t \in \mathbb{H}_{h,s}$.}} 
This case follows a similar line of reasoning to the previous one. We consider $x_t = H_k^s = u_k(t) e_k$ for $\mu_k \ge 2/T$, where $u_k = e^{- \mu_k t} + e^{- \mu_k (T-t)}$ and the ansatz \eqref{diveransatz} for $(z_t,y_t)$. The construction of $v_k$ and $w_k$ remains the same as in the previous case. The only difference lies in the bound related to $\norm{x_t}_{T,\mhh}^2 = \frac{1}{T}\int_0^T u_k^2(t) \ud t$ as $u_k$ is different. By the computation in \cite{eberle2024space}, we conclude with:
\begin{equation}\label{case5}
    c_1 = 1 + \frac{1}{\sqrt{3}}, \quad c_2 = \frac{1}{\sqrt{30}}, \quad c_3  = \frac{T}{2} \left( 1 + \frac{1}{\sqrt{3}} \right), \quad c_4 = 5 + \sqrt{2}\,. 
\end{equation}
Collecting \cref{case1,case2,case3,case4,case5}, the proof is complete. 
    
\end{proof}

\end{appendix}


\end{document}